\documentclass[12pt, a4paper]{amsart}
\usepackage{amssymb,fullpage}
\usepackage{amsmath,esint,  bm}
\usepackage[colorlinks, allcolors=RedViolet]{hyperref}
\usepackage{colortbl}
\usepackage[dvipsnames]{xcolor}

\newtheorem{theorem}{Theorem}[section]
\newtheorem{lemma}[theorem]{Lemma}

\newtheorem{corollary}[theorem]{Corollary}
\theoremstyle{definition}

\newtheorem{assump}{Assumption}

\usepackage{enumitem}
\theoremstyle{remark}
\newtheorem{remark}[theorem]{Remark}
\numberwithin{equation}{section}

\newcommand{ \mr }{ \mathbb{R} }
\newcommand{ \R }{ \mathbb{R} }

\newcommand{ \bA }{ \mathbf{A} }
\newcommand{ \A }{ \mathcal{A} }
\newcommand{ \Q }{ \mathcal{Q} } 
\newcommand{ \I }{ \mathcal{I} }
\newcommand{ \bV }{ \mathbf{V} }
\newcommand{ \ba }{ \mathbf{a} }
\newcommand{ \bb }{ \mathbf{b} }
\newcommand{ \bP }{ \mathbf{P} }
\newcommand{ \bQ }{ \mathbf{Q} }

\newcommand{ \bfw }{ \mathbf{w} }
\newcommand{ \bfh }{ \mathbf{h} }
\newcommand{ \bfu }{ \mathbf{u} }
\newcommand{ \zero }{ \mathbf{0} }

\newcommand{ \bfv }{ \mathbf{v} }
\newcommand{\tphi}{{\tilde\phi}}
\renewcommand{\d}{\mathrm{d}}
\newcommand{\dz}{\mathrm{d} z}

\newcommand{\loc}{{\operatorname{loc}}}
\renewcommand{\epsilon}{\varepsilon}
\renewcommand{\phi}{\varphi}
\renewcommand{\le}{\leqslant}
\renewcommand{\ge}{\geqslant}
\renewcommand{\leq}{\leqslant}
\renewcommand{\geq}{\geqslant}
\renewcommand{\div}{\operatorname{div}}

\providecommand{\Ol}{\ensuremath{\mathcal{O}_\lambda}}

\newcommand{\ainc}[1]{\hyperref[ainc]{{\normalfont(aInc){\ensuremath{_{#1}}}}}}
\newcommand{\adec}[1]{\hyperref[adec]{{\normalfont(aDec){\ensuremath{_{#1}}}}}}
\newcommand{\inc}[1]{\hyperref[inc]{{\normalfont(Inc){\ensuremath{_{#1}}}}}}
\newcommand{\dec}[1]{\hyperref[dec]{{\normalfont(Dec){\ensuremath{_{#1}}}}}}

\newcommand{\wMAe}[1]{\hyperref[wMAe]{{\normalfont(wMA){\ensuremath{_{#1}}}}}}

\begin{document}

\title{
Partial regularity  for degenerate parabolic systems with general growth  via caloric approximations
}

\author{Jihoon Ok}
% % Address of record for the research reported here
\address[Jihoon Ok]{Department of Mathematics, Sogang University, 35 Baekbeom-ro, Mapo-gu, Seoul 04107, Republic of Korea.}
\email{jihoonok@sogang.ac.kr}

\author{Giovanni Scilla}
\address[Giovanni Scilla]{Department of Mathematics and Applications ``R. Caccioppoli'', University of Naples Federico II, Via Cintia, Monte S. Angelo, 80126 Naples, Italy}
\email{giovanni.scilla@unina.it}

\author{Bianca Stroffolini}
\address[Bianca Stroffolini]{Department of Mathematics and Applications ``R. Caccioppoli'', University of Naples Federico II, Via Cintia, Monte S. Angelo, 80126 Naples, Italy}
\email{bstroffo@unina.it}
%
%\author{Jihoon Ok}
% % Address of record for the research reported here
%\address{Department of Applied Mathematics and the Institute of Natural Sciences, Kyung Hee University, Yongin 17104, Republic of Korea}
%\email{jihoonok@khu.ac.kr}

% \thanks will become a 1st page footnote.

% General info
%\subjclass[2010]{d}

%\date{\today}

%\dedicatory{This paper is dedicated to our advisors.}


\begin{abstract}

We establish a partial regularity result for solutions of parabolic systems with general $\phi$-growth, where $\phi$ is an Orlicz function. % hence, the systems can be degenerate or singular, or may not be able to determine  degeneracy. 
In this setting we can develop a unified approach that is independent of the degeneracy of system and relies on two caloric approximation results: the $\phi$-caloric approximation, which was introduced in \cite{DieSchStrVer17}, and an improved version of the $\mathcal A$-caloric approximation, which we prove without using the classical compactness method. %for which in this paper we obtain an improved version without using the classical compactness method.}
\end{abstract}

\maketitle

%%%%%%%%%%%%%%%%%%%%%%%%%%%%%%%%%%%%%%%%%%%%%%%%%%%%%%%%%%%%%%%%%%%%%%%%
%%%%%%%%%%%%%%%%%%%%%%%%%%%%%%%%%%%%%%%%%%%%%%%%%%%%%%%%%%%%%%%%%%%%%%%%
%%%%%%%%%%%%%%%%%%%%%%%%%%%%%%%%%%%%%%%%%%%%%%%%%%%%%%%%%%%%%%%%%%%%%%%%

%\tableofcontents

\section{Introduction}
The aim of this paper is to prove partial regularity for solutions of
the following autonomous parabolic system:
\begin{equation}\label{system1}
\partial_t \bfu - \mathrm{div}\, \bA(D\bfu) =\zero \quad \text{in }\ \Omega\times (0,T],
\end{equation}
where $\bA\in C(\R^{Nn}, \R^{Nn}) \cap C^1(\R^{Nn}\setminus \{\zero \}, \R^{Nn})$ is modeled as the $\phi$-Laplacian, see Assumption~\ref{Ass}.
We would like to point out that, already in the stationary case, the best result we can expect for non-radial systems is the $C^{1,\alpha}$-regularity outside a set of Lebesgue measure zero, see the survey \cite{Min_darkside} and references therein.

In this direction, a powerful tool is the comparison and closeness with suitable smooth maps, for which excess decay estimates are available. 
The first use of a compactness argument for approximately harmonic maps goes back to De Giorgi, in the context of regularity of minimal surfaces in geometric measure theory, see \cite{degiorgi,simon}. De Giorgi's Lemma states that there is a rigidity behaviour of approximately harmonic maps, in the sense that they are close to harmonic ones. This lemma has been generalized to strongly elliptic bilinear forms, the so-called $\A$-harmonic approximation, in \cite{DuzSte02} for
applications to the boundary regularity of minimizing elliptic currents. For elliptic systems and related quasiconvex functionals of $p$-growth, we refer to for instance \cite{DuzSte02,DuzGasGro00,DuzGroKro05,Beck07,FosMin08,Beck09,BogDuzHabSch11}. Here partial regularity results are proved, using a two-scale approach. As long as the excess functional is small compared to the gradient average in a ball, one can linearize the system via the $\A$-harmonic approximation Lemma. When, instead, the system is degenerate, one compares with the $p$-Laplacian via the $p$-harmonic approximation, see \cite{DuzMin041} and also \cite{Bog12,DuzMin04}.
The final partial regularity result is then achieved with an exit time argument.
In a more general setting, the $\A$-harmonic approximation in Orlicz spaces and the $\phi$-harmonic approximation were proven in \cite{DLSV12} and \cite{DieStrVer12}, respectively, by utilizing a refined Lipschitz truncation argument. Using these approximation results, in recent years, partial regularity for elliptic systems or quasiconvex functionals with general growth has been studied in \cite{CelOk20,Str20,IseLeoVer21,GooSciStr22,OkSciStr22}.

Regularity results for the evolutive $p$-Laplacian system
$$
\partial_t \bfu - \div \left(|D\bfu|^{p-2}D\bfu\right)=\zero \quad \text{in }\ \Omega_T
$$
were established by DiBenedetto and Friedman in \cite{DiBeFried84,DiBeFried85}. Their key idea was to look at the system in a new geometry, that, in a sense, reduces the system to the classical heat system.
Roughly speaking, if the average of the gradient of a solution is locally comparable with $\lambda$, the system looks like
\begin{equation*}
\partial_t \bfu =\div(\lambda^{p-2} D\bfu) \,.
\end{equation*}
This suggests to consider a \lq \lq new metric" in which the scaling is homogeneous and to consider \lq\lq balls" centered at the point $z_0=(x_0, t_0)$ with respect to this metric; i.e., the cylinders:
\begin{equation*}
Q^{\lambda}_{r}(z_0) := B_{\rho}(x_0)\times (t_0-\lambda^{2-p} r^2, t_0+\lambda^{2-p} r^2)\,.
\end{equation*}
This method is known as the \textit{intrinsic scaling} method.
As for the evolutive Uhlenbeck system with general $\phi$-growth
\begin{equation}\label{system_Orlicz}
\partial_t \bfu - \div \left(\frac{\phi'(|D\bfu|)}{|D\bfu|}D\bfu\right)=\zero \quad \text{in }\ \Omega_T \,,
\end{equation}
everywhere $C^{1,\alpha}$-regularity is finally established in our recent paper \cite{OSS23} using cylinders that are intrinsic with respect to the function $\phi$, namely,
\begin{equation*}
Q^{\lambda}_{r}(z_0) : =B_{r}(x_0)\times \left(t_0-\frac{ \lambda^2}{\phi(\lambda)}r^2, t_0+\frac{\lambda^2}{\phi(\lambda)} r^2\right)\,.
\end{equation*}
We also refer to \cite{Lie06,DieSchSch19} for related regularity results for the system \eqref{system_Orlicz}.

Partial regularity for parabolic systems using caloric approximations has been studied in \cite{DuzMin05,BoDuzMin10_1,BoDuzMin10_2,DuzMinSte11,BoDuzMin13,Scheven11}. In particular, in \cite{BoDuzMin13} B\"ogelein, Duzaar and Mingione obtained the $\A$-caloric and $p$-caloric approximations, and, using these approximations, proved partial H\"older continuity of the gradient of weak solutions to the degenerate parabolic system \eqref{system1} with standard $p$-growth.
Let us review the proof of partial regularity in \cite{BoDuzMin13}. By assuming a smallness condition on the relevant excess at some scale, it is possible to linearize the system at the gradient average in the nondegenerate case, and compare the solution of the original system with one of the linearized system. The comparison argument ensures the decay estimate for the excess at smaller scales.
In the degenerate regime, one compares with a suitable $p$-Laplace evolutive system via the $p$-caloric approximation. At this stage, one can proceed using the intrinsic cylinders \`a la Di Benedetto.
Finally, the degenerate and nondegenerate regimes are matching together keeping track of the  so called \lq\lq switching radius".\par

In this paper we consider degenerate parabolic systems with general $\phi$-growth and obtain partial H\"older continuity of the gradient of weak solutions. Our result covers a large class of systems whose degeneracy need not to be determined. In particular, we extend  the results of the subquadratic and superquadratic systems obtained in  \cite{BoDuzMin13}.

We emphasize that our method, inspired by \cite{BoDuzMin13}, deals with  systems in a unified way, without any distinction between the superquadratic and subquadratic cases. 
The main tools are some caloric approximations in the Orlicz setting. In the nondegenerate regime, we prove a new version of the $\mathcal{A}$-caloric approximation which is more suitable to our setting. We would like to mention a recent related result in \cite{FILV23}, where the authors obtained an $\A$-caloric approximation using the classical compactness method and, applying this, they proved partial H\"older regularity for nondegenerate parabolic systems with general growth. The proof of our version, stated in Theorem~\ref{thm:Acaloric}, relies on a parabolic duality argument and an improved parabolic Lipschitz truncation, along the lines of \cite{DLSV12} for the elliptic case. Consequently, we can obtain closeness with comparing mappings in terms of gradients directly, which is sharper than the one considered in \cite{FILV23}. Moreover, we underline explicitly how the constant $\delta$ in Theorem~\ref{thm:Acaloric} depends only on $p$ and $q$ instead of $\phi$.
As for the degenerate regime, we use the $\phi$-caloric approximation lemma proven in \cite{DieSchStrVer17}.
Now, we state the main theorem of our paper.
\begin{theorem}\label{mainthm}
Let $\bfu$ be a weak solution to \eqref{system1} where $\bA$ satisfies Assumption~\ref{Ass}. There exist $U\subset \Omega_T$ with $|\Omega_T\setminus U| =0$ and $\alpha\in(0,1)$ such that $D\bfu\in C^{\alpha,\frac{\alpha}{2}}_{\mathrm{loc}}(U)$. Moreover,
we have $(\Omega_T\setminus U) \subset (\Sigma_1\cup \Sigma_2)$, where
\begin{equation}\label{Sigma1}
\Sigma_1 : = \left\{ z_0\in\Omega_T: \liminf_{r\to 0^+} \fint_{Q_{r}(z_0)} |\bV(D\bfu) - (\bV(D\bfu))_{Q_r(z_0)} |^2\,\dz >0\right\}
\end{equation}
and
\begin{equation}\label{Sigma2}
\Sigma_2 : = \left\{ z_0\in\Omega_T: \limsup_{r\to 0^+} |(D\bfu)_{Q_r(z_0)} | =\infty \right\} \,.
\end{equation}
\end{theorem}

\textit{Overview of the paper.}
 In Section~\ref{Sec2}, we fix the basic notation and collect some definitions and results about Orlicz functions. In Section~\ref{Sec:caloricapproximation}, we prove the $\mathcal{A}$-caloric approximation and recall the $\phi$-caloric one. In Section~\ref{sec3}, we obtain the Caccioppoli inequality and the higher integrability results. In Sections~\ref{Sec:nondegenerate} and ~\ref{Sec:degenerate}, we consider the nondegenerate and degenerate regimes, respectively. 
 In Section~\ref{Sec:iteration}, we perform the iteration procedure and  then prove the main theorem, Theorem~\ref{mainthm}.

\section{Preliminaries}\label{Sec2}

\subsection{Notation}
For $z_0=(x_0,t_0)\in \R^n \times \R$ and $r,\tau>0$, we define
$$
Q_{r,\tau}(z_0):=B_{r}(x_0)\times (t_0-\tau,t_0+\tau)\,,
$$ 
where $B_r(x_0)$ is the open ball in $\R^n$ centered at $x_0$ with radius $r$. In particular, we write $Q_r(z_0):=Q_{r,r^2}(z)$ which is the usual parabolic cylinder. Moreover, for a given function $\phi:(0,\infty)\to (0,\infty)$ and $\lambda>0$, we write $Q^\lambda_r(z_0): =Q_{r,\tau}(z_0)$ with $\tau=\frac{\lambda^2}{\phi(\lambda)}r^2$. If the center $z_0$ is the origin or is not important, we omit writing the center of the cylinders.
The notation $f\lesssim g$ or $f\sim g$ means that there exists constant $c\ge 1$ such that $f \le c g$ or $\frac{1}{c}f\le g \le c f$. We write the average of a function $f$ on $Q_r(z_0)$ and on $Q^\lambda_r(z_0)$ as
$$
(f)_{r} = (f)_{Q_r(z_0)} : = \fint_{Q_r(z_0)} f\, \dz 
\quad\text{and}\quad
(f)^{\lambda}_{r} = (f)_{Q^\lambda_r(z_0)}:= \fint_{Q_r^\lambda(z_0)} f\, \dz\,, 
$$
respectively. 
%We will use the Einstein summation convention, that is, we will omit the summation symbol for indexes that appear twice, see e.g. \eqref{ellipticity} and the next inequality.

%
%
%Let $z_0=(x_0,t_0)\in \R^n\times \R$. Denote
%$$
%Q_r(z_0)=B_r(x_0)\times (t_0-r^2,t_0+r^2),
%$$
%$$Q^\lambda_r(z_0)=B_r(x_0)\times I^\lambda_r,(t_0) \quad \text{where }\ I^\lambda_r(t_0):=\left(t_0-\frac{r^2}{\phi''(\lambda)},t_0+\frac{r^2}{\phi''(\lambda)}\right).
%$$

%We write  ${\bf u}=(u^\alpha)=(u^1,\dots,u^N)\in \R^N$ and  
%${\bf Q}=(Q^\alpha_i)\in\R^{N\times n}=\R^{Nn}$ 
%%${\bf Q}=(Q^\alpha_i)\in\R^{Nn}=\R^{Nn}$ 
%where 1$\le i\le n$ and  $1\le \alpha \le N$.
%For $z=(x,t)\in\R^n\times\R$, we introduce the parabolic cylinder
%\begin{equation}
%%Q_r(z):=B_r(x)\times (t-(2r)^2,t]\,,
%Q_r(z):=B_r(x)\times (t-r^2,t]\,, 
%\label{eq:cylinder}
%\end{equation}
%where $B_r(x)$ denotes the open ball in $\R^n$ with center $x$ and radius $r$. The symbol $\partial_{\rm p} Q_r(z)$ denotes the usual parabolic boundary of $Q_r(z)$. 
%

\subsection{Orlicz functions and related operators}

In this paper, $\phi:[0,\infty)\to[0,\infty)$ is always an $N$-function; that is, $\phi(0)=0$, there exists a right continuous derivative $\phi'$ of  $\phi$,  $\phi'$ is 
{increasing}
%nondecreasing
with $\phi'(0)=0$ and $\phi'(t)>0$ when $t>0$. 
%Without loss of generality, we can assume that  
For simplicity, we shall assume that
$$
\phi(1)=1\,.$$
{Note that if we do not assume this condition, constants $c$ in this paper may depend on $\phi(1)$.}
We say that $\phi$ satisfies the $\Delta_2$ condition denoted by $\Delta(\phi)<\infty$ if there exists a positive constant $K =: \Delta(\phi)$ such that $\phi(2t) \le K\phi(t)$ for all $t>0$.   
%Moreover, we assume that $\phi$ satisfies the following growth conditions: 
%The almost decreasing and increasing conditions in Assumption~\ref{Ass1} are equivalent to the $\Delta_2$ and $\nabla_2$ conditions for $\phi$, respectively. Compared with the $\Delta_2$ type conditions, the benefit of the almost increasing/decreasing condition is that we can directly see the lower and upper bounds of an exponent factor of $\phi$.   We also remark that Assumption~\ref{Ass1} with $L=1$ is equivalent to the following inequality 
%\begin{equation}\label{characteristic}
%1< p \le  \frac{t \phi'(t)}{\phi(t)} \le q\,,\quad t>0\,.
%\end{equation}
%{Without loss of generality, we sometimes assume that 
%\begin{equation}
% 1< p < 2 <q\, .
% \label{p2q}
%\end{equation}}
%
The conjugate function of $\phi$ is defined as
\begin{equation}\label{conjugate}
\phi^*(t):= \sup_{s \geq 0}\, (st - \phi(s))\,.
\end{equation}
From the definition, the following Young's inequality  
\begin{equation}  \label{eq:young}
  st \leq \phi(t) +  \phi^\ast(s)\,,\quad s,t\geq0\,,
\end{equation}
holds true. From now on we always assume that $\phi$ and $\phi^*$ satisfy the $\Delta_2$ condition and  this is 
%denoted 
indicated by $\Delta(\phi,\phi^*)<\infty$, where $\Delta(\phi,\phi^*)$ denotes the relevant constants $K$.  We note that the exact value of $\phi^*$ is not always explicitly computable and instead the estimate
\begin{equation}
\phi^*\left(\frac{\phi(t)}{t}\right)\sim \phi^*(\phi'(t)) \sim \phi(t)
\label{eq:hok2.4}
\end{equation}
will often be useful in computations (see \cite[Theorem~2.4.10]{HH}).

If $\phi$ satisfies $\Delta(\phi, \phi^*)$, we define the Orlicz space $L^\phi(\Omega,\R^N)$ as the set of all measurable functions $f:\Omega\to\R^N$ such that
$$
\int_\Omega \phi(|f(x)|)\, \d x < \infty, 
$$ 
and 
the Orlicz-Sobolev space $W^{1,\phi}(\Omega,\R^N)$ as the set of all  $f\in L^\phi(\Omega,\R^N)\cap W^{1,1}(\Omega,\R^N)$ such that
$$
\int_\Omega \phi(|Df(x)|)\, \d x < \infty.
$$ 
 $L^\phi(\Omega,\R^N)$ and $W^{1,\phi}(\Omega,\R^N)$ are endowed with the usual Luxembourg type norms. Then they are reflexive Banach spaces. Moreover, for an interval $I$ in $\R$, the parabolic spaces
$L^\phi(I;W^{1,\phi}(\Omega,\R^N))$ or $L^\phi(I;W^{1,\phi}_0(\Omega,\R^N))$ denote the set of all functions $f:\Omega\times I\to \R^N$ such that $f(\cdot,t)\in W^{1,\phi}(\Omega,\R^N)$ or $f(\cdot,t)\in W^{1,\phi}_0(\Omega,\R^N)$  for a.e. $t\in I$ and 
$$
\int_{I} \int_\Omega \phi(|Df(x,t)|)\, \d x\, \d t < \infty.
$$ 
% 
%We define parabolic Sobolev spaces.
%$$
%L^\phi(I;W^{1,\phi}(\Omega,\R^{N})) = \{u\in L^{1}(I;W^{1,1}(\Omega,\R^N)): |Du| \in L^\phi(\Omega)\}.
%$$
%$$
%L^\phi(I;W^{1,\phi}_0(\Omega,\R^{N})) = \{u\in L^{1}(I;W^{1,1}_0(\Omega,\R^N)): |Du| \in L^\phi(\Omega)\}.
%$$
%Moervoer, $(L^{\phi}(I; W^{1,\phi}_0(\Omega,\R^N)))'$ is the dual space of $L^\phi(I;W^{1,\phi}_0(\Omega,\R^{N}))$.
We recall the Jensen's type inequality in \cite[Lemma 2.2]{HasOk21}.
\begin{lemma}\label{lem:Jensen}
If $\psi:[0,\infty)\to [0,\infty]$ is increasing with $\psi(0)=0$ and satisfies 
that $\psi(t)/t\le L\psi(s)/s$ for every $0\le t \le s$ with constant $L\geq 1$, then 
\[
\psi\bigg( \frac{1}{L^2} \fint_U |f|\, {\mathrm{d}}z\bigg) \le \fint_U \psi(|f|)\, {\mathrm{d}}z.
\] 
\end{lemma}

Now we further assume for $\phi$ that
\begin{equation}\label{pq}
\frac{2n}{n+2} < p \le \frac{t\phi''(t)}{\phi'(t)} +1 \le q, \quad \text{for all }t\in(0,\infty).
\end{equation}
Without loss of generality, we always assume that $p<2<q$.
Note that this implies
\begin{equation}\label{characteristic}
1< p \le  \frac{t \phi'(t)}{\phi(t)} \le q\,,\quad t>0\,.
\end{equation}
and hence the $\Delta_2$ conditions of $\phi$ and $\phi^*$.
We define vector valued functions $\bV:\R^{Nn}\to \R^{Nn}$
%$\bA, \bV:\R^{Nn}\to \R^{Nn}$ 
by
\begin{equation}\label{Vfunction}
\bV({\bf Q}):= \sqrt{\frac{\phi'(|{\bf Q}|)}{|{\bf Q}|}} {\bf Q}.
\end{equation}

We recall equivalent relations in \cite[Lemmas 3 and  20]{DieEtt08} and \cite[Lemma~3.1]{DieSchSch19}:
\begin{equation}\label{monotonicity1}
\frac{\phi'(|{\bf P}| + |{\bf Q}|)}{|{\bf P}| + |{\bf Q}|}  |{\bf P}-{\bf Q}|^2 \sim |\bV( {\bf P})-\bV( {\bf Q})|^2   \sim \varphi_{|{\bf Q}|}(| {\bf P} -  {\bf Q}|) \,,
\end{equation}
\begin{equation}\label{DE08lem20}
\frac{\phi'(|{\bf P}| + |{\bf Q}|)}{|{\bf P}| + |{\bf Q}|} \sim \int_0^1 \frac{\phi'(|\tau {\bf P} + (1-\tau) {\bf Q}|)}{|\tau {\bf P} + (1-\tau) {\bf Q}|} \, \d \tau \,,
\end{equation}
and
\begin{equation}
|{\bf A}( {\bf P})-{\bf A}({\bf Q})|  \sim \varphi'_{|{\bf Q}|}(|{\bf P} - {\bf Q}|)\,.
\label{eq:(3.4)}
\end{equation}

Moreover, by the same proof of \cite[Lemma A.2]{DieKaSch12}, we have that for every ${\bf g}\in L^\phi(Q_r; \R^{Nn})$,
\begin{equation}\label{Vintequivalent}
\fint_{Q_r} |\bV({\bf g}) - (\bV({\bf g}))_{Q_r}|^2\,\d z \sim   \fint_{Q_r} |\bV({\bf g}) - \bV(({\bf g})_{Q_r})|^2\,\d z \,.
\end{equation}
Finally we recall the following Young type inequality from \cite[Proposition 3.8 (3)]{HasOk22}: for every $\epsilon \in (0,1)$,
\begin{equation}\label{phipqinequality}
\phi(|\bP - \bQ|) \le \epsilon \left(\phi(|\bP|) + \phi(|\bQ|)\right) + c \epsilon^{-1} |\bV(\bP)- \bV(\bQ)|^2 \,.
\end{equation}
 {Note that all constants concerned with the relation $\sim$ and $c$  in above depend only on $p$ and $q$.}

\subsection{Shifted $N$-functions} %and related operators}

The following definitions and results about shifted $N$-functions can be found in \cite{DieEtt08, DieStrVer09}.

For an $N$-function  $\phi$ and for $a\ge0 $,  we define the shifted $N$-function $\phi_a$ by
$$
\phi_{a}(t):=\int_{0}^{t} \frac{\phi'(a+s) s}{a+s} \, \d s  \quad \left(\text{i.e., }\ \phi_a'(t) =\frac{\phi'(a+t)}{a+t}t\right).
$$
We note that if $\phi$ satisfies 
\eqref{pq}
%Assumption~\ref{Ass1} or \ref{offdiagonal},
 then $\phi_a$ also satisfies 
 \eqref{pq}
 %Assumption~\ref{Ass1} or \ref{offdiagonal}  
 uniformly in $a\ge 0$ with the same $p$ and $q$.
%with $1\le p \le 2 \le q$.  
%We then recall useful inequalities for shifted $N$-function $\phi_a$ in \cite{DieEtt08}.

%\comment{\color{blue}Jihoon: I couldn't find exact references of the following inequalities.}
%\begin{lemma} {\color{red} references?}
%Let $\varphi$
%satisfy Assumption~\ref{Ass1}. 
%Let $\varphi$ be an $N$-function such that $\Delta_2(\varphi,\varphi^*)<+\infty$. 
%Then for every $a,t>0$ and ${\bf a}, {\bf b}\in \R^m$ we have
Under assumption \eqref{pq} on $\phi$, we have the following relations (see, e.g., \cite[Proposition 2.3]{CelOk20} and \cite{DieStrVer09}), which hold uniformly with respect to $a\geq0$:
\begin{align}
&\phi_a(t) \sim \phi'_a(t)\,t\,; \label{(2.6a)} \\
&\phi_a(t) \sim \phi''(a+t)t^2\sim\frac{\varphi(a+t)}{(a+t)^2}t^2\sim \frac{\varphi'(a+t)}{a+t}t^2\,,\label{(2.6b)}\\
& \phi(a+t)\sim [\phi_a(t)+\phi(a)]\,.\label{eq:approx}
\end{align}
%\begin{equation}
%\phi(a+t) \sim \phi_a(t) + \phi(a)\,, \label{eq:approx}
%\end{equation}
%and 
%\begin{equation}
%\varphi_{|{\bf a}|}(t) \leq c_\eta\varphi_{|{\bf b}|}(t) + \eta \varphi_{|{\bf a}|}(|{\bf a}-{\bf b}|)\,.
%\label{(5.4diekreu)}
%\end{equation}
%and for all $\delta>0$ there exists $c_\delta>0$ depending only on 
%$p$, $q$, $L$ and $\delta$,  
%%and $\Delta_2(\varphi,\varphi^*)$ 
%such that for all $t,u,a\geq0$
%\begin{equation}\label{eq:young4}
%\left\{\begin{aligned}
%& tu\leq \delta \varphi(t) + c_\delta \varphi^*(u)\,,\\
%& t\varphi'(u) + u\varphi'(t) \leq \delta \varphi(t) + c_\delta\varphi(u)\,, \\
%& tu\leq \delta \varphi_a(t) + c_\delta \varphi_a^*(u)\,,\\
%& t\varphi_a'(u) + u\varphi_a'(t) \leq \delta \varphi_a(t) + c_\delta\varphi_a(u)\,. 
%\end{aligned}\right.
%\end{equation}
%\end{lemma}
The following lemma (see \cite[Corollary~26]{DieKre08}) deals with the \emph{change of shift} for $N$-functions.
\begin{lemma}[change of shift]\label{lem:changeshift}
Let $\varphi$ be an $N$-function with $\Delta_2(\varphi),\Delta_2(\varphi^*)<\infty$. Then for any $\eta>0$ there exists $c_\eta>0$, depending only on $\eta$ and $\Delta_2(\varphi)$, such that for all ${\bf a}, {\bf b}\in\R^m$ and $t\geq0$
\begin{equation}
\varphi_{|{\bf a}|}(t) \leq c_\eta \varphi_{|{\bf b}|}(t) + \eta \varphi_{|{\bf a}|}(|{\bf a}-{\bf b}|)\,.
\label{(5.4diekreu)}
\end{equation}
\end{lemma}

\subsection{Assumption and weak solution}

We state the assumption of the main theorem, Theorem~\ref{mainthm}, and the definition of weak solution to \eqref{system1}.
\begin{assump} \label{Ass}
The operator $\bA$ verifies the following assumptions with constants $0<\nu\le 1\le  L$ and an $N$-function $\phi\in C^{1}([0,\infty))\cap C^2((0,\infty))$ satisfying \eqref{pq}.
\begin{itemize}
\item[(A1)] ($\varphi$-growth condition) 
\begin{equation}\label{growth}
|\bA(\bP)|+ |D\bA(\bP)| |\bP| \le L \phi'(|\bP|)\,,
\end{equation}
\begin{equation*}
%\label{ellipticity}
\big[D\bA(\bP) (\ba \otimes  \bb )\big] :  (\ba \otimes  \bb)  \ge \nu \phi''(|\bP|)|\ba||\bb|\,,
\end{equation*}
for all $\bP\in \R^{Nn}\setminus\{{\bf 0 }\}$, $\ba\in \R^n$ and $\bb\in \R^N$.

\item[(A2)] (Off diagonal condition on $\bA$ and $\phi$)
\begin{equation}\label{offdiagonal}
|D\bA(\bP) - D\bA(\bQ)|+|D^2 (\phi(|\bP|)) - D^2(\phi(|\bQ|))|\le L \left(\frac{|\bP-\bQ|}{|\bP|}\right)^{\gamma}\phi''(|\bP|)
\end{equation}
for some $\gamma\in(0,1)$,  all $\bP,\bQ\in \R^{Nn}$ with $|\bP-\bQ|\le \frac{1}{2}|\bP|$, $\ba\in \R^n$ and $\bb\in \R^N$.
%, where $\bA_\phi(\bP):= \frac{\phi'(|\bP|)}{|\bP|}\bP$.

\item[(A3)] (Almost $\phi$-isotropic condition near the origin) 
For every $\epsilon>0$ there exists $\delta=\delta(\epsilon)>0$ such that
\begin{equation}\label{almostphi}
\left|\bA(\bP)-\frac{\phi'(|\bP|)}{|\bP|}\bP\right| \le \epsilon \phi'(|\bP|)
\end{equation}
for all $\bP\in \R^{Nn}$ with $|\bP|\le \delta$.

\end{itemize}
\end{assump}

We note that the assumption (A1) implies the following monotonicity
\begin{equation}\label{monotonicity}
(\bA(\bP)-\bA(\bQ)) : (\bP-\bQ) \ge \tilde\nu\phi''(|\bP|+|\bQ|) |\bP-\bQ|^2, \quad \bP,\bQ\in \R^{Nn}.
\end{equation}
for some $\tilde\nu=\tilde\nu(\nu,L)>0$.

A function $\bfu=(u^1,u^2,\dots,u^N) \in C_{\loc}(0,T; L^2_{\loc}(\Omega,\R^N)) \cap L^\phi_\loc(0,T;W^{1,\phi}_{\loc}(\Omega,\R^N))$ is said to be a (local) \emph{weak solution} to \eqref{system1} if it satisfies the following weak form of \eqref{system1}: 
$$%\begin{equation}
-\int_{\Omega_T} \bfu \cdot {\bm\zeta}_t \,\d z + \int_{\Omega_T}   \bA(D\bfu) : D{\bm\zeta} \, \d z =0 
\quad  \text{for all }\ {\bm\zeta}\in C^\infty_{\mathrm c}(\Omega_T,\R^N)\,,
$$%\label{eq:weakformul}
%\end{equation}
where  ``$\cdot$" and ``$:$"  are the Euclidean inner products in $\R^N$ and $\R^{Nn}$, respectively. 
 By the density of smooth functions in Orlicz-Sobolev spaces and a standard approximation argument one can see that the weak solution $\bfu$ to \eqref{system1} also satisfies for every $0<t_1<t_2\le T$,
 \begin{equation*}
\left.\int_{\Omega'} \bfu \cdot {\bm\zeta}(x,t)\,\d x\right|_{t=t_1}^{t=t_2} + \int_{\Omega'}\int_{t_1}^{t_2} \left[-\bfu \cdot {\bm\zeta}_t  +   \bA(D\bfu) : D{\bm\zeta}\right] \, \d t\, \d x =0 
%\label{eq:weakformul1}
\end{equation*}
for all ${\bm\zeta}\in W^{1,2}([t_1,t_2];L^2(\Omega',\R^N))\cap L^\phi([t_1,t_2];W^{1,\phi}_0(\Omega',\R^N))$ and $\Omega'\Subset\Omega$.

\begin{remark}
The weak solution $\bfu$ to \eqref{system1} is not weakly differentiable in $t$. Consequently, we are unable to employ a test function $\bm\zeta$ that directly involves the weak solution. However, this problem can be overcome by utilizing an approximation method known as the Steklov average, as described in \cite[I. 3-(i) and II. Proposition 3.1]{DiB_book}. This technique has become a standard approach for addressing such problems. Henceforth, we shall assume that $\bfu$ is differentiable and proceed to consider test functions that involve the weak solution without further explicit clarification.
\end{remark}

%We say that $\bfu \in L^\phi(I; W^{1,\phi}(\Omega))$ is a weak solution to \eqref{system1}
%if  $\bfu_t \in (L^{\phi}(I; W^{1,\phi}_0(\Omega)))'$ and 
%$$
%-\langle \partial_t \bfu, \zeta \rangle +\int_{\Omega\times I} \bA(D\bfu) : D\zeta \,dz =0 
%$$
%for every $\zeta\in L^{\phi}(I; W^{1,\phi}_0(\Omega,\R^N))$. Note that if $\zeta$ is differentiable in $t\in [t_1,t_2]\subset (0,T]$ then we have that
%$$
%\langle \partial_t\bfu, \zeta \rangle = - \int_{Q_r} \bfu \cdot \zeta_t\, \d z + \int_{B_r} \bfu(x,t_2) \cdot \zeta (x, t_2) \, \d x   -\int_{B_r} \bfu(x,t_1) \cdot \zeta (x, t_1) \, \d x\,.
%$$
%

%
%\textcolor{blue}{
%\begin{lemma}\label{Lem:iteration}
%Let $Z$ be a bounded non-negative function in the  interval $[r,R] \subset \R$ 
%and let $X$ be a doubling function in $[0, \infty)$.  
%Assume that there exists $\theta\in [0,1)$ such that 
%\[
%Z(t) \leq X(\tfrac1{s-t}) + \theta Z(s)
%\]
%for all $r \leq t < s \leq R$. Then
%\[
%Z(r) \lesssim X(\tfrac1{R-r}),
%\]
%where the implicit constant depends only on the doubling constant and $\theta$.
%\end{lemma}
%}

%\newpage

\section{$\A$-caloric and $\phi$-caloric approximations}\label{Sec:caloricapproximation}

In this section, we introduce two caloric approximations. They play  a crucial role in the proof of partial regularity. $\phi$-caloric approximation was obtained in \cite{DieSchStrVer17} and we just recall it. On the other hand, we derive a new version of $\A$-caloric approximation with gradient estimates by using parabolic duality and Lipschitz truncation. In Sections~\ref{subsec:CZ} and \ref{subsec:Lip}, we obtain auxiliary lemmas for $\A$-caloric approximation.

\subsection{Regularity estimates for linear systems with constant coefficients}\label{subsec:CZ}
 We introduce Lipschitz estimates for $\A$-caloric maps and  Calder\'on-Zygmund estimates for parabolic linear systems with constant coefficient $\A$. Let $ \A  = (\A^{\alpha\beta}_{ij}) \in \R^{N^2n^2}$ satisfy the \emph{Legendre-Hadamard condition}: for every $\mathbf{a}= (a^\alpha)\in \R^N$ and  $\mathbf{b}=(b_i)\in \R^n$,
\begin{equation*}\label{constantA}
\A ({\bf a} \otimes {\bf b})  \, : \, ({\bf a}\otimes {\bf b})   =\A^{\alpha\beta}_{ij} a^\alpha a^\beta b_i  b_j \ge \nu |\mathbf{a}|^2 |\mathbf{b}|^2 
\end{equation*}
for some $\nu>0$.  Then a weak solution $\bfh : Q_r \to \R^N$ to the linear system with coefficient $\A$ 
$$
\bfh_t -\div (\A D\bfh) = {\bf 0} \quad \text{in }\ Q_r
$$
is called an \textit{$\A$-caloric map}. By standard regularity theory, see for instance \cite{Cam66}, $\bfh\in C^\infty(Q_r,\R^N)$, and in particular we have the following Lipschitz estimate and excess decay estimate, which will be used in Section~\ref{Sec:nondegenerate}.
\begin{lemma}\label{Lem:Acaloricmap}
Suppose $\bfh\in C^\infty(Q_r,\R^N)$ is an $\A$-caloric map in $Q_r$. Then we have that 
\begin{equation}\label{DhLip}
\sup_{Q_{r/2}} |D\bfh|  \le c  \fint_{Q_{r}} |D\bfh|  \, \d z\,.
\end{equation}
Moreover, for every $\theta\in (0,1)$,
\begin{equation}\label{Dhdecay}
\fint_{Q_{\theta r}} |D\bfh- (D\bfh)_{\theta r}| \, \d z \le c \theta \fint_{Q_{r}} |D\bfh- (D\bfh)_{r}| \, \d z\,.
\end{equation}
\end{lemma}

\begin{proof} 
%\comment{Jihoon: I followed the proof of \cite[Lemma 2.8]{FILV23}. The reference \cite{Cam66} is written in Italian. Do you agree with that the next next estimates can be followed from  \cite{Cam66}? }
In view of \cite[(5.9)--(5.12)]{Cam66}, one can have 
$$
\sup_{Q_{r/2}} (|\bfh |+ r |D\bfh| +r^2|D^2\bfh|)  \le c \fint_{Q_r} |\bfh |\dz\,.
$$
Since every $\bfh_{x_i}$, $i=1,2,\dots,n$, is also $\A$-harmonic,  \eqref{DhLip} directly follows from the previous inequality. 
Let 
$$
{\bm \ell}(x) := (D\bfh)_{r} x.
$$
Suppose $\theta\in(0,1/2]$. We note from the mean value theorem for $D\bfh$ in $Q_{\theta\rho}$ that
$$\begin{aligned}
\sup_{Q_{\theta\rho}} |D\bfh- (D\bfh)_{\theta r}| 
&\le 2 \theta r  \sup_{Q_{\theta r}} |D^2\bfh| + 2 \theta^2r^2 \sup_{z\in Q_{\theta r}} |[D\bfh]_t|\\
&= 2 \theta r \sup_{Q_{\theta r}} |D^2(\bfh-{\bm \ell})| + 2 \theta^2 r^2 \sup_{Q_{\theta r}} |D \div [\A D (\bfh -{\bm \ell})]|\\
&\le 2 \theta r \sup_{Q_{r/2}} |D^2(\bfh-{\bm \ell})| +  2 \theta^2 r^2 \sup_{ Q_{r/2}} |D^3(\bfh -{\bm \ell})|\\
&\le 2 \theta \left(r \sup_{Q_{r/2}} |D^2(\bfh-{\bm \ell})| +   r^2 \sup_{ Q_{r/2}} |D^3(\bfh -{\bm \ell})|\right) \,.
\end{aligned}$$
Since every $(\bfh-{\bm \ell})_{x_i}$, $i=1,2,\dots,n$, is also $\A$-caloric in $Q_{r}$, we have from \eqref{DhLip} that
$$\begin{aligned}
\sup_{Q_{\theta r}} |D\bfh- (D\bfh)_{\theta r}|   
&\le  c \theta  \fint_{Q_r} |D(\bfh-{\bm \ell})| \, \d z\,,
%\qedhere
\end{aligned}$$
which implies \eqref{Dhdecay}.
\end{proof}

We introduce the parabolic Calder\'on-Zygmund estimates for an $N$-function $\psi$ with $\Delta (\psi,\psi^*)<\infty$. We shall assume that $\psi(1)=1$ without loss of generality.
In the next lemma, if $\psi(\tau)=\tau^p$, $1<p<\infty$, the estimate \eqref{CZestimate} is well known, see for instance \cite{KRW22} and references therein. Furthermore, for general Orlicz functions it can be obtained by applying a standard interpolation argument for linear operators as in \cite[Theorem 18]{DLSV12}. We also refer to \cite{BOPS16} for more general results for parabolic Calder\'on-Zygmund estimates in Orlicz spaces.
 
\begin{lemma} \label{Lem:paraCZ}
(Calder\'on-Zygmund estimates)
Let $\psi$ be an N-function with $\Delta (\psi,\psi^*)<\infty$ and  $\mathbf{G}\in L^{\psi}(Q_r, \R^{Nn})$ where $Q_r=B_r\times (-r^2,r^2)$. There exists a unique weak solution  $\bfu \in L^\psi(-r^2,r^2; L^\psi(B_r))$ with $\bfw_t \in (L^{\psi}((-r^2,r^2); W^{1,\psi}_0(\Omega)))'$  to the system
$$
\begin{cases}
\partial_t \bfw- \div (\A D\bfw) = -\div \, \mathbf{G}  \quad \text{in}\ \ Q_r, \\
 \bfw =\zero \quad \text{on}\ \ \partial_{\mathrm{p}} Q_r,
\end{cases}$$
and we have the estimates
%
%
%If $\bfw$ is a weak solution to the system
%$$
%\begin{cases}
%\partial_t \bfw- \div (\A D\bfw) = -\div \, \mathbf{G}  \quad \text{in}\ \ Q_r, \\
% \bfw =\zero \quad \text{on}\ \ \partial_{\mathrm{p}} Q_r,
%\end{cases}$$
%and  $\psi$ is an N-function  with $\Delta (\psi,\psi^*)<\infty$, then
$$
\|D\bfw \|_{L^\psi(Q_r)} \le c \|{\bf G} \|_{L^\psi(Q_r)}\,,
$$
and
\begin{equation}\label{CZestimate}
\int_{Q_r} \psi(|D\bfw |)\, \d z \le c \int_{Q_r} \psi(|\mathbf G |) \,\d z\,,
\end{equation}
where the constant $c>0$ depends on $n, N, \nu, |\A|$ and $\Delta (\psi,\psi^*)$.
\end{lemma}

\begin{remark}
\label{Rem:paraCZ}
Analogous estimates as above can be inferred for the weak solution $\bfv$ to 
\begin{equation}\label{system:dual}
\begin{cases}
\partial_t \bfv + \div (\A^T D\bfv) = -\div \, \mathbf{G}  \quad \text{in}\ \ Q_r, \\
\bfv =\zero \ \ \ \text{on}\ \ (\partial B_r\times\{ -r^2< t \le r^2\})\cup (B_r \times \{t=r^2\}),
\end{cases}\end{equation}
by considering the reflecting function $\widetilde \bfv(x,t)=\bfv(x,-t)$. Here, $\A^T$ is the transpose of $\A$, and note that if  $\A$  satisfies the Legendre-Hadamard condition then so does $\A^T$.
%\comment{Jihoon: By the reflection, sign in front of $\div$ is changed and the initial condition is replace by the terminal condition.}
\end{remark}
%\comment{Bianca:maybe we can directly put $\A^T$ for $\bfv_G.$
%
%Jihoon: I fix it.}

We estimate the gradient of a function in $L^{\psi}(-r^2,r^2;W^{1,\psi}_0(B_r))$ in terms of functions in the dual space $L^{\psi^*}$. 

\begin{lemma}\label{Lem:Acaloric1}
For every $\bfw \in L^{\psi}(-r^2,r^2;W^{1,\psi}_0(B_r))$, we have
$$\begin{aligned}
 \int_{Q_r}\psi (|D \bfw |)\,\d z   \le   \sup_{{\bf G}\in L^{\psi^*}\cap C^{\infty}(Q_r,\R^{Nn})}  \left( \int_{Q_r} \bfw \cdot ( \bfv_{{\bf G}} )_t -\langle \A D\bfw ,  D\bfv_G \rangle       - \psi^*( | {\bf G} | ) \, \d z \right),
\end{aligned}$$
where $\bfv_{{\bf G}}$ is the weak solution to \eqref{system:dual}.
\end{lemma}

\begin{proof}
We note from the definition of the conjugate function in \eqref{conjugate} that
$$
\psi(t) =\psi^{**}(t) = \sup_{s\ge 0 }\left( s t  - \psi^*(s)\right) = t \psi'(t) - \psi^*( \psi'(t) )\,.
$$
Hence, denoting ${\bf G}_{\bfw} := \psi'(|D\bfw|)\frac{D\bfw}{|D\bfw|}$, we have ${\bf G}_{\bfw}\in L^{\psi^*}(Q_r,\R^{Nn})$ by \eqref{eq:hok2.4}, and 
$$\begin{aligned}
\int_{Q_r} \psi(|D\bfw|)\, \d z & = \int_{Q_R} |D\bfw| \psi'(|D\bfw|) - \psi^*( \psi'(|D\bfw|) ) \, \d z \\
& = \int_{Q_r}  \langle  {\bf G}_{\bfw}, D\bfw \rangle  - \psi^*( |{\bf G}_{\bfw}| ) \, \d z  \\
& \le \sup_{{\bf G}\in L^{\psi^*}(Q_r,\R^{Nn}) }\left( \int_{Q_r} \langle  {\bf G}, D\bfw\rangle   - \psi^*( |{\bf G}| ) \, \d z \right) \\
& = \sup_{{\bf G}\in L^{\psi^*}\cap C^\infty(Q_r,\R^{Nn}) }\left( \int_{Q_r} \langle  {\bf G}, D\bfw\rangle   -  \psi^*( |{\bf G}| ) \, \d z \right) \,.
\end{aligned}$$
For each  ${\bf G}\in L^{\psi^*}(Q_r,\R^{Nn})\cap C^{\infty}(Q_r,\R^{Nn})$, let $\bfv_G$ be the weak solution to \eqref{system:dual}. %with $\A^T$ in place of $\A$. 
Then we see that $\bfv_{{\bf G}}\in C^\infty(Q_r,\R^N)$ since $\A$ is constant, and by testing \eqref{system:dual} with $\bfw$ we have
$$\begin{aligned}
\int_{Q_r} \langle  {\bf G}, D\bfw\rangle  \, \d z 
=  \int_{Q_r}  ( \bfv_{{\bf G}} )_t \cdot \bfw - \langle \A^T D\bfv_{{\bf G}} , D\bfw \rangle  \, \d z  =  \int_{Q_r}  \bfw\cdot  ( \bfv_{{\bf G}} )_t  - \langle \A D\bfw , D\bfv_{{\bf G}} \rangle   \, \d z \,.
\end{aligned}$$
This concludes the proof.
\end{proof}

\subsection{Parabolic Lipschitz truncation}\label{subsec:Lip} 
We recall the parabolic Lipschitz truncations introduced in \cite{DieSchStrVer17} and their main properties, in the particular case when the scaling quantity $\alpha$ therein is equal to 1.

Let $\bfv \in L^\psi(- r^1, r^1; W^{1,1}_0(B_r,\R^{Nn}))$ and ${\bf G}\in L^1(Q_r,\R^{Nn})$ satisfy the system 
\begin{equation}\label{system:vG}
\bfv_t = \div {\bf G} \quad \text{in}\ \ Q_r , \qquad \bfv= {\bf{0}} \quad \text{on }\ \partial_{\mathrm p} Q_r,
\end{equation}
in the distributional sense.  We take as \lq\lq bad set" a superlevel set for the maximal function of the spatial gradient and of the time derivative in the following way:
\begin{align*} 
 \Ol & := { \{\mathcal M(\chi_{Q_r}\nabla \bfv)>\lambda\} \cup \{\mathcal M(\chi_{Q_r}{\bf G}) > \lambda\}}, \quad \lambda>0\,,
\end{align*}
where $\mathcal M$  is the  parabolic maximal operator defined by 
\begin{align*}
  \mathcal M(f)(\tilde z) &:= \sup_{Q_{\rho}(z_0) \,:\, \tilde z \in Q_{\rho}(z_0)}
  \fint_{Q_\rho(z_0)}|f|\,\d z\,.
\end{align*}
Then we have the following properties, which can be inferred by \cite[Theorem 2.3]{DieSchStrVer17} with $\alpha=1$, $\mathcal M:=\mathcal{M}^1$ and $\Ol:=\Ol^1$.
\begin{lemma}\label{Lem:Lip1}
Let $\bfv \in L^\psi(- r^2, r^2; W^{1,\psi}_0(B_r,\R^{Nn}))$ satisfy the system \eqref{system:vG}
in the distributional sense, and let $\lambda>0$. Then there exists $\bfv_\lambda \in L^{1}(- r^2, r^2;W^{1,1}_0(B_r))$ with $|D\bfv_\lambda| \in L^{\psi}(Q_r)$ such that
\begin{enumerate}
\item $\bfv_\lambda = \bfv$ on $({\Ol})^c$. 
\item $\mathcal M (D\bfv_\lambda) \le c \lambda$.
\item we have
$$
\int_{Q_r} \psi(|D(\bfv_\lambda -\bfv )|)\, \d z \le c   \int_{Q_r} \psi(|D\bfv |)\, \d z + c \psi(\lambda) |\Ol|\,.
$$
%\item $\val$ is H\"older continous with respect to the natural parabolic metric; i.e., 
%  \begin{align*}
%   |{\val(t,x) - \val(s,y)}| \leq c\, \lambda \max
%   \left\{ |t-s|^{\frac 12}, |x-y|
%   \right \}
%  \end{align*}
%  for all $(x,t),(y,s)\in Q_r$.
% \item $\partial_t\val\in W^{-1,\infty}(-r^2,r^2,W^{1,\infty}(B_r))$, moreover
%  \begin{align*}
% |\langle \partial_t\wal, \xi\rangle|\leq c\lambda \|{\partial_t\xi}\|_{L^1(-r^2,r^2,W^{-1,1}(B_r)}.
%  \end{align*}
% \item For $\xi\in W^{1,\psi^*}(-r^2,r^2,W^{-1,\psi^*}(B_r))$, we find that 
%    \begin{align*}
%  |\langle \partial_t(\wal-\bfw), \xi\rangle|\leq  c \|\nabla \bfw\|_{L^{\psi}(\Ol)}\|{\partial_t\xi}\|_{L^{\psi^*}(-r^2,r^2,W^{-1,p'}(B_r))}
%  \end{align*}
\end{enumerate}
Here the constants $c>0$ depend on $n,N$ and $\Delta_2(\psi,\psi^*)$.
\end{lemma}

We note from \cite[Section~2.3]{DieSchStrVer17}  that the function $\bfv_\lambda$ is determined by
\begin{equation}\label{def_vlambda}
\bfv_\lambda := \bfv - \sum_{i} \zeta_i (\bfv-\bfv_{i}), \quad \text{where } \bfv_i := 
\begin{cases}
(\bfv)_{\zeta_i} & \text{if }\ \frac{3}{4}Q_i \subset B_r\times (-r^2,3r^2)\,,\\
\zero &\text{otherwise,}
\end{cases}
\end{equation}
%\comment{\color{magenta} Jihoon: In the definition of $\bfv_i$ in \cite{DieSchStrVer17}, it is written that $\frac{3}{4}Q_i\subset J\times \Omega$, where $J=(-t_0,0]$. I think this is typo and $(-t_0,t_0)\times \Omega$ seems correct. Because, $\bfv$ is zero in the outside of not $(-t_0,0]\times \Omega$ but $(-t_0,t_0)\times \Omega$, which is used in \cite[Lemma 2.11 (b)]{DieSchStrVer17}.}
where we extend $\bfv$ and $\bf G$ to $B_r \times (r^2,3r^2)$ by $\bfv(x,2r^2-t)$ and $- {\bf G}(x,2r^2-t)$  and to the outside of $B_r \times (-r^2,3r^2)$ by zeros, hence this extended $\bfv$ satisfies  the system $\bfv_t=\div {\bf G}$ in $B_r\times (-\infty,\infty)$ in the sense of distributions, 
and $\{Q_i\}_{i=1}^\infty$ is a parabolic Whitney covering  of $\Ol$ such that $Q_j =Q_{r_i}(z_i)$,
\begin{enumerate}[label={\rm (W\arabic{*})}, leftmargin=*]
\item\label{itm:whit1} $\bigcup_i\frac {1} {2} Q_i \,=\, \Ol$,
\item\label{itm:whit2} for all $j\in \mathbb{N}$ we have $8
  Q_i \subset \Ol$ and $16 Q_i \cap (\mathbb{R}^{m+1}\setminus
  \Ol)\neq \emptyset$, 
\item \label{itm:whit3} if $ Q_i \cap Q_j \neq \emptyset
  $ then $ \frac 12 r_j\le r_i\leq 2\, r_j$,
\item \label{itm:whit3b}$\frac 14 Q_i \cap \frac 14Q_j =
  \emptyset$  for all $i \neq j$,
\item \label{itm:whit4}  each $x\in  \Ol $ belongs to at most 
  $120^{n+2}$ of the sets $4Q_i$.
\end{enumerate}
Here, $\kappa Q_i := Q_{\kappa r_i}$ for $\kappa>0$, and 
$\{\zeta_i\} \subset C^\infty_0(\mathbb{R}^{n+1})$ is a partition of
unity  with respect to  $\{Q_i\}$ that satisfies  
\begin{enumerate}[label={\rm (P\arabic{*})}, leftmargin=*]
\item \label{itm:P1} $\chi_{\frac{1}{2}Q_i}\leq \zeta_i\leq
  \chi_{\frac 34 Q_i}$,
%\item \label{itm:P2} $\sum_k\rho_k=\sum_{k\in A_j}\rho_k=1$ on $Q_j^{\alpha}$,
\item \label{itm:P3} $\|{\zeta_i}\|_\infty + r_i \|{D
    \zeta_i}\|_\infty + r_i^2 \|{D^2 \zeta_i}\|_\infty + 
  r_i^2 \| (\zeta_i)_{t}\|_\infty \leq c$,
\item \label{itm:P4} For each $j \in \mathbb{N}$ we define $A_j:= \{ i \,:\, \frac 34
  Q_j \cap \frac 34 Q_i \neq \emptyset\}$. Then
 $\sum_{i \in A_j} \zeta_i = 1$ on $\frac 34 Q_j$.
\end{enumerate}

The following result provides an upper bound for the measure of the bad set $\Ol$, see \cite[Lemma 4.1]{DieSchStrVer17} with $\alpha=1$. 

\begin{lemma}\label{Lem:Lip2} %(\cite[Lemma 4.1]{DieSchStrVer17} with $\alpha=1$)
Let $\bfv \in L^\psi(-r^2,r^2;W^{1,\psi}_0(B_r))$ and ${\bf G}\in L^{\psi^*}(Q_r)$ satisfy \eqref{system:vG} in the distribution sense. Set $\gamma>0$ such that 
$$
\psi(\gamma):= \fint_{Q_r} \psi(|D \bfv|) \, \d z +\fint_{Q_r} \psi(|{\bf G}|) \, \d z \,.
$$
Then, for every $m_0\in \mathbb N$, there exists $\lambda\in [\gamma, 2^{m_0}\gamma]$ such that
$$
|\mathcal O_\lambda | \le c \frac{\psi(\gamma)}{m_0 \psi(\lambda)} |Q_r|
$$
for some $c>0$ depending on $n,N$ and $\Delta_2(\psi,\psi^*)$.
\end{lemma}

%\comment{\color{blue} Jihoon: I add the following lemma.}

We end this subsection presenting a Poincar\'e-type inequality. Note that the following lemma is irrelevant to the above setting.

\begin{lemma}\label{Lem:poincare_w}
Let $\bfw \in L^\psi(- r^2, r^2; W^{1,1}_0(B_r,\R^{Nn}))$ and ${\bf H}\in L^1(Q_r,\R^{Nn})$ satisfy the system 
$$
\bfw_t = \div {\bf H} \quad \text{in}\ \ Q_r , \qquad \bfw={\bf 0} \quad \text{on }\ \partial_{\mathrm p} Q_r,
$$
in the distributional sense. Extend $\bfw$ and $\bf H$ by $\bfw(x,2r^2-t)$ and $-{\bf H}(x,2r^2-t)$  to $B_r\times (r^2,3r^2)$ and by zero outside $B_r\times (-r^2,3r^2)$. For any parabolic cylinder $Q_\rho=Q_\rho(z)$ in $\R^{n+1}$ and any $\zeta\in C^\infty_0(\frac{3}{4}Q_\rho)$ with $\zeta\ge 0$ and $\|\zeta\|_{L^\infty(\frac{3}{4}Q_\rho)}\le c_0|{\frac{3}{4}}Q_\rho|^{-1}\|\zeta\|_{L^1(\frac{3}{4}Q_\rho)}$, set
$$
\overline\bfw := 
\begin{cases}
(\bfw)_{\zeta} & \text{if }\ \frac{3}{4}Q_\rho \subset B_r\times (-r^2,3r^2)\,,\\
\zero &\text{otherwise.}
\end{cases}
$$
Then we have
$$
\fint_{\frac{3}{4}Q_{\rho}} \psi\left(\frac{\bfw-\overline \bfw}{\rho}\right)\,\dz \le c  \fint_{Q_\rho} \psi(|D\bfw|)\,\dz + c \psi\left(\fint_{Q_\rho} |{\bf H}|\,\dz\right)
$$ 
for some $c>0$ depending on $n,N, \Delta_2(\psi,\psi^*)$ and $c_0$.
\end{lemma}

\begin{proof}
The proof is almost the same as the one of \cite[Lemma 2.11]{DieSchStrVer17} with replacing \cite[Lemma 2.8]{DieSchStrVer17}  by \cite[Lemma 2.9]{DieSchStrVer17}. In fact, if  $\frac{3}{4}Q_\rho \subset B_r \times (-r^2,3r^2)$, then the inequality follows directly from  \cite[Lemma 2.9]{DieSchStrVer17}.

If $\frac{3}{4}Q_\rho \not\subset B_r\times (-r^2,3r^2)$ and $\frac{4}{5}Q_\rho \subset B_r\times (-\infty,\infty)$, choose $\tilde\zeta\in C^\infty_0(\frac{4}{5}Q_\rho)$ with $\tilde\zeta\ge 0$, $\mathrm{supp}(\tilde\zeta)\subset \frac{4}{5}Q_\rho\setminus (B_r\times (-r^2,3r^2))$ and $\|\zeta\|_{L^\infty(\frac{4}{5}Q_\rho)}\le c(n)|{\frac{4}{5}}Q_\rho|^{-1}\|\zeta\|_{L^1(\frac{4}{5}Q_\rho)}$. Then, since $\bfw \equiv {\bf 0}$ in $\mathrm{supp}(\tilde\zeta)$, we have $(\bfw)_{\tilde \zeta}=\bf 0$ hence again by \cite[Lemma 2.9]{DieSchStrVer17}
$$
\fint_{\frac{3}{4}Q_{\rho}} \psi\bigg(\frac{\bfw-\overline \bfw}{\rho}\bigg)\,\dz   = \fint_{\frac{3}{4}Q_{\rho}} \psi\bigg(\frac{\bfw-(\bfw)_{\tilde\zeta}}{\rho}\bigg)\,\dz \le c \fint_{\frac{4}{5}Q_\rho} \psi(|D\bfw|)\,\dz + c \psi\bigg(\fint_{\frac{4}{5}Q_\rho} |{\bf H}|\,\dz\bigg)\,.
$$
Finally, if $\frac{4}{5}Q_\rho \subset B_r\times (-\infty,\infty)$, then there exists $c(n)>0$ such that $\frac{|B_r|}{|\{\bfw(x,t)= {\bf 0}\}\cap B_r\}|}\le c(n)$ for a.e. time slice of $Q_\rho$. Therefore by the Poincar\'e inequality for the space variable, see \cite[Theorem 7]{DieEtt08}, we have
$$
\fint_{\frac{3}{4}Q_{\rho}} \psi\left(\frac{\bfw-\overline \bfw}{\rho}\right)\,\dz \le  \fint_{Q_{\rho}} \psi\left(\frac{\bfw}{\rho}\right)\,\dz \le c \fint_{Q_\rho} \psi(|D\bfw|)\,\dz\,. 
$$ 
The proof is concluded. 
\end{proof}

\subsection{Caloric approximations}

We first obtain the $\A$-caloric approximation. 
As a novelty with respect to previous caloric type approximations, we do not need to restrict the choice of the test functions $\bm\zeta$ in $C^\infty_0(Q_r)$, but we only assume them to be zero on the lateral boundary. This allows us to choose as test functions also the solutions of suitable linear systems.

\begin{theorem}\label{thm:Acaloric}  ($\mathcal A$-caloric approximation)
Let $\mu,\sigma,C_0>0$ and $\psi$ be an $N$-function with $\Delta_2(\psi,\psi^*)<\infty$. For every $\epsilon\in(0,1)$, there exists $\delta>0$ depending on $\sigma$, $C_0$, $\Delta_2(\psi,\psi^*)$ and $\epsilon$   such that
if $\bfu\in L^1(-r^2,r^2;W^{1,\psi^{1+\sigma}}(B_r,\R^N))$ and ${\bf H}\in L^{\psi^{1+\sigma}}(Q_r,\R^{Nn})$ satisfy
$$
\partial_t \bfu = \div {\bf H}  \quad \text{in }\ Q_r,
$$
in the distributional sense, with the inequality
\begin{equation}\label{Acaloric_ass1}
\left( \fint_{Q_r} \psi(|D\bfu|)^{1+\sigma}+ \psi(|{\bf H}|)^{1+\sigma}\, \d z \right)^{\frac{1}{1+\sigma}}  \le C_0 \psi(\mu),
\end{equation}
and for every $\bm\zeta\in C^\infty(Q_r;\R^{N})$ with $\bm\zeta={\bf 0}$ on $\partial B_r \times (-r^2,r^2)$,
\begin{equation}\label{Acaloric_ass2}
\frac{1}{|Q_r|}\left|  \int_{Q_r} \bfu \cdot \bm\zeta_t - \langle\mathcal A D\bfu , D\bm\zeta\rangle \, \d z - \left[\int_{B_r} \bfu \cdot \bm\zeta_t \, \d x\right]^{t=r^2}_{t=-r^2}\right| \le \delta \mu  \|D\bm\zeta\|_{L^\infty(Q_r,\R^{Nn})},
\end{equation}
%\comment{Jihoon: Note that $\zeta$ is NOT zero on the bottom and the top on $Q_r$.}
then 
\begin{equation}\label{eq:Acaloric}
\fint_{Q_r} \psi(|D\bfu-D\bfh|)\,\d z \le \epsilon \psi(\mu),
\end{equation}
where $\bf h$ is the weak solution to
$$
\begin{cases}
\partial_t \bfh- \div (\A D\bfh) = {\bf 0}  \quad \text{in}\ \ Q_r, \\
 \bfh =\bfu \quad \text{on}\ \ \partial_{\mathrm{p}} Q_r.
\end{cases}$$
\end{theorem}

\begin{proof}
It will suffice to prove the assertion in the case $\mu=1$ with $\psi(1)=1$, as the general case 
%We prove the theorem only the case $\mu=1$. Note that  the general case} 
can be obtained by a scaling argument with the functions $\tilde \bfu = \mu^{-1}\bfu$, $\tilde {\bf H} = \mu^{-1} {\bf H}$ and $\tilde\psi(\tau)=\frac{\psi(\mu \tau)}{\psi(\mu)}$.
Set $\bfw:= \bfu-\bfh$. Then $\bfw $ satisfies  
\begin{equation}\label{system:w}
\begin{cases}
\partial_t \bfw- \div (\A D\bfw) = -\div \, (\A D \bfu+ {\bf H})  \quad \text{in}\ \ Q_r, \\
 \bfw =\zero \quad \text{on}\ \ \partial_{\mathrm{p}} Q_r,
\end{cases}\end{equation}
in the distributional sense. Moreover, by applying Lemma~\ref{Lem:paraCZ} to the $N$-function $\psi^{1+\sigma}$ and \eqref{Acaloric_ass1}, we see that
\begin{equation}\label{Dwestimate}
\fint_{Q_r} \psi(|D\bfw|)^{1+\sigma}\, \d z 
\le  c  \fint_{Q_r} \psi(|\A D\bfu + {\bf H}|)^{1+\sigma}\, \d z \le c \,.
\end{equation}
We will apply the inequality in Lemma~\ref{Lem:Acaloric1}. Fix any ${\bf G}\in L^{\psi^*}(Q_r,\R^{Nn})\cap C^{\infty}(Q_r,\R^{Nn})$ and consider the weak solution $\bfv_{\bf G}$ to \eqref{system:dual}. Note that, by Lemma~\ref{Lem:paraCZ} with Remark~\ref{Rem:paraCZ} and $\psi^*$ in place of $\psi$,
\begin{equation}\label{Dvestimate}
\int_{Q_r} \psi^*(|D\bfv_{\bf G}|)\, \d z \le c \int_{Q_r} \psi^*(|{\bf G}|)\, \d z\,.
\end{equation}
Moreover, $\bfv_{\bf G} \in C^\infty(Q_r,\R^N)$ since ${\bf G}\in C^{\infty}(Q_r, \R^{Nn})$. To enlighten the notation, from now on we will simply denote $\bfv_{\bf G}$ by $\bfv$.

Choose $\gamma\in [0,\infty)$ such that
\begin{equation}\label{Acaloric_gamma}
\psi^*(\gamma) = \int_{Q_r} \psi^*(|D\bfv|) \, \dz+
  \int_{Q_r} \psi^*(|\A^T D\bfv +{\bf G}|) \, \dz\,.
\end{equation}
Then, with \eqref{Dvestimate} and the subadditivity of $\psi^*$ we have 
$$
\psi^*(\gamma) \le c \int_{Q_r} \psi^*(|{\bf G}|) \, \dz\,.
$$
Let $m_0\in\mathbb N$ be large enough, to be determined later. Then by Lemma~\ref{Lem:Lip1}(1) and Lemma~\ref{Lem:Lip2} with $\psi^*$ in place of $\psi$, there exists   $\lambda \in[\gamma,2^{m_0}\gamma]$ such that  
$\{\bfv\neq \bfv_\lambda \} \subset \Ol$
and 
\begin{equation}\label{Acaloric_pf0}
\frac{|\mathcal O_\lambda|}{|Q_r|} \le \frac{c\psi^*(\gamma)}{m_0 \psi^*(\lambda)} \le \frac{c}{m_0}\,,
\end{equation}
where $\bfv_\lambda$ is the parabolic Lipschitz truncation of $\bfv$ provided by Lemma~\ref{Lem:Lip1}.  
Note that $\bfv$ is zero on the top, but not on the base, of the cylinder $Q_r$. Hence we apply the Lipschitz truncation and related results in the previous subsection to the function $\bfv(x,-t)$. Accordingly,  $\bfv$ and ${\bf \tilde G}:= -\A^T D\bfv -{\bf G}$ are extended to $B_r\times(-3r^2,-r^2)$ by $\bfv(x,t)=\bfv(x,-2r^2-t)$ and  ${\bf \tilde G}(x,t)= - {\bf \tilde G}(x,-2r^2-t)$ and to the outside of $B_r\times(-3r^2,r^2)$ by zeros, hence from \eqref{system:dual} $\bfv$ satisfies the system
$\bfv_t = \div {\bf \tilde G}$ in $B_R\times (-\infty,\infty)$ in the distribution sense.

Then we observe that 
\begin{equation}\label{Acaloric_pf1}\begin{aligned}
 \int_{Q_r}  \bfw \cdot \bfv_t  - \langle \A D\bfw ,  D\bfv \rangle  \, \d z     
&=  \int_{Q_r}  \bfw \cdot (\bfv_\lambda)_t - \langle \A D\bfw ,  D\bfv_\lambda \rangle  \, \d z \\
&\quad 
+ \int_{Q_r}  \bfw \cdot (\bfv-\bfv_\lambda)_t  \, \d z 
-  \int_{Q_r} \langle \A D\bfw ,  D (\bfv-\bfv_\lambda \rangle  \, \d z    \\
& =: I_1 + I_2 - I_3 \,.
\end{aligned}\end{equation}
For $I_1$,
%since $\bfv_\lambda(x,r^2)= \bfv(x,r^2)=0$ for $x\in B_r$ (see Remark~\ref{Rem:paraCZ}),
since $\bfw= \bfu + \bfh$,
$$\begin{aligned}
I_1&= \int_{Q_r}  \bfw \cdot (\bfv_\lambda)_t - \langle \A D\bfw ,  D\bfv_\lambda \rangle  \, \d z -\left[\int_{Q_r}  \bfw \cdot \bfv_\lambda \, \d x\right]^{t=r^2}_{t=-r^2} \\
&= \int_{Q_r}  \bfu \cdot (\bfv_\lambda)_t - \langle \A D\bfu ,  D\bfv_\lambda \rangle  \, \d z -\left[\int_{Q_r}  \bfu \cdot \bfv_\lambda \, \d x\right]^{t=r^2}_{t=-r^2}.
\end{aligned}$$
Then by \eqref{Acaloric_ass2}, Lemma~\ref{Lem:Lip1}(2) and Young's inequality we have that for any $\kappa_1\in(0,1)$,
$$
\frac{1}{|Q_r|} I_1 \le \delta \|D\bfv_\lambda\|_{L^\infty(Q_r,\R^{Nn})} \le c\delta \lambda \le c_{\kappa_1}\psi\left(\delta \right) +  \kappa_1 \psi^*(\lambda) \le c_{\kappa_1}\psi\left(\delta\right) +  \kappa_1 \psi^*( 2^{m_0}\gamma).
$$
We next estimate $I_3$.  By Young's inequality, H\"older's inequality and Lemma~\ref{Lem:Lip1}(3) with $\psi^*$ in place of $\psi$ and Lemma~\ref{Lem:Lip2}, we have that for any $\kappa_2\in(0,1)$
$$\begin{aligned}
|I_3| &\le c_{\kappa_2} \int_{\mathcal O_\lambda \cap Q_r} \psi(|D\bfw|) \,\d z +  \kappa_2 \int_{\mathcal O_\lambda \cap Q_r} \psi^*(|D(\bfv-\bfv_\lambda)|) \,\d z    \\
& \le  c_{\kappa_2} \int_{\mathcal O_\lambda \cap Q_r} \psi(|D\bfw|) \,\d z  + c \kappa_2  \int_{Q_r} \psi^*(|D\bfv|) \,\d z +  c \kappa_2|\mathcal O_\lambda  |\psi^*(\lambda)   \\
& \le  c_{\kappa_2} \left(\int_{Q_r} \psi(|D\bfw|)^{1+\sigma} \,\d z\right)^{\frac{1}{1+\sigma}} |\mathcal O_\lambda |^{\frac{\sigma}{1+\sigma}} + c \kappa_2 \int_{Q_r} \psi^*(|D\bfv|) \,\d z   +  c \kappa_2 |\mathcal O_\lambda |\psi^*(\lambda)\,,   
\end{aligned}$$
hence applying \eqref{Dwestimate} and \eqref{Acaloric_pf0}
$$\begin{aligned}
\frac{1}{|Q_r|}|I_3| &\le  c_{\kappa_2}m_0^{-\frac{\sigma}{1+\sigma}}  +  c\kappa_2 \fint_{Q_r} \psi^*(|D\bfv|) \,\d z+    \frac{c \kappa_2}{m_0} \psi^*(\gamma)\,.
\end{aligned}$$
Finally, we estimate $I_2$. Recall the parabolic Whitney covering $\{Q_i\}_{i=1}^\infty$ of $\mathcal O_\lambda$ and the partition of unity $\{\zeta_i\}_{i=1}^\infty\subset C^\infty_0(\frac{3}{4}Q_i)$ with respect to $\bfv$ and the definition of $\bfv_\lambda$ in \eqref{def_vlambda}. 
In addition, we extend $\bfw$ and ${\bf \tilde H}:= \A (D\bfw - D \bfu-{\bf H})$ to $B_r\times(r^2,3r^2)$ by $\bfw(x,t)=\bfw(x,2r^2-t)$ and  ${\bf \tilde H}(x,t)=-{\bf \tilde H}(x,2r^2-t)$ and to the outside of $B_r\times(-r^2,3r^2)$ by zeros, hence from \eqref{system:w} $\bfw$ satisfies the system
$\bfw_t = \div {\bf \tilde H}$ in $B_R\times (-\infty,\infty)$ in the distribution sense. With this extended $\bfw$, we set
$$
 \bfw_i := 
\begin{cases}
(\bfw)_{\zeta_i} & \text{if }\ \frac{3}{4}Q_i \subset B_r\times (-r^2,3r^2)\,,\\
\zero &\text{otherwise.}
\end{cases}
$$
Then, since $\bfw_i$'s are constants and $\bfv$ solves \eqref{system:dual}, we have
$$\begin{aligned}
I_2  &\le c \sum_{\frac{3}{4}Q_i\cap Q_r \neq \emptyset} \left| \int_{\frac{3}{4} Q_i \cap Q_r}  \bfw  \cdot [ \zeta_i (\bfv-\bfv_i)]_t    \, \d z \right| \\ 
&= c \sum_{\frac{3}{4}Q_i\cap Q_r \neq \emptyset} \left| \int_{\frac{3}{4} Q_i \cap Q_r}  (\bfw  - \bfw_i ) \cdot [ \zeta_i (\bfv-\bfv_i)]_t    \, \d z \right| \\ 
&= c \sum_{\frac{3}{4}Q_i\cap Q_r \neq \emptyset} \left|\int_{\frac{3}{4} Q_i \cap Q_r}  (\bfw  - \bfw_i ) \cdot  \left[(\bfv-\bfv_i) (\zeta_i)_t + \bfv_t  \zeta_i\right]    \, \d z\right|  \\
&= c \sum_{\frac{3}{4}Q_i\cap Q_r \neq \emptyset} \left|\int_{\frac{3}{4} Q_i \cap Q_r}   (\bfw  - \bfw_i ) \cdot  (\bfv-\bfv_i) (\zeta_i)_t + \langle (\A^T D\bfv +{\bf G}), D [ (\bfw-\bfw_i) \zeta_i]\rangle    \, \d z  \right| \\ 
&\le c \sum_{\frac{3}{4}Q_i\cap Q_r \neq \emptyset} \int_{\frac{3}{4} Q_i \cap Q_r}   \frac{|\bfw  - \bfw_i|}{r_i} \frac{ |\bfv-\bfv_i|}{r_i} + (|D\bfv| +|{\bf G}|) \left(|D\bfw| + \frac{|\bfw- \bfw_i|}{r_i}\right)    \, \d z \,.
\end{aligned}$$
Moreover, by Young's inequality we have that for any $\kappa_3\in(0,1)$,
$$\begin{aligned}
I_2   &\le  c_{\kappa_3} \sum_{\frac{3}{4}Q_i\cap Q_r \neq \emptyset} \int_{\frac{3}{4} Q_i}   \psi\left(\frac{|\bfw  - \bfw_i|}{r_i} \right)  +  \psi(|D\bfw|)    \, \d z\\
&\qquad + \kappa_3  \sum_{\frac{3}{4}Q_i\cap Q_r \neq \emptyset}  \int_{\frac{3}{4} Q_i }   \psi^*\left( \frac{ |\bfv-\bfv_i|}{r_i}\right)  +\psi^*(|D\bfv| +|{\bf G}|)    \, \d z  \,.
\end{aligned}$$
Then, applying Lemma~\ref{Lem:poincare_w} to $\bfw$ and $\bfv$ with the extensions of $\bfw$, $\bfv$, $\bf H$ and $\bf G$, we have that
$$\begin{aligned}
\int_{\frac{3}{4} Q_i}   \psi\left(\frac{|\bfw  - \bfw_i|}{r_i} \right)    \, \d z & \le 
c \int_{ Q_i}   \psi (|D\bfw |)    \, \d z +  \int_{  Q_i}   \psi (|\A D\bfw -\A D \bfu -{\bf H} |)    \, \d z \\
&\le c \int_{ Q_i}  \psi (|D\bfw| + |D\bfu| +|{\bf H} |)    \, \d z \,,
\end{aligned}$$
and 
$$\begin{aligned}
\int_{\frac{3}{4} Q_i}   \psi^*\left(\frac{|\bfv  - \bfv_i|}{r_i} \right)    \, \d z & \le 
c \int_{  Q_i}   \psi^* (|D\bfv |)    \, \d z +  \int_{  Q_i}   \psi^* (|-\A^T D\bfv -{\bf G} |)    \, \d z \\
&\le c \int_{  Q_i}   \psi^* (|D\bfv |+|{\bf G}|)    \, \d z \,.
\end{aligned}$$
Using these inequalities, the fact that $\sum_{i} \chi_{Q_i}\le c(n)$ and considering the extension of functions, we estimate $I_2$ as  
$$\begin{aligned}
I_2& \le c_{\kappa_3} \int_{\mathcal O_\lambda} \psi (|D\bfw| + |D\bfu| +|{\bf H} |)    \, \d z +  c \kappa_3 \int_{\mathcal O_\lambda} \psi^* (|D\bfv| +|{\bf G} |)    \, \d z\\
& \le c_{\kappa_3} \left(\int_{\mathcal O_\lambda} \psi (|D\bfw| + |D\bfu| +|{\bf H} |)^{1+\sigma}    \, \d z\right)^{\frac{1}{1+\sigma}} |\mathcal O_\lambda|^{\frac{\sigma}{1+\sigma}}+  c \kappa_3 \int_{\mathcal O_\lambda} \psi^* (|D\bfv| +|{\bf G} |)    \, \d z\\
& \le c_{\kappa_3} \left(\int_{Q_{r}} \psi (|D\bfw| + |D\bfu| +|{\bf H} |)^{1+\sigma}    \, \d z\right)^{\frac{1}{1+\sigma}} |\mathcal O_\lambda|^{\frac{\sigma}{1+\sigma}}+  c \kappa_3 \int_{Q_r} \psi^* (|D\bfv| +|{\bf G} |)    \, \d z\,.
\end{aligned}$$
Therefore, by H\"older's inequality  and the estimates \eqref{Acaloric_ass1}, \eqref{Dwestimate}, \eqref{Dvestimate} and \eqref{Acaloric_gamma}, we have 
$$\begin{aligned}
\frac{1}{|Q_r|}I_2\le c_{\kappa_3} \left(\frac{|\mathcal O_\lambda|}{|Q_r|}\right)^{\frac{\sigma}{1+\sigma}}+  c \kappa_3 \fint_{Q_r} \psi (|D\bfv| +|{\bf G} |)    \, \d z \le c_{\kappa_3}m_0^{-\frac{\sigma}{1+\sigma}} + c\kappa_3\fint_{Q_r} \psi^* (|{\bf G} |)    \, \d z\, .
\end{aligned}$$
Inserting the estimates for $I_1$, $I_2$ and $I_3$ into \eqref{Acaloric_pf1}, we have 
$$\begin{aligned}
\fint_{Q_r}  \bfw \cdot \bfv_t  - \langle \A D\bfw ,  D\bfv \rangle  \, \d z    & \le c_{\kappa_1} \psi(\delta) + (c_{\kappa_2}+c_{\kappa_3})m_0^{-\frac{\sigma}{1+\sigma}} \\
&\qquad +(c_{m_0} \kappa_1 + c\kappa_2+c\kappa_3)\fint_{Q_r} \psi^* (|{\bf G} |)    \, \d z\,.
\end{aligned}$$
We choose $\kappa_2,\kappa_3$  small so that $c\kappa_2+c\kappa_3\le \frac{1}{2}$, then $m_0$  large so that $ (c_{\kappa_2}+c_{\kappa_3})m_0^{-\frac{\sigma}{1+\sigma}} \le \frac{\epsilon}{2}$, then $\kappa_1$ small so that $c_{m_0} \kappa_1 \le \frac12$, and then  $\delta$  small so that $c_{\kappa_1} \psi(\delta) \le \frac{\epsilon}{2}$. Then we have 
$$
\fint_{Q_r}  \bfw \cdot \bfv_t  - \langle \A D\bfw ,  D\bfv \rangle  \, \d z \le \epsilon + \fint_{Q_r}\psi^*(|{\bf G}|)\,\dz\,.
$$
Since $\bf G$ is an arbitrary function in $L^{\psi}(Q_r,\R^{Nn})\cap C^\infty(Q_r,\R^{Nn})$, by Lemma~\ref{Lem:Acaloric1} we deduce \eqref{eq:Acaloric}. This concludes the proof.  %\eqref{eq:Acaloric1}.
\end{proof}

The $\phi$-caloric approximation has been proved in  \cite[Theorem 4.2]{DieSchStrVer17}.  It states
 that  every {\lq\lq almost $\phi$-caloric\rq\rq}  function has a $\phi$-caloric function {\lq\lq close enough\rq\rq}.

\begin{theorem}($\phi$-caloric approximation)\label{thm:phi-caloric}
Let $\gamma_1,\gamma_2\in (0,1)$, $\gamma_3\ge 1$, $I:=(t^-,t^+)$. Suppose $\phi$ be an N-function with $\Delta_2(\phi,\phi^*)<\infty$ with $\phi(1)=1$. Then   for every $\epsilon>0$ there exists $\delta>0$ depending on $n,N,\Delta_2(\phi,\phi^*),\gamma_1,\gamma_2,\gamma_2$ and $\epsilon$ such that the  following holds: if  $\bfu\in L^{\phi}(I,W^{1,\phi}(B))$ satisfying $\bfu_t=\div {\bf G}$ in the distribution sense is almost  {  $\phi$}-caloric in the sense that for all {  $ \bm\zeta\in C_0^\infty(Q),$}
$$
   \Big|\fint_Q  \bfu   \cdot\bm\zeta_t + \frac{\phi'(|D\bfu|)}{|D\bfu|} \langle D\bfu , D \bm\zeta \rangle\,  \d z\Big|
    \leq \delta \Big[ \fint_Q \phi(|D\bfu|) +\phi^*(|{\bf G}|) \, \dz+\phi( \|D\bm\zeta\|_\infty)
    \Big] \,,
 $$
 then there exists a {$\phi$}-caloric function {$\bfh$} such that {$\bfh=\bfu$} on {$\partial_{\mathrm p} Q$} and
{ \begin{equation} \begin{aligned}
    \left(\fint_I\Big(\fint_B \Big({|\bfu-\bfh|^{2}\over |t^+-t^-|}\Big)^{\gamma_2}\d x\Big)^{\frac{\gamma_3}{ \gamma_2}}\d t\right)^{\frac{1}{\gamma_3}}  +\Big(\fint_Q |{\bf V}(D\bfu )- &{\bf V}(D\bfh)|^{2\gamma_1} \dz
    \Big)^{\frac{1}{\gamma_1}} \\
    &\leq \epsilon \fint_Q
  \phi( |D\bfu|)+\phi^*(|{\bf G}|)\, \dz  \,.\nonumber
\end{aligned}  \end{equation}}

\end{theorem}
\noindent

Note that a first version of the  $p$-caloric approximation method was developed by B\"ogelein-Duzaar-Mingione~\cite{BoDuzMin13} by using a contradiction argument. They used it to show a partial regularity result for  solutions of parabolic systems of $p$-growth; that is,  almost everywhere $\nabla u\in C^{\alpha}$ for some $\alpha>0$. We wish to quickly point out the improvements of the approximation lemma here with respect to the one  in~\cite{BoDuzMin13}. The proof is done directly by a comparison argument and the parabolic Lipschitz truncation. This direct approach allows for showing the closeness both in $L^{2\gamma_2}(L^{2\gamma_3})$ and $L^{\phi^{\gamma_1}}(W^{1,{\phi^{\gamma_1}}})$ norms (the last closeness is via the function ${\bf V}$ in \eqref{Vfunction}). 
%${\bf V}({\bf z})=|{\bf z}|^{\frac{p-2}{2}} {\bf z} $).  }

%\newpage

\section{Caccioppoli type inequality and Higher integrability}
\label{sec3}

Let $\bfu$ be a weak solution to \eqref{system1}. We always assume that $\phi$ and $\bA$ satisfy Assumption~\ref{Ass}. Let  ${\bm \ell}: \R^n \to \R^N$ be any fixed linear map of the form
\begin{equation}\label{def_l}
{\bm \ell}(x) : = \bP (x-x_0) + \bb ,
\quad x\in \R^n,
\end{equation}
where $\bP \in \R^{Nn}$, $x_0\in \R^n$ and $\bb \in \R^N$, and set
\begin{equation}\label{def_ul}
\bfu_{\bm \ell}:= \bfu-{\bm \ell}\,.
%\quad \text{and} \quad
%\phi_L(\tau) := \phi_{|D{\bf L}|}(\tau),
\end{equation}
%that is, $\phi_L$ is the shifted $N$-function $\phi_a$ with $a=|DL|$. 
%\comment{\textcolor{blue}{Giovanni: the notation $\phi_L$ is misleading, since $\phi_L$ denotes the shifted function with shift $L$, according to the standard notation! I propose to use $\phi_{|DL|}$.}}
In this section we will obtain the higher integrability of not only $D\bfu$ but also of $D\bfu_{\ell}$. We follow the argument in \cite{HasOk21}.

We first recall a Gagliardo-Nirenberg type inequality for Orlicz functions, see Lemma \ref{Lem:hok}, which has been proved in \cite[Lemma~2.13]{HasOk21}. In order to do that, we fix some notation.
A function $\varphi:[0,\infty)\to[0,\infty)$ is said to be a \emph{weak $\Phi$-function} if it is increasing with $\varphi(0)=0$, $\lim_{t\to0^+}\varphi(t)=0$, $\lim_{t\to +\infty}\varphi(t)=+\infty$ and such that the map $t\to \frac{\varphi(t)}{t}$ is almost increasing. Note that every $N$-function is a weak $\Phi$-function.

\begin{lemma}
\label{Lem:hok}
Assume that $\psi:[0,\infty)\to[0,\infty)$ is a weak $\Phi$-function and such that $t\mapsto \frac{\psi(t)}{t^{q_1}}$ is almost decreasing with constant $L\ge 1$ for some $q_1\geq 1$. For $p\in [1,n)$ and $q_2>0$ we have
\begin{equation*}
\bigg(\fint_{B_{r}}\psi\big(\big|\tfrac{f}{r}\big|\big)^\gamma\, \d x\bigg)^{\frac1\gamma}
\le c \bigg(\fint_{B_r}\left[\psi(|Df|)^p+\psi\big(\big|\tfrac{f}{r}\big|\big)^p\right]\,\d x\bigg)^{\frac{\theta}{p}}
\psi\bigg(\Big(\fint_{B_r}\big|\tfrac{f}{r}\big|^{q_2}\,\d x\Big)^\frac1{q_2}\bigg)^{1-\theta}
\end{equation*}
for some $c=c(n,L,q_1,q_2)>0$, provided that $\theta\in (0,1)$ and $\gamma$ satisfies
\begin{equation*}
\frac{1}{\gamma}\geq \frac\theta{p^*}+\frac{(1-\theta)q_1}{q_2}\,.
\end{equation*}
\end{lemma}

 We start with a Caccioppoli type inequality for $\bfu_{\bm \ell}$. 

%\begin{lemma}[Caccioppoli inequality]\label{Lem:caccio0}
%Let $\bfu$ be a weak solution to \eqref{system1}. For every
%$Q_{r_1,\tau_1}(z_0) \subset Q_{r_2,\tau_2}(z_0) \Subset\Omega_T$ with $0<r_1<r_2$ and $0<\tau_1<\tau_2$,  and $\bb\in \R^N$, we have 
%\begin{equation}\label{eq:caccio0}\begin{aligned}
%\sup_{t\in I_{\tau_1}}\int_{B_{r_1}}  |\bfu(t)-\bb|^2\,\d x & + \int_{Q_{t_1,\tau_1}} \phi(|D\bfu|)\, \d z\\
%& \leq c\int_{Q_{r_2,\tau_2}}\left[\frac{|\bfu-\bb|^2}{t_2-t_1}+ \phi\Big(\Big|\frac{\bfu-\bb}{r_2-r_1}\Big|\Big)\right]\, \d z
%\end{aligned}\end{equation}
%for some $c=c(n,N,p,q,L,\nu)>0$, where $\bfu(t)=\bfu(x,t)$. 
%\end{lemma}

\begin{lemma}[Caccioppoli inequality for $\bfu_{\bm \ell}$]\label{Lem:caccio}
Let $\bfu$ be a weak solution to \eqref{system1}. For every pair of concentric cylinders
$Q_{r_1,\tau_1}(z_0) \subset Q_{r_2,\tau_2}(z_0) \Subset\Omega_T$ with $z_0=(x_0,t_0)$, $0<r_1<r_2$ and $0<\tau_1<\tau_2$,  and $\bb\in \R^N$, we have 
\begin{equation}\label{eq:caccio}\begin{aligned}
\sup_{t\in I_{\tau_1}(t_0)}\int_{B_{r_1}(x_0)} &  |\bfu_{\bm \ell}(t)-\bb|^2\,\d x  + \int_{Q_{r_1,\tau_1}(z_0)} \phi_{|D{\bm \ell}|}(|D\bfu_{\bm \ell}|)\, \d z\\
& \leq c\int_{Q_{r_2,\tau_2}(z_0)}\left[\frac{|\bfu_{\bm \ell}-\bb|^2}{\tau_2-\tau_1}+ \phi_{|D{\bm \ell}|}\Big(\Big|\frac{\bfu_{\bm \ell}-\bb}{r_2-r_1}\Big|\Big)\right]\, \d z
\end{aligned}\end{equation}
for some $c=c(n,N,p,q,L,\nu)>0$, where $\bfu(t)=\bfu(x,t)$ and $I_{\tau}(t_0)=(t_0-\tau,t_0+\tau)$.
\end{lemma}

\begin{proof}
We assume without loss of generality that 
the center of  $Q_{r_1,\tau_1}$ and $Q_{r_2,\tau_2}$ is the origin. 
%$\Q_R^\lambda$ is centered at the origin $(0,0)$. 
Let $\xi\in C^1_0(B_R)$ with $\xi\equiv 1$ in $B_r$  and $|D\xi|\leq 2
/(r_2-r_1)$ and $\eta\in C^1(\mr)$ with $\eta\equiv 0$ in $(-\infty,-\tau_2]$, $\eta\equiv 1$ in $[-\tau_1,\infty)$ and $0\leq \eta'\leq 2/(\tau_2-\tau_1)$.
Using
\[
\bm\zeta(x,t):=\xi(x)^{q}\eta(t)^2 (\bfu_{\bm \ell}(x,t)-\bb)
\]
as a test function in \eqref{system1}, we have that for $s\in I_{\tau_1}$,
$$
\int_{-\tau_2}^s\int_{B_{r_2}}\left[\partial_t \bfu \cdot  \bm\zeta+\bA(D\bfu) : D\bm\zeta\right]\, \d x\, \d t\, = 0\,.
$$
Moreover, since $\partial_t \bb = \partial_t \bm\ell= \div \bA(D\bm\ell)=\zero$, we further have 
$$
\int_{-\tau_2}^s\int_{B_{r_2}}\left[\partial_t (\bfu_{\bm\ell}-\bb) \cdot  \bm\zeta+(\bA(D\bfu)-\bA(D\bm\ell)) : D\bm\zeta\right]\, \d x\, \d t\,  = 0\,.
$$
Note that
\[\begin{split}
&\int_{-\tau_2}^s\int_{B_{r_2}}\partial_t \bfu \cdot \bm\zeta \, \d x\,\d t
= \int_{-\tau_2}^s\int_{B_{r_2}}\frac12\partial_t [\xi^{q}\eta^2 |\bfu_{\bm\ell}-\bb|^2] - \xi^{q}\eta \eta' |\bfu_{\bm\ell}-\bb|^2 \, \d x\,\d t \\
&= \frac12 \int_{B_{r_2}} \xi^{q}\eta(s)^2 |\bfu_{\bm\ell}(s)-\bb|^2\,\d x - \int_{-\tau_2}^s\int_{B_{r_2}} \xi^{q}\eta\eta' |\bfu_{\bm\ell}-\bb|^2 \, \d x\,\d t,
\end{split}\]
where $\bfu_{\bm\ell}(s)=\bfu_{\bm\ell}(x,s)$. 
Then, applying \eqref{growth}, \eqref{monotonicity}, \eqref{monotonicity1} and \eqref{eq:(3.4)} we have  that for every $s\in I_{\tau_1}$,
\[\begin{split}
&\frac12 \int_{B_{r_2}} \xi^{q}\eta(s)^2 |\bfu_{\bm\ell}(s)-\bb|^2\,\d x + \frac{1}{c}\int_\sigma^s\int_{B_{r_2}} \xi^{q}\eta^2 \phi_{|D{\bm \ell}|}(|D\bfu_{\bm \ell})\, \d x\,\d t\\
& \leq    -q\int_{-\tau_2}^s\int_{B_{r_2}}\xi^{q-1}\eta^2
(\bA (D\bfu)-\bA(D{\bm \ell})) : D\xi \otimes (\bfu_{\bm \ell}-\bb)\, \d x\,\d t \\
&\,\,\,\,\,\, +\int_{-\tau_2}^s\int_{B_{r_2}} \xi^{q}\eta\eta' |\bfu_{\bm \ell}-\bb|^2 \, \d x\,\d t\\
& \leq  c \int_{-\tau_2}^s\int_{B_{r_2}}\xi^{q-1} \eta^2 
\phi'_{|D{\bm \ell}|}(|D\bfu_{\bm \ell}|) \Big|\frac{\bfu_{\bm\ell}-\bb}{r_2-r_1}\Big| \, \d x\,\d t 
+\int_{-\tau_2}^s\int_{B_{r_2}} \frac{|\bfu_{\bm\ell}-\bb|^2}{\tau_2-\tau_1}\, \d x\,\d t.
\end{split}
\]
Finally, applying Young's inequality \eqref{eq:young} with \eqref{eq:hok2.4} and using the properties the cut-off functions $\xi$ and $\eta$ we have the estimate \eqref{eq:caccio}.
\end{proof}

Considering the intrinsic cylinders with the $N$-functions $\phi$ and $\phi_{|D{\bm \ell}|}$, defined by 
$$
Q^\lambda_r(z_0):= B_r(x_0) \times I^\lambda_r(t_0), \quad \text{where }\  I^\lambda_r(t_0):=\left(t_0-\frac{\lambda^2}{\phi(\lambda)}r^2,t_0+\frac{\lambda^2}{\phi(\lambda)}r^2\right),
$$ 
 and 
$$
\Q^\lambda_r(z_0):= B_r(x_0) \times \I^\lambda_r(t_0), \quad \text{where }\  \I^\lambda_r(t_0):=\left(t_0-\frac{\lambda^2}{\phi_{|D{\bm \ell}|}(\lambda)}r^2,t_0+\frac{\lambda^2}{\phi_{|D{\bm \ell}|}(\lambda)}r^2\right),
$$ 
%$$
%\Q^\lambda_r(z_0):= B_r(x_0) \times \I^\lambda_r(t_0), \quad \text{where }\  \I^\lambda_r(t_0):=\left(t_0-\frac{r^2}{\phi_{|D{\bm \ell}|}''(\lambda)},t_0+\frac{r^2}{\phi_{|D{\bm \ell}|}''(\lambda)}\right),
%$$ 
respectively, we obtain the following Caccioppoli-type estimates.

\begin{corollary}\label{cor:caccio}
Let $\bfu$ be a weak solution to \eqref{system1}. For every 
$Q^\lambda_R\Subset\Omega_T$ or $\Q^\lambda_R\Subset\Omega_T$ with $\lambda,R>0$ and $r\in(0,R)$ and $\bb\in \R^N$, we have 
\begin{equation}\label{eq:caccio1}\begin{aligned}
\sup_{t\in I^\lambda_r}\int_{B_r}  |\bfu_{\bm\ell}(t)-\bb|^2\,\d x & + \int_{Q_r^\lambda} \phi_{|D{\bm \ell}|}(|D\bfu_{\bm\ell}|)\, \d z\\
& \leq c\int_{Q^\lambda_R}\left[\frac{\phi(\lambda)}{\lambda^2}\Big|\frac{\bfu_{\bm\ell}-\bb}{R-r}\Big|^2+ \phi_{|D{\bm \ell}|}\Big(\Big|\frac{\bfu_{\bm\ell}-\bb}{R-r}\Big|\Big)\right]\, \d z
\end{aligned}\end{equation}
and
\begin{equation}\label{eq:caccio2}\begin{aligned}
\sup_{t\in \I^\lambda_r}\int_{B_r}  |\bfu_{\bm\ell}(t)-\bb|^2\,\d x & + \int_{\Q_r^\lambda} \phi_{|D{\bm \ell}|}(|D\bfu_{\bm\ell}|)\, \d z\\
& \leq c\int_{\Q^\lambda_R}\left[\frac{\phi_{|D{\bm \ell}|}(\lambda)}{\lambda^2}\Big|\frac{\bfu_{\bm\ell}-\bb}{R-r}\Big|^2+ \phi_{|D{\bm \ell}|}\Big(\Big|\frac{\bfu_{\bm\ell}-\bb}{R-r}\Big|\Big)\right]\, \d z
\end{aligned}\end{equation}
for some $c=c(n,N,p,q,L,\nu)>0$, where $\bfu(t)=\bfu(x,t)$. 
\end{corollary}

In the remaining part of this section, we obtain higher integrability estimates for  $\phi_{|D{\bm \ell}|}(|D\bfu_{\bm \ell}|)$. Note that the higher integrability of  $\phi(|D\bfu|)$ (i.e., the case $\bm\ell= \zero$) is proved in  \cite{HasOk21}. We then follow the argument therein.

%In the remaining part of the section, we consider the intrinsic cylinders with the $N$-function $\phi_{|D{\bm \ell}|}$, {\color{blue}and follow the argument in \cite[Section 3]{HasOk21}.}

As a first key tool, we introduce a Sobolev--Poincar\'e type inequality, Lemma \ref{Lem:poin1}. 
%The result in the next lemma and its proof are almost the same as the ones in 
The proof of \eqref{eq:complicatedpoincare} could be obtained with minor modifications as in \cite[Lemma 3.4]{HasOk21}, just replacing  $\bfu$ and $\phi$ by $\bfu_{\bm\ell}$  and $\phi_{|D\bm\ell|}$, respectively, and modifying the estimate for $|\langle \bfu_{\bm\ell} \rangle_\xi(t) - (\bfu_{\bm\ell})_\rho^\lambda|$ (see \eqref{uaverageueta} below) according to assumption \eqref{eq:(3.4)}. However, for the reader's convenience, we prefer to provide a detailed proof. %Note that the main differences are replacing  $\bfu$ and $\phi$ by $\bfu_{\bm\ell}$  and $\phi_{|D\bm\ell|}$, respectively, and modifying the estimates in \eqref{uaverageueta} according to assumption . 
Note that the simplified version of the Poincar\'e inequality in \eqref{cor:poin1a} can be also found in \cite[Lemma 2.9]{DieSchStrVer17}.

Let $\xi\in C_0^\infty(B_{\rho})$  satisfy $0\leq \xi \leq 1$, $\xi\equiv 1$ in $B_{\rho/2}$, $|D\xi|\leq \frac 4\rho$.  Note that $2^{-n}|B_\rho| \le \|\xi\|_{1}\le |B_\rho|$. Define
$$
(f)_{\rho}^\lambda:= \frac{1}{\|\xi\|_1} \fint_{\I_\rho^\lambda}\int_{B_\rho} f \xi \, \d x\,\d t
\quad\text{and}\quad
\langle f \rangle_\xi(t):= \frac{1}{\|\xi\|_1}\int_{B_\rho}f(x,t)\xi \, \d x \ \ \text{for }t\in \I^\lambda_\rho.
$$

\begin{lemma} \label{Lem:poin1}
Let $\bfu$ be a weak solution to \eqref{system1}. For an N-function  $\psi$ satisfying \eqref{characteristic} with $1\le p_1\le q_1$ in place of $1< p \le q$ , $\Q^\lambda_{4\rho}\Subset\Omega_T$ with $\lambda>0$ and $\rho\le r<R\le  4\rho$, we have
\begin{equation}
\begin{split}
\fint_{\Q^\lambda_{r}}\psi\bigg(\bigg|\frac{\bfu_{\bm\ell}-(\bfu_{\bm\ell})^\lambda_\rho}{r}\bigg|\bigg)\,\d z
&\le
c \psi (A_0) + 
c \psi\big(\mathcal{T}(r,R)^\frac12\big)^{(1-\theta_0)} \fint_{\Q_r^\lambda}\psi(|D\bfu_{\bm\ell}|)^{\theta_0} \,\d z, 
\end{split}
\label{eq:complicatedpoincare}
\end{equation}
for some $c=c(n,N,p,q,p_1,q_1,\theta_0,L,\nu,\Lambda)>0$ provided that 
\[
\theta_0 p_1\in[1,n)
\quad\text{and}\quad
\frac{ n q_1}{n q_1 + 2p_1}\le \theta_0 \le 1.
\]
Here $(\bfu_{\bm\ell})^\lambda_\rho$ it the average of $\bfu$ on  $\Q^\lambda_\rho$,
\begin{equation}\label{A0}
A_0:= \frac{\lambda^2}{\phi_{|D{\bm \ell}|}(\lambda)}\fint_{\Q^\lambda_r}\phi_{|D{\bm \ell}|}'(|D\bfu_{\bm\ell}|)\,\d z,
\end{equation}
\begin{equation*}
\mathcal{T}(r,R)
:= 
\fint_{\Q_{R}^\lambda}\bigg[ \bigg|\frac{\bfu_{\bm\ell}-(\bfu_{\bm\ell})_\rho^\lambda}{R-r}\bigg|^2 + 
\frac{\lambda^2}{\phi_{|D{\bm \ell}|}(\lambda)}\phi_{|D{\bm \ell}|}\bigg(\bigg|\frac{\bfu_{\bm\ell}-(\bfu_{\bm\ell})_\rho^\lambda}{R-r}\bigg|\bigg)\bigg]\, \d z + 
A_0^2.
\label{eq:termT}
\end{equation*}
In particular, when $\theta_0=p_1=1$, we have
\begin{equation}\label{cor:poin1a}
\fint_{\Q^\lambda_{r}}\psi\bigg(\bigg|\frac{\bfu_{\bm\ell}-(\bfu_{\bm\ell})^\lambda_\rho}{r}\bigg|\bigg)\, \d z
\le
c\fint_{\Q^\lambda_r}\psi(|D\bfu_{\bm\ell}|)\,\d z + c \psi (A_0).
\end{equation}
\end{lemma}
\begin{proof}
The triangle inequality implies  
\begin{equation}\label{Lem:poin1:pf1}\begin{split}
\fint_{\Q_r^\lambda}\psi\bigg(\bigg|\frac{\bfu_{\bm\ell}-(\bfu_{\bm\ell})_\rho^\lambda}{r}\bigg|\bigg) \, \d z
&=
\fint_{\Q_r^\lambda}\psi\bigg(\bigg|\frac{\bfu_{\bm\ell}(z)-\langle \bfu_{\bm\ell}\rangle_\xi(t)
+\langle \bfu_{\bm\ell}\rangle_\xi(t)-(\bfu_{\bm\ell})_\rho^\lambda}{r}\bigg|\bigg) \,\d z \\
&\hspace{-3cm}\le
c\fint_{\I_r^\lambda} \psi\bigg(\bigg| \frac{\langle \bfu_{\bm\ell}\rangle_{\xi}(t)-(\bfu_{\bm\ell})_\rho^\lambda}r \bigg|\bigg) \,\d t
+
c\fint_{\Q_r^\lambda}\psi\bigg(\bigg|\frac{\bfu_{\bm\ell}(z)-\langle \bfu_{\bm\ell}\rangle_{\xi}(t)}{r}\bigg|\bigg) \,\d z.  
\end{split}\end{equation}
We start with an estimate of the first term in the right hand side above.  %We first take care of the first term. 
By the definition of $\langle \bfu_{\bm \ell}\rangle_\xi$ and using the weak formulation of \eqref{system1} with test-function $\bm\zeta(x,t):=(\xi(x), \dots, \xi(x))$, we find from \eqref{eq:(3.4)} that 
\begin{equation}\label{uaverageueta}
\begin{split}
|\langle \bfu_{\bm\ell} \rangle_\xi(t) - (\bfu_{\bm\ell})_\rho^\lambda|
&\le
\sup_{\tau\in  \I_r^\lambda } |\langle \bfu_{\bm\ell}\rangle_\xi(t)-\langle \bfu_{\bm\ell}\rangle_\xi(\tau)| =
 \sup_{\tau\in  \I_r^\lambda }\bigg|\int^{t}_{\tau} \partial_t \langle \bfu_{\bm\ell}\rangle_\xi(s)\, \d s\bigg|  \\
&
= \sup_{\tau\in  \I_r^\lambda }\bigg|\int^{t}_{\tau}\frac{1}{\|\eta\|_1}\int_{B_r}\partial_t \bfu_{\bm\ell}(x,s)\xi(x) \, \d x\, \d s\bigg|\\
&
= \sup_{\tau\in  \I_r^\lambda }\bigg|\int^{t}_{\tau}\frac{1}{\|\eta\|_1}\int_{B_r}\partial_t \bfu(x,s)\xi(x) \, \d x\, \d s\bigg|\\
&\approx \sup_{\tau\in  \I_r^\lambda } \bigg|\int^{t}_{\tau}\fint_{B_r} (\bA(D\bfu)-\bA(D\bm\ell) )D\xi\,\d x\,\d s\bigg|\\
&\leq \frac{c r\lambda^2}{\phi_{|D{\bm \ell}|}(\lambda)}\fint_{\Q_r^\lambda} \phi'_{|D{\bm \ell}|} (|D\bfu_{\bm\ell}|)\,\d z = r A_0.
\end{split}
\end{equation}
%This gives the first term on the right-hand side of the claim. 

%We next use the Gagliardo--Nirenberg inequality in Lemma~\ref{Lem:hok} with 
%$(\psi,\gamma,p,q_1,q_2)$ given by $(\psi^{1/p_1}, p_1,\theta_0 p_1,\frac {q_1}{p_1},2)$
%to conclude that  
%\[
%\fint_{B_r}\psi\big(\big|\tfrac{f}{r}\big|\big)\, \d x
%\le%sssim
%c\left(\fint_{B_r}\Big[\psi(|Df|)^{\theta_1}+\psi\big(\big|\tfrac{f}{r}\big|\big)^{\theta_0}\Big]\,\d x\right)
%\ \psi\Bigg(\bigg[\fint_{B_r}\big|\tfrac{f}{r}\big|^2\,\d x\bigg]^\frac12\Bigg)^{1-\theta_0}
%\]
%provided $\theta_0 p_1\in[1,n)$ and
%\[
%\frac1{p_1} \ge 
%\frac{\theta_0}{(\theta_0 p_1)^*} + \frac{1-\theta_0}2 \frac{q_1}{p_1}
%=
%\frac{1}{p_1} - \frac{\theta_0}{n} + \frac{1-\theta_0}2 \frac{q_1}{p_1}.
%\]
%This can be written as $\theta_0\ge \frac{n q_1}{n q_1 + 2p_1}$.
%%Furthermore, we require that $\theta_0 \le 1$, since the estimate is useless otherwise.

We next estimate the second term in \eqref{Lem:poin1:pf1}. From the Gagliardo--Nirenberg type inequality in Lemma \ref{Lem:hok} with 
$(\psi,\gamma,p,q_1,q_2):=(\psi^{1/p_1}, p_1,\theta_0 p_1,\frac {q_1}{p_1},2)$
we conclude that %\cite[Lemma 2.13]{HasOk21} 
%$$
\begin{equation}\label{GNinequ}
\fint_{B_r}\psi\big(\big|\tfrac{f}{r}\big|\big)\, \d x
\le%sssim
c\left(\fint_{B_r}\Big[\psi(|Df|)^{\theta_1}+\psi\big(\big|\tfrac{f}{r}\big|\big)^{\theta_0}\Big]\,\d x\right)
\ \psi\Bigg(\bigg[\fint_{B_r}\big|\tfrac{f}{r}\big|^2\,\d x\bigg]^\frac12\Bigg)^{1-\theta_0}
\end{equation}
%$$
provided $\theta_0 p_1\in[1,n)$ and
\[
\frac1{p_1} \ge 
\frac{\theta_0}{(\theta_0 p_1)^*} + \frac{1-\theta_0}2 \frac{q_1}{p_1}
=
\frac{1}{p_1} - \frac{\theta_0}{n} + \frac{1-\theta_0}2 \frac{q_1}{p_1}.
\]
This can be written as $\theta_0\ge \frac{n q_1}{n q_1 + 2p_1}$.
Applying \eqref{GNinequ} with $f:=\bfu_{\bm\ell}-\langle \bfu_{\bm\ell}\rangle_\xi$ on each time slice gives 
\begin{equation}\label{Lem:poin1:pf2}\begin{split}
\fint_{\Q_r^\lambda}\psi\Big(\Big|\frac{\bfu_{\bm\ell}(z)-\langle \bfu_{\bm\ell} \rangle_\xi (t)}{r}\Big|\Big) \,\d z &\le c
\left(\fint_{\Q_r^\lambda}\left[\psi(|D\bfu_{\bm\ell}|)^{\theta_0} + \psi\left(\frac{\bfu_{\bm\ell}-\langle \bfu_{\bm\ell}\rangle_\xi}{r}\right)^{\theta_0}\right]\,\d z \right)\\
&\qquad \times  \psi\Bigg(\bigg( \sup_{t\in \I_r^\lambda} \fint_{B_r}
\left|\frac{\bfu_{\bm\ell}-\langle \bfu_{\bm\ell}\rangle_\xi}{r}\right|^2\,\d x\bigg)^\frac12\Bigg)^{1-\theta_0}.
\end{split}\end{equation}
Note that for each time slice of $\Q_r^\lambda$, since $\theta_0 p_1\ge 1$, we can apply the weighted Poincar\'e inequality \cite[Theorem 7]{DieEtt08}, so that
$$
\fint_{B_r}\psi\Big(\Big|\frac{\bfu_{\bm\ell}(x,t)-\langle \bfu_{\bm\ell}\rangle_{\xi}(t)}{r}\Big|\Big)^{\theta_0} \,\d x 
 \leq c \fint_{B_r}\psi(|D\bfu_{\bm\ell}(x,t)|)^{\theta_0} \,\d x\,.
$$
Finally, from the Caccioppoli inequality in \eqref{eq:caccio2} with $\bb := (\bfu_{\bm\ell})^\lambda_\rho$
and \eqref{uaverageueta} we conclude that 
$$\begin{aligned}
&\sup_{t\in \I_r^\lambda} \fint_{B_r}\left|\frac{\bfu_{\bm\ell}-\langle \bfu_{\bm\ell}\rangle_\xi}{r}\right|^2 \,\d x\\
&\le c \sup_{t\in \I_r^\lambda} \fint_{B_r}\bigg|\frac{\bfu_{\bm\ell}(x,t)-(\bfu_{\bm\ell})_\rho^\lambda}{r}\bigg|^2\,\d x
+ c \sup_{t\in \I_r^\lambda} \bigg|\frac{(\bfu_{\bm\ell})_\rho^\lambda-\langle \bfu_{\bm\ell}\rangle_{\xi}(t)}{r}\bigg|^2\\
&\le c
\fint_{\Q_{R}^\lambda}\bigg[ \bigg|\frac{\bfu_{\bm\ell}-(\bfu_{\bm\ell})_\rho^\lambda}{R-r}\bigg|^2 + 
\frac{\lambda^2}{\phi_{|D{\bm \ell}|}(\lambda)}\phi_{|D{\bm \ell}|}\bigg(\bigg|\frac{\bfu_{\bm\ell}-(\bfu_{\bm\ell})_\rho^\lambda}{R-r}\bigg|\bigg)\bigg]\, \d z + 
cA_0^2\,. %\qedhere
\end{aligned}$$
Therefore, inserting the above two estimates into  \eqref{Lem:poin1:pf2} and combining with \eqref{Lem:poin1:pf1}--\eqref{uaverageueta}, we complete the proof of \eqref{eq:complicatedpoincare}. 
\end{proof}

%If we choose $\theta_0=p_1=1$ in the previous lemma, we obtain the following Poincar\'e inequality, which is also can be obtained from \cite[]{}. 

%\begin{corollary}[Poincar\'e inequality] \label{cor:poin1a}
%Let $\bfu$ be a weak solution to \eqref{system1}. For every $N$-function $\psi$ and $\Q^\lambda_r\Subset\Omega_T$ with $\lambda>0$,
% we have
%\[
%\fint_{\Q^\lambda_{r}}\psi\bigg(\bigg|\frac{\bfu_{\bm\ell}-(\bfu_{\bm\ell})^\lambda_\rho}{r}\bigg|\bigg)\, \d z
%\le
%c\fint_{\Q^\lambda_r}\psi(|D\bfu_{\bm\ell}|)\,\d z + c \psi (A_0).
%\]
%\end{corollary}

The next two lemmas show that the right hand side of the estimate in Lemma~\ref{Lem:poin1}
can be controlled by suitable quantities when we are in suitable intrinsic cylinders.

%\textcolor{blue}{With the following two results we will show how the additional terms appearing in the estimate of Lemma~\ref{Lem:poin1} %Over the course of the next two results we will show how the extra terms in 
%%the previous lemma 
%can be controlled by suitable quantities when we are in suitable 
%intrinsic cylinders.}

\begin{lemma}\label{Lem:2ndTerm}
Let the assumptions of Lemma~\ref{Lem:poin1} be in force, and assume additionally that 
$$
\fint_{\Q^\lambda_{4\rho}}\phi_{|D{\bm \ell}|}(|D\bfu_{\bm\ell}|)\, \d z\leq  \phi_{|D{\bm \ell}|}(\lambda) \,.
$$
Then,
%we have that for any $\rho\leq r <R \leq 4\rho$,
%\[
%T(r,R)\le c\big(\tfrac{R}{R-r}\big)^{\max\{q,\frac{nq}{n+2}+2,\frac{4}{n}+2\}} \lambda^2
%\]
for some $c=c(n,N,p,q,p_1,q_1,\theta_0,L,\nu,\Lambda)>0$, 
\[
\begin{split}
\fint_{\Q^\lambda_{2\rho}}\psi\bigg(\bigg|\frac{\bfu_{\bm\ell}-(\bfu_{\bm\ell})^\lambda_\rho}{\rho}\bigg|\bigg)\,\d z
&\le
c \psi (A_0)
+
c \psi(\lambda)^{1-\theta_0} \fint_{Q_{2\rho}^\lambda}\psi(|D\bfu_{\bm\ell}|)^{\theta_0} \,\d z,
\end{split}
\]
where $A_0$ is that of \eqref{A0} with $r=2\rho$.
\end{lemma}
\begin{proof}
The proof is exactly the same as the one of \cite[Lemma 3.9]{HasOk21} with $\bfu_{\bm\ell}$ and $\phi_{|D{\bm \ell}|}$ in place of $\bfu$ and $\phi$, respectively.
\end{proof}

Finally, we obtain a reverse H\"older inequality for $\phi_{|D{\bm \ell}|}(|D\bfu_{\bm\ell}|)$. The proof is almost the same as that of \cite[Lemma 3.12]{HasOk21}. The main difference is the use of the Caccioppoli estimate  \eqref{eq:caccio1} in place of the usual one \cite[Lemma 3.1]{HasOk21}.

\begin{lemma}%\label{Lem:reverse}
Let $\bfu$ be a weak solution to \eqref{system1} and 
$Q^\lambda_{4\rho}\Subset\Omega_I$ with $\lambda,\rho>0$. Suppose that 
\begin{equation}\label{lemreverseass}
\phi_{|D{\bm \ell}|}(\lambda) \leq \fint_{\Q_\rho^\lambda}\phi_{|D{\bm \ell}|}(|D\bfu_{\bm\ell}|)\, \d z
\quad\text{and}\quad
\fint_{\Q_{4\rho}^\lambda}\phi_{|D{\bm \ell}|}(|D\bfu_{\bm\ell}|)\, \d z\leq \phi_{|D{\bm \ell}|}(\lambda).
\end{equation}
Then there exist $\theta=\theta(n,p,q)\in(0,1)$ and $c=c(n,N,p,q,L,\nu,\Lambda)>0$ such that
\begin{equation}\label{reverseestimate}
\fint_{\Q^\lambda_\rho}
\phi_{|D{\bm \ell}|}(|D\bfu_{\bm\ell}|)\,\d z\leq c\bigg(\fint_{\Q^\lambda_{4\rho}}\phi_{|D{\bm \ell}|}(|D\bfu_{\bm\ell}|)^{\theta}\,\d z\bigg)^{\frac{1}{\theta}}.
\end{equation}
\end{lemma}
\begin{proof} 
We denote $p_0 := \frac{2n}{n+2}$, and recall $A_0$ in \eqref{A0}. Arguing as in \cite[Lemma 2.9]{HasOk21} we have that for every $\delta\in(0,1)$ and $\theta_0\in (1-\frac{1}{q},1]$,
\begin{equation}\label{A0estimate}
A_0\le 
\begin{cases}
\delta\lambda + c_\delta \phi_{|D{\bm \ell}|}^{-1}\left(\Big(\fint_{\Q^\lambda_{2\rho}}\phi_{|D{\bm \ell}|}(|D\bfu_{\bm\ell}|)^{\theta_0}\,\d z\Big)^{\frac{1}{\theta_0}}\right),\\
c\lambda \,.
\end{cases}
\end{equation}
By the Caccioppoli inequality \eqref{eq:caccio1} with $\bb:=(\bfu_{\bm \ell})_\rho^\lambda$, 
we find that 
\begin{equation}\label{22}
%\begin{split}
\fint_{\Q_\rho^\lambda} \phi_{|D{\bm \ell}|}(|D\bfu_{\bm\ell}|)\, \d z 
\le%sssim
c \frac{\phi_{|D{\bm \ell}|}(\lambda)}{\lambda^2} \fint_{\Q^\lambda_{2\rho}}\bigg|\frac{\bfu_{\bm\ell}-(\bfu_{\bm\ell})_{\rho}^\lambda}{\rho}\bigg|^2\,\d z 
+ c \fint_{\Q^\lambda_{2\rho}} \phi_{|D{\bm \ell}|}\bigg(\bigg|\frac{\bfu_{\bm\ell}-(\bfu_{\bm\ell})_{\rho}^\lambda}{\rho}\bigg|\bigg)\, \d z .
%\end{split}
\end{equation}
We then estimate the two integrals in the right hand side. 

By Lemma~\ref{Lem:2ndTerm} for $\psi:=\phi_{|D{\bm \ell}|}$, considering also \eqref{A0estimate} and the classical Young's inequality for conjugate exponents $\frac{1}{\theta_0}, \frac{1}{1-\theta_0}$, we have that for any $\delta\in(0,1)$
\begin{equation}\label{22-1}
\begin{aligned}
\fint_{Q^\lambda_{2\rho}} \phi_{|D{\bm \ell}|}\bigg(\bigg|\frac{\bfu_{\bm\ell}-(\bfu_{\bm\ell})_{\rho}^\lambda}{\rho}\bigg|\bigg)\, \d z 
& \le
c \phi_{|D{\bm \ell}|}(A_0) + c \phi_{|D{\bm \ell}|}(\lambda)^{(1-\theta_0)} \fint_{\Q^\lambda_{2\rho}}\phi_{|D{\bm \ell}|}(|D\bfu_{\bm\ell}|)^{\theta_0}\,\d z \\
& \le 
c_{\delta} \left(%\fint_{\Q^\lambda_{2\rho}}
\fint_{\Q^\lambda_{2\rho}}\phi_{|D{\bm \ell}|}(|D\bfu_{\bm\ell}|)^{\theta_0}\,\d z\right)^{\frac{1}{\theta_0}} + c\delta \phi_{|D{\bm \ell}|}(\lambda)\,. 
\end{aligned}
\end{equation}
An analogous argument in the case $\psi(t):=t^2$ and $\theta_0:=\frac{p_0}2$, 
shows that for any $\delta\in(0,1)$
\[\begin{split}
\bigg(\fint_{\Q^\lambda_{2\rho}}\bigg|\frac{\bfu_{\bm\ell}-(\bfu_{\bm\ell})_{\rho}^\lambda}{\rho}\bigg|^2\,\d z\bigg)^{\frac{1}{2}}
&\le%sssim 
c A_0
+ c \bigg(\lambda^{2-p_0} \fint_{\Q^\lambda_{2\rho}}|D\bfu_{\bm\ell}|^{p_0} \, \d z\bigg)^{\frac{1}{2}}\\
&\le c_\delta \bigg(\fint_{\Q^\lambda_{2\rho}}|D\bfu_{\bm\ell}|^{p_0}\, \d z\bigg)^\frac1{p_0} +c A_0 + \delta \lambda.
\end{split}\]
In particular, we also have 
\[
\bigg(\fint_{\Q^\lambda_{2\rho}}\bigg|\frac{\bfu_{\bm\ell}-(\bfu_{\bm\ell})_{\rho}^\lambda}{\rho}\bigg|^2\,\d z\bigg)^{\frac{1}{2}}
\leq c \lambda.
\]
Now, multiplying the previous two inequalities and using Young's inequality \eqref{eq:young}, \eqref{eq:hok2.4}, \eqref{(2.6a)},  the Jensen inequality in Lemma~\ref{lem:Jensen} with $\psi(t):=\phi_{|D\bm\ell |}(t^{\frac{1}{p_0}})^{\theta_0}$ with $\theta_0\in(0,1)$ sufficiently close to $1$ and \eqref{A0estimate}, 
we obtain that  for any $\delta\in (0,1)$
\begin{equation}\label{22-2}
\begin{split}
\frac{\phi_{|D{\bm \ell}|}(\lambda)}{\lambda^2} \fint_{\Q^\lambda_{2\rho}}\bigg|\frac{\bfu_{\bm\ell}-(\bfu_{\bm\ell})^\lambda_\rho}{\rho}\bigg|^2\,\d z
&\le%sssim 
c\phi_{|D{\bm \ell}|}'(\lambda) \bigg[c_\delta\bigg(\fint_{\Q^\lambda_{2\rho}}|D\bfu_{\bm\ell}|^{p_0}\,\d z\bigg)^\frac1{p_0} + A_0 + \delta \lambda \bigg] \\
&\le%sssim 
c_\delta \phi_{|D{\bm \ell}|}\Bigg(\bigg(\fint_{\Q^\lambda_{2\rho}}|D\bfu_{\bm\ell}|^{p_0}\,\d z\bigg)^\frac1{p_0}\Bigg) + c_\delta\phi_{|D{\bm \ell}|}(A_0) + c \delta \phi_{|D{\bm \ell}|}(\lambda) \\
&\le%sssim 
c_{\delta}\left(\fint_{\Q^\lambda_{2\rho}}\phi_{|D{\bm \ell}|}(|D\bfu_{\bm\ell}|)^{\theta_0}\,\d z\right)^{\frac{1}{\theta_0}} +c\delta \phi_{|D{\bm \ell}|}(\lambda).
\end{split}
\end{equation}
%To get the last inequality, we use the Jensen inequality in  Lemma~\ref{lem:Jensen} with $\psi(t):=\phi_{|D\ell |}(t)^{\frac{1}{p_0}}$.
Finally, inserting \eqref{22-1} and \eqref{22-2} into \eqref{22}, we find that 
\[
\fint_{\Q_\rho^\lambda} \phi_{|D{\bm \ell}|}(|D\bfu_{\bm\ell}|)\, \d z \le c_\delta  \left(\fint_{\Q^\lambda_{2\rho}}\phi_{|D{\bm \ell}|}(|D\bfu_{\bm\ell}|)^{\theta_0}\,\d z\right)^{\frac{1}{\theta_0}} + c \delta \phi_{|D{\bm \ell}|}(\lambda)\,.
\]
Choosing $\delta$ so small that $c\delta=\frac12$ and 
absorbing the term in the left-hand side by \eqref{lemreverseass} we obtain the reverse 
H\"older inequality \eqref{reverseestimate}.
\end{proof}

Finally, by arguing exactly as in \cite[Section 4]{HasOk21} with $\phi_{|D{\bm \ell}|}$ and $\bfu_{\bm \ell}$ in  place of  $\phi$ and $\bfu$, respectively, we have the following higher integrability result for $D\bfu_{\bm \ell}$.

\begin{theorem}\label{thm:high}
Let $\bfu$ be a 
local weak solution to \eqref{system1}.
There exists $\sigma=\sigma(n,N,p,q,L,\nu)>0$ such that $\phi_{|D{\bm \ell}|}(|D\bfu_{\bm\ell}|)\in L^{1+\sigma}_{loc}(\Omega_T)$ with the following estimate: for any $Q_{2\rho}\Subset\Omega_T$,
\[
\fint_{Q_{\rho}} \phi_{|D{\bm \ell}|}(|D\bfu_{\bm\ell}|)^{1+\sigma}\,\d z \leq c \left[(\phi_{|D{\bm \ell}|} \circ \mathcal{D}^{-1})\bigg(\fint_{Q_{2\rho}}\phi_{|D{\bm \ell}|}(|D\bfu_{\bm\ell}|)\,\d z\bigg)\right]^\sigma\fint_{Q_{2\rho}}\phi_{|D{\bm \ell}|}(|D\bfu_{\bm\ell}|)\,\d z
\]
for some $c=c(n,N,p,q,L,\nu,\Lambda)>0$, where 
$\mathcal{D}(t):=\min\{t^2,\phi_{|D{\bm \ell}|}(t)^{\frac{n+2}{2}}t^{-n}\}$
and $\mathcal{D}^{-1}$ is the inverse of $\mathcal D$.
\end{theorem}

Moreover, by a scaling argument, we have the following homogeneous higher integrability result in intrinsic parabolic cylinders with $\phi$.

\begin{corollary}\label{cor:high}
Let $\bfu$ be a 
local weak solution to \eqref{system1}.
There exists $\sigma=\sigma(n,N,p,q,L,\nu)>0$ such that  if  $Q^\lambda_{2\rho}\Subset\Omega_T$ and 
\begin{equation}
\fint_{Q^\lambda_{2\rho}} \phi (|D\bfu|) \, \d z \le \phi(\lambda)
\quad\text{and}\quad
|D\bm\ell|\le \lambda,
\label{eq:assumpt}
\end{equation}
then
\begin{equation}\label{eq:high1}
\left(\fint_{Q^\lambda_{\rho}} \phi_{|D{\bm \ell}|} (|D \bfu_{{\bm\ell}}|)^{1+\sigma}\, \d z\right)^{\frac{1}{1+\sigma}} \leq c \phi(\lambda)
\end{equation}
for some $c=c(n,N,p,q,L,\nu)>0$.
\end{corollary}

\begin{proof}
 
Let
\begin{equation}
\tilde \bfu(x,t) := \frac{1}{\lambda}\bfu(x, t \lambda^2/{\phi(\lambda)}), 
\quad \tilde{\bA}({\bf P}) := \frac{\lambda \bA(\lambda {\bf P} )}{\phi(\lambda)}, 
\quad \tilde \phi(\tau):=\frac{\phi(\lambda \tau)}{\phi(\lambda)}, 
\quad \tilde {\bm\ell} :=\frac{1}{\lambda} {\bm\ell} \,.
\label{eq:scaledfunct}
\end{equation}
Note that $\tilde{\bA}$ satisfies the same properties of $\bA$ listed in Assumption \ref{Ass}, with $\tilde\phi$ in place of $\phi$.
Then $\tilde \bfu$ is a weak solution to 
$$
\partial_t  \tilde \bfu - \mathrm{div}\, \tilde{\bA}(D\tilde \bfu) = {\bf 0} \quad \text{in }\ Q_{2\rho}\,.
$$
%where \textcolor{blue}{$\tilde{\bA}$ satisfies the same properties of $\bA$ with $\tilde\phi$ in place of $\phi$. }
Moreover, by \eqref{eq:assumpt}, we have
$$
\fint_{Q_{2\rho}} \tilde \phi (|D\tilde \bfu|) \, \d z \le 1
\quad \text{and}\quad
|D\tilde {\bm\ell}| \le  1\,,
$$
whence, taking into account \eqref{eq:approx},  %by Theorem~\ref{thm:high},
$$
\fint_{Q_{2\rho}} \tilde \phi_{{|D \tilde{\bm \ell}|}} (|D\tilde \bfu_{\tilde {\bm\ell}}|) \, \d z \le c \fint_{Q_{2\rho}} \tilde \phi (|D\tilde \bfu|+ |D \tilde {\bm\ell}|) \,\d z \le c \,.
$$
Therefore, by Theorem~\ref{thm:high} we have 
\begin{equation}
\left(\fint_{Q_{\rho}} \tilde \phi_{{|D \tilde{\bm \ell}|}} (|D\tilde \bfu_{\tilde {\bm\ell}}|)^{1+\sigma} \,\d z \right)^{\frac{1}{1+\sigma}}\le c \,.
\label{eq:highint0}
\end{equation}
In addition, since by \eqref{(2.6b)} and \eqref{eq:scaledfunct}
$$
\tilde \phi_{{|D \tilde{\bm \ell}|}} (\tau) \sim \frac{\phi'(|D{\bm\ell}|+\lambda \tau)}{\phi(\lambda)(|D{\bm\ell}|+\lambda \tau)} (\lambda \tau)^2 \,,
$$
we have, taking into account also \eqref{(2.6a)},
$$
\tilde \phi_{{|D \tilde{\bm \ell}|}} (|D\tilde \bfu_{\tilde {\bm\ell}}|) \sim \frac{\phi'(|D{\bm\ell}|+|D \bfu_{\bm\ell}|)}{\phi(\lambda)(|D{\bm\ell}|+|D \bfu_{\bm\ell}|)} |D \bfu_{\bm\ell}|^2 \sim   \frac{\phi_{|D{\bm \ell}|} (|D \bfu_{{\bm\ell}}|)}{\phi(\lambda)}\,.
$$
Therefore, inserting the above estimate in \eqref{eq:highint0} we obtain \eqref{eq:high1}.
\end{proof}

\section{Nondegenerate regime}\label{Sec:nondegenerate}

In this section we consider the \emph{nondegenerate} regime, which means that the average of the gradient of solution is relatively greater than the relevant excess, see for instance \eqref{Lem:nondegenerate1_Ass2}. In this regime, we apply the $\A$-caloric approximation.

We first show that the solution $\bfu$ to \eqref{system1} is an almost weak solution of a linear system with constant coefficients.  
%As a novelty with respect to previous caloric type approximations, we do not need to restrict the choice of the test functions in $C^\infty_0$, but we only assume them to be zero on the lateral boundary. This allows us to choose as test functions solutions itself
\begin{lemma}
Let $\bfu$ be a weak solution to \eqref{system1} and $Q^\lambda_r\Subset \Omega_T$. Then for every $\bm\zeta \in C^\infty(Q^\lambda_r;\R^N)$ with $\bm\zeta= {\bf 0} $ on $\partial B_r \times I^\lambda_r$, we have 
\begin{equation}\label{Acalori_pf1}
\begin{aligned}
&\frac{1}{Q^\lambda_r}\left|\int_{Q^\lambda_r} \bfu_{\bm\ell}  \cdot \bm\zeta_t - D\bA(D\bm\ell) \langle D\bfu_{\bm\ell} , D\bm\zeta\rangle \, \d z  - \left[\int_{B_r} \bfu_{\bm\ell}  \cdot \bm\zeta  \, \d x \right]^{t=r^2/\phi''(\lambda)}_{t=-r^2/\phi''(\lambda)}  \right| \\
& \le c \phi'(|D\bm\ell|)\left(\mu^{\gamma}+ \mu \right) \mu \, \|D\bm\zeta\|_{L^\infty(Q^\lambda_r;\R^{Nn})}\,, 
\end{aligned}\end{equation}
for some $c=c(n,N,p,q,L,\nu)$, where $\gamma$, $\bm\ell$ and $\bfu_{\bm\ell}$ are from \eqref{offdiagonal}, \eqref{def_l} and \eqref{def_ul} and 
\begin{equation*}\label{def:gamma}
\mu:= \left(\frac{1}{\phi(|D\bm\ell|)} \fint_{Q^\lambda_r}\phi_{|D \bm\ell|} (|D\bfu_{\bm\ell}|)\,\d z\right)^{\frac{1}{2}}\,.
\end{equation*}
%\comment{Jihoon:  Linearization estimate in \cite{FILV23} is strange. For example, in the definition of $\gamma$ (it is $\sqrt S$ in \cite{FILV23}) $\phi(|DL|)$ is correct. But I couldn't obtain this term from their proof. Moreover, as I wrote in a previous email, the condition in \cite{FILV23} corresponding to \eqref{offdiagonal} is strong.}
\end{lemma}

\begin{proof} It is enough to consider $\bm\zeta\in C^\infty(Q^\lambda_r;\R^N)$ with $\|D\bm\zeta\|_{L^\infty(Q^\lambda_r;\R^{Nn})} \le 1$ by linearity.
From the weak form of \eqref{system1} and the fact that ${\bm\ell}_t=\div (\bA(D{\bm\ell}))=\zero$, we observe that
\begin{equation}
\begin{aligned}
&\fint_{Q^\lambda_r} \bfu_{\bm\ell}  \cdot \bm\zeta_t - D\bA (D{\bm\ell})   \langle D\bfu_{\bm\ell} , D\bm\zeta\rangle \, \d z  - \left[\int_{B_r} \bfu_{\bm\ell}  \cdot \bm\zeta  \, \d x \right]^{t=r^2\lambda^2/\phi(\lambda)}_{t=-r^2\lambda^2/\phi(\lambda)}\\
&=\fint_{Q^\lambda_r} \bfu  \cdot \bm\zeta_t - D\bA (D{\bm\ell}) \langle D\bfu_{\bm\ell} , D\bm\zeta\rangle \, \d z  - \left[\int_{B_r} \bfu  \cdot \bm\zeta  \, \d x \right]^{t=r^2\lambda^2/\phi(\lambda)}_{t=-r^2/\lambda^2\phi(\lambda)} \\
&= \fint_{Q^\lambda_r} \langle \bA(D\bfu) -\bA(D{\bm\ell}) , D\bm\zeta \rangle - D\bA (D{\bm\ell}) \langle D\bfu_{\bm\ell} , D\bm\zeta\rangle \, \d z\\
&=\fint_{Q^\lambda_r} \int_0^1 \langle [D\bA( sD\bfu_{\bm\ell}+ D{\bm\ell}) - D\bA (D{\bm\ell})] D\bfu_{\bm\ell} , D\bm\zeta\rangle \, \d s \, \d z\\
&\le \fint_{Q^\lambda_r} \left[\int_0^1 | D\bA( sD\bfu_{\bm\ell}+ D{\bm\ell} ) - D\bA (D{\bm\ell})|  \, \d s \right]|D\bfu_{\bm\ell} | |D\bm\zeta|\, \d z \,.
\end{aligned}
\label{eq:estimalmostcal}
\end{equation}
Set $S_1= \{z\in Q^\lambda_r : |D\bfu_{\bm\ell}(z)| >\frac{1}{2}|D{\bm\ell}|\}$ and  $S_2= \{z\in Q^\lambda_r : |D\bfu_{\bm\ell}(z)|  \le \frac{1}{2}|D{\bm\ell}|\}$.

If $z\in S_1$, using \eqref{DE08lem20} and \eqref{growth} and the fact that $ |D\bfu|+|D{\bm\ell}|\le |D\bfu_{\bm\ell}| + 2|D{\bm\ell}| \le 5 |D\bfu_{\bm\ell}|$,
$$\begin{aligned}
&\int_0^1 | D\bA( sD\bfu_{\bm\ell}(z) + D{\bm\ell}) - D\bA (D{\bm\ell})|  \, \d s \\
& \le  c \int_0^1 \phi''( |sD\bfu(z)+(1-s)D{\bm\ell}|) \, \d s + c \phi'' (|D{\bm\ell}|) \\
%&\le  c  \phi''( |D\bfu(z)| +|D{\bm\ell}|)  + c \phi'' (|D{\bm\ell}|) \\
&\le c  \frac{\phi'( |D\bfu(z)| +|D{\bm\ell}|)}{ |D\bfu(z)| +|D{\bm\ell}|} +c  \frac{\phi'( |D\bfu(z)| +|D{\bm\ell}|)}{ |D{\bm\ell}|}  \\
&\le \frac{c}{|D{\bm\ell}|}  \phi'( |D\bfu_{\bm\ell}(z)| )\le \frac{c}{|D{\bm\ell}|}    \phi'( |D\bfu_{\bm\ell}(z)| +|D{\bm\ell} | ) \frac{5|D\bfu_{\bm\ell}(z)|}{ |D\bfu_{\bm\ell}(z)| + 2|D{\bm\ell} |} \le  \frac{c}{|D{\bm\ell}|}  \phi_{|D{\bm \ell}|}'( |D\bfu_{\bm\ell}(z)|)\,,
\end{aligned}$$
hence 
\begin{equation}
\fint_{Q^\lambda_r} \left[\int_0^1 | D\bA( sD\bfu_{\bm\ell}+ D{\bm\ell} ) - D\bA (D{\bm\ell})|  \, \d s \right]|D\bfu_{\bm\ell}| \chi_{S_1}\, \d z\le   c\phi'(|D{\bm\ell}|) \mu^2\,.
\label{eq:estimlargegrad}
\end{equation}

On the other hand, if $z\in S_2$,
applying \eqref{offdiagonal} with $\bP=D{\bm\ell}$ and $\bQ=sD\bfu_{\bm\ell}(z) + D{\bm\ell}$
$$
\int_0^1 | D\bA( sD\bfu_{\bm\ell}(z) + D{\bm\ell}) - D\bA (D{\bm\ell})|  \, \d s \le c \, \left(\frac{|D\bfu_{\bm\ell}|}{|D{\bm\ell}|}\right)^{\gamma} \phi''(|D{\bm\ell}|) \,
$$
and 
$$\begin{aligned}
\frac{|D\bfu_{\bm\ell}(z)|^2}{|D{\bm\ell}|^2} &\le \frac{\phi'(|D\bfu_{\bm\ell}(z)|+|D{\bm\ell}|)}{\phi'(|D{\bm\ell}|)|D{\bm\ell}|}\frac{|D\bfu_{\bm\ell}(z)|^2}{|D{\bm\ell}|} \\
&\le \frac{\phi'(|D\bfu_{\bm\ell}(z)|+|D{\bm\ell}|)}{\phi'(|D{\bm\ell}|)|D{\bm\ell}|}\frac{|D\bfu_{\bm\ell}(z)|^2}{|D\bfu_{\bm\ell}|+ \frac{1}{2}|D{\bm\ell}|}  \le c \frac{\phi_{|D\bm\ell|}(|D\bfu_{\bm\ell}|)}{\phi(|D{\bm\ell}|)}\,.
\end{aligned}$$
Then using these estimates and  applying H\"older's inequality, we have
\begin{equation}
\begin{aligned}
&\fint_{Q^\lambda_r} \left[\int_0^1 | D\bA( sD\bfu_{\bm\ell}+ D{\bm\ell} ) - D\bA (D{\bm\ell})|  \, \d s \right] |D\bfu_{\bm\ell} | \chi_{S_2}\, \d z \\
&\le c \phi'(|D{\bm\ell}|) \fint_{Q^\lambda_r}   \frac{|D\bfu_{\bm\ell}(z)|}{|D{\bm\ell}|} \, \left(\frac{|D\bfu_{\bm\ell}(z)|}{|D{\bm\ell}|}\right)^{\gamma}\chi_{S_2}\, \d z\\
&\le c \phi'(|D{\bm\ell}|) \left[  \fint_{Q^\lambda_r}  \frac{|D\bfu_{\bm\ell}(z)|^2}{|D{\bm\ell}|^2}\chi_{S_2} \, \d z \right]^{\frac{1}{2}}  \left[  \fint_{Q^\lambda_r} \left(\frac{|D\bfu_{\bm\ell}(z)|}{|D{\bm\ell}|}\right)^{2\gamma}\chi_{S_2}\, \d z\right]^{\frac{1}{2}}\\
&\le c \phi'(|D{\bm\ell}|) \left[  \fint_{Q^\lambda_r}  \frac{\phi_{|D\bm\ell|}(|D\bfu_{\bm\ell}|)}{\phi(|D{\bm\ell}|)} \, \d z \right]^{\frac{1}{2}} \left[\fint_{Q^\lambda_r} \frac{\phi_{|D\bm\ell|}(|D\bfu_{\bm\ell}|)}{\phi(|D{\bm\ell}|)}\,\d z\right]^{\frac{\gamma}{2}}\\
&\le c \phi'(|D{\bm\ell}|) \mu^{1+\gamma} \,.
\end{aligned}
\label{eq:estimsmallgrad}
\end{equation}
Therefore, plugging \eqref{eq:estimlargegrad} and \eqref{eq:estimsmallgrad} into \eqref{eq:estimalmostcal} we obtain \eqref{Acalori_pf1}. 
\end{proof}
%
%We note that there is the $L^2$-integral term on the right hand side of the Caccioppoli estimate in \eqref{eq:caccio1}. The following lemma shows that this term can be estimated by $L^\phi$-integral. The proof is exactly the same as the estimation of $T(2\rho,3\rho)$ in the proof of Lemma~\ref{Lem:2ndTerm} (see also the proof pf \cite[]{}). Hence we omit the  proof. 
%\begin{lemma}\label{Lem:L2estimate}
%Let $\bfu$ be a weak solution to \eqref{system1}. Suppose that $Q^\lambda_{2r}\Subset \Omega_T$ and 
%$$
%\fint_{Q^\lambda_{2r}}\phi(|D\bfu|)\, dz\leq  \phi(\lambda) \,.
%$$
%Then 
%\[
%\fint_{Q^\lambda_{3r/2}}\left|\frac{\bfu-(\bfu)^\lambda_{3r/2}}{r}\right|^2\,dz
%\le
%c \lambda^2.
%\]
%for some $c=c(n,N,p,q,L,\nu)>0$.
%\end{lemma}

Now, we derive an excess decay estimate in the nondegenerate regime.

\begin{lemma} \label{Lem:nondegenerate1}
Let $Q_{r}^\lambda = Q_{r}^\lambda(z_0)\Subset \Omega_T$, $\beta\in(0,1)$, and $\bfu$ be a weak solution to \eqref{system1}.
Suppose that
\begin{equation}\label{Lem:nondegenerate1_Ass1}
\frac{\lambda}{2K}\le |(D\bfu)^\lambda_{r}| \le 2K \lambda
\end{equation}
for some $K>0$. There exist small $\delta_0,\theta \in(0,1)$ depending on $n,N,p,q,L,\nu,K,\gamma$ and $\beta$ such that if 
\begin{equation}\label{Lem:nondegenerate1_Ass2}
\fint_{Q^\lambda_{r}} \phi_{|(D\bfu)_{r}^\lambda|} (|D\bfu - (D\bfu)_{r}^\lambda|) \, \d z  \le \delta_0 \phi(|(D\bfu)^\lambda_{r}|)
\end{equation}
then
\begin{equation}\label{nondegenerate_decay}
\fint_{Q^\lambda_{\theta r}} \phi_{|(D\bfu)_{\theta r}^\lambda|} (|D\bfu - (D\bfu)_{\theta r}^\lambda|) \, \d z  \le  \theta^{2\beta} \fint_{Q^\lambda_{r}} \phi_{|(D\bfu)_{r}^\lambda|} (|D\bfu - (D\bfu)_{r}^\lambda|) \, \d z\,.
\end{equation}
\end{lemma}

\begin{proof}
For simplicity, we assume that $z_0=(x_0,t_0)=(0,0)$. We fix the linear function  
$$
\bm\ell(x) := (D\bfu)_{r}^\lambda \, x +(\bfu)_{r}^\lambda ,\quad x\in \R^n.
$$
Then we have  $D\bm\ell=(D\bfu)^\lambda_{r}$ and
$$
\fint_{Q^\lambda_{r}} \phi_{|D\bm\ell|} (|D\bfu_{\bm\ell}|) \, \d z =\fint_{Q^\lambda_{r}} \phi_{|(D\bfu)_{r}^\lambda|} (|D\bfu - (D\bfu)_{r}^\lambda|) \, \d z \,.
$$
We divide the proof into three steps.

\vspace{0.3cm}

\textit{Step 1. (Scaling)} 
%We first observe from \eqref{eq:approx}, \eqref{Lem:nondegenerate1_Ass1} and \eqref{Lem:nondegenerate1_Ass2} that
%$$
%\fint_{Q^\lambda_{r}} \phi (|D\bfu|) \, \d z \le 2^{n+2}\fint_{Q^\lambda_{2r}} \phi (|D\bfu|) \, \d z \le c\fint_{Q^\lambda_{2r}} \phi_{|D\bm\ell|} (|D\bfu_{\bm\ell}|) \, \d z +  c\phi(|D\bfu_{\bm\ell}|) \le c \phi(\lambda)\,.
%$$
%and
%$$\begin{aligned}
%\sup_{t\in I_r^\lambda} \fint_{B_r}\frac{|\bfu-(\bfu)_{r}^\lambda|^2}{r^2}\, \d x 
%& \le c\fint_{Q^\lambda_{3r/2}}\frac{|\bfu-(\bfu)^\lambda_{3r/2}|^2}{r^2}\,\d z + \frac{c}{\phi''(\lambda)} \fint_{Q^\lambda_{3r/2}}\phi\left(\frac{|\bfu-(\bfu)^\lambda_{3r/2}|}{r}\right)\, \d z\\
%& \le c\lambda^2 + \frac{c}{\phi''(\lambda)} \fint_{Q^\lambda_{3r/2}}\phi\left(|D\bfu|\right)\, \d z \le c \lambda^2\,.
%\end{aligned}$$
%Now,
We consider the following scaled functions: 
$$
\tilde \bfu(x,t) :=\frac{1}{\lambda} \bfu(x, {t}\lambda^2/{\phi(\lambda)}), 
\quad \tilde{\bA}({\bf P}) := \frac{\lambda \bA(\lambda {\bf P} )}{\phi(\lambda)}, 
\quad \tilde \phi(\tau):=\frac{\phi(\lambda \tau)}{\lambda\phi'(\lambda)}, 
\quad \tilde {\bm\ell} :=\frac{1}{\lambda} \bm\ell\,.
$$
Then, by a direct computation, we have
\begin{equation*}
\tphi_{\frac{a}{\lambda}}(\tau) = \frac{\phi_{a}(\lambda \tau)}{\lambda \phi'(\lambda)}  \quad \mbox{ and } \quad \tilde{\phi}_{\frac{a}{\lambda}}'(\tau) = \frac{ {\phi}_{{a}}'(\lambda \tau)}{\phi'(\lambda)}\,,
%\phi_{\frac{a}{\lambda}}(t) = \frac{1}{\lambda} \phi_{{a}}(t)} \quad \mbox{ and } \quad \tilde{\phi}_{\frac{a}{\lambda}}'(t) = \frac{\lambda {\phi}_{\frac{a}{\lambda}}'(\lambda t)}{\phi'(\lambda)} = \frac{ {\phi}_{{a}}'(\lambda t)}{\phi'(\lambda)}\,,
\end{equation*}
whence 
$$
\tilde\phi_{|D\tilde{\bm\ell}|}'(\tau)=\frac{\phi'_{|D\bm\ell|}(\lambda\tau)}{\phi'(\lambda)}=\frac{\phi'(|D\bm\ell| +\lambda \tau)}{\phi'(\lambda)(|D\bm\ell| +\lambda \tau)}(\lambda\tau)\,. 
$$ 
In particular, taking into account \eqref{Lem:nondegenerate1_Ass1}, 
\begin{equation*}
\tilde\phi_{|D\tilde{\bm\ell}|}'(1)\sim \tilde\phi_{|D\tilde{\bm\ell}|}(1)\sim 1\,.
\end{equation*}
Moreover, $\tilde \bfu$ is a weak solution to 
$$
\partial_t  \tilde \bfu - \mathrm{div}\, \tilde{\bA}(D\tilde \bfu) = {\bf 0} \quad \text{in }\ Q_{r}\,,
$$
where $\tilde {\bA}$ satisfies the same properties of $\bA$ listed in Assumption \ref{Ass}, with $\tilde\phi$ in place of $\phi$, and satisfies 
\begin{equation}\label{aveDu}
\frac{1}{2K}\le |D\tilde {\bm\ell}|= |(D\tilde\bfu)_r| \le  2K,
%\quad \frac{1}{c} \le \fint_{Q_{r}} \tilde \phi (|D\tilde \bfu|)\,\mathrm{d}z =\frac{1}{\lambda \phi'(\lambda)} \fint_{Q^\lambda_{r}} \phi (|D \bfu|) \, \d z \le c \,,
\end{equation}
and by \eqref{Lem:nondegenerate1_Ass2}
\begin{equation}\label{gammadelta}
\mu := \left(\frac{1}{\phi(|D\bm\ell|)} \fint_{Q^\lambda_r}\phi_{|D{\bm\ell}|} (|D\bfu_{\bm\ell}|)\,\d z\right)^{\frac{1}{2}} \le  \sqrt{\delta_0} \le 1 \,, 
\end{equation}
where  $\delta_0\le 1$ will be determined later.
%where $\mu$ is from \eqref{def:gamma}.
Note that  the last inequality also yields
\begin{equation}\label{Duphidelta}
  \fint_{Q_r}\tilde\phi_{|D\tilde {\bm\ell}|} (|D\tilde{\bfu}_{\tilde {\bm\ell}}|)\,\d z= \frac{1}{\lambda\phi'(\lambda)}\fint_{Q^\lambda_r}\phi_{|D\bm\ell|} (|D\bfu_{\bm\ell}|)\,\d z \sim \mu^2 \le  \delta_0 \,.
\end{equation}
%In the last inequality we chose $\delta_0$ sufficiently small.
Moreover, by Theorem~\ref{thm:high} and the estimate \eqref{Acalori_pf1} we also have that
\begin{equation}\label{Duphisigma}
\left(\fint_{Q_{r/2}} \tilde \phi_{|D\tilde {\bm\ell}|} (|D\tilde \bfu_{\tilde {\bm\ell}}|)^{1+\sigma_0} \,\d z\right)^{\frac{1}{1+\sigma_0}} \le c\mu^2 \le c\delta_0
\end{equation}
for some $\sigma_0>0$, where we used the fact that $\phi_{|D{\bm\ell}|}(\mu) \sim \mu^2$ hence $\mathcal D\left((\phi_{|D{\bm\ell}|})^{-1}(\mu^2)\right)=\max\left\{(\phi_{|D{\bm\ell}|})^{-1}(\mu^2)^2, \mu^{n+2}(\phi_{|D{\bm\ell}|})^{-1}(\mu^2)^{-n}\right\}$, this
 implies that
\begin{equation*}\label{Acalori_pf2}\begin{aligned}
\frac{1}{|Q_r|}\left|\int_{Q_r} \tilde{\bfu}_{\tilde {\bm\ell}}  \cdot {\bm\zeta}_t - D\tilde{\bA} (D\tilde{\bm\ell}) \langle D\tilde{\bfu}_{\tilde{\bm\ell}} , D{\bm\zeta}\rangle \, \d z -\int_{B_r} \tilde{\bfu}_{\tilde{\bm\ell}}  \cdot {\bm\zeta} \, \d x \right|& \\
&\hspace{-3cm}\le c \left(\big(\sqrt{\delta_0}\big)^{\gamma}+\sqrt{\delta_0} \right) \mu\, \sup_{Q_r}  |D\bm\zeta|\,, 
\end{aligned}\end{equation*}
for all $\bm\zeta\in C^\infty(Q_r)$ with $\bm\zeta={\bf 0}$ on $\partial B_r\times (-r^2,r^2)$\,.

\vspace{0.3cm}

\textit{Step 2. ($\A$-caloric approximation)} Observe that 

$$
\partial_t \tilde \bfu_{\bm\ell} = \partial_t \tilde \bfu  = \div (\tilde{\bA}(D\tilde \bfu) - \tilde{\bA}(D\tilde {\bm\ell})) =: \div \, {\bf H} \quad \text{in } \ Q_r\,,
$$
i.e., ${\bf H}:=\tilde{\bA}(D\tilde \bfu) - \tilde{\bA}(D\tilde {\bm\ell})$, in the distributional sense, and writing $\tilde{\phi}_{|D\tilde {\bm\ell}|}^*:=(\tilde{\phi}_{|D\tilde {\bm\ell}|})^*$
\begin{equation}
\fint_{Q_r} \tilde{\phi}_{|D\tilde {\bm\ell}|}^*(|{\bf H}|)\, \d z =  \fint_{Q_r} \tilde{\phi}_{|D\tilde {\bm\ell}|}^*(|\tilde{\bA}(D\tilde \bfu) - \tilde{\bA}(D\tilde {\bm\ell})|)\, \d z \le c \fint_{Q_r} \tilde{\phi}_{|D\tilde {\bm\ell}|}(|D\tilde \bfu|)\, \d z \le c\mu^2\,.
\label{Duphideltabis}
\end{equation} 
Set 
$$
p_1:=\min\left\{p,\frac{q}{q-1}\right\}
\quad\text{and}\quad
p_0:=\frac{1+p_1}{2} .
$$
Then we see from the Jensen inequality in Lemma~\ref{lem:Jensen} with $\psi(t):=\tphi_{|D\tilde{\bm\ell |}}(t^{\frac{1}{p_1}})$, \eqref{Duphidelta} and \eqref{Duphideltabis} that
$$
\left(\fint_{Q_r} |D\tilde \bfu_{\tilde {\bm\ell}}|^{p_1}\, \d z\right)^{\frac{1}{p_1}}  \le c \tphi_{|D\tilde {\bm\ell}|}^{-1}\left(\fint_{Q_r} \tilde{\phi}_{|D\tilde {\bm\ell}|}(|D\tilde \bfu_{\tilde {\bm\ell}}|)\, \d z \right) \le c \tphi_{|D\tilde {\bm\ell}|}^{-1}(\mu^2)
$$
and, arguing as before with  $\psi(t):=\tphi^*_{|D\tilde {\bm\ell}|}(t^{1/p_1})$,
$$
\left(\fint_{Q_r} |{\bf H}|^{p_1}\, \d z \right)^{\frac{1}{p_1}} \le c(\tilde\phi_{|D\tilde {\bm\ell}|}^*)^{-1}\left(\fint_{Q_r} \tilde{\phi}_{|D\tilde {\bm\ell}|}^*({\bf H})\, \d z \right)\le c(\tilde\phi_{|D\tilde {\bm\ell}|}^*)^{-1}(\mu^2)\,.
$$
Furthermore, we notice from the first inequality in \eqref{aveDu} and the fact that $\mu\in(0,1)$ that 
$$
 \frac{\tilde\phi'(|D\tilde {\bm\ell}|+\mu)}{|D\tilde {\bm\ell}|+\mu} \sim 1 
 $$
 and for $\tau_1\ge 0$ satisfying that $\tilde \phi_{|D\tilde {\bm\ell}|}'(\tau_1)=\mu$,
$$ 
 \tau_1 \lesssim 1   \ \ \text{hence}\ \ \tau_1\sim \frac{\tilde\phi'(|D\tilde {\bm\ell}|+\tau_1)}{|D\tilde {\bm\ell}|+\tau_1} \tau_1 =\mu  \,,
$$
which imply
\begin{equation}\label{gammaequiv}
\mu^2 \sim  \tilde\phi_{|D\tilde {\bm\ell}|} (\mu)
\quad \text{and}\quad
\mu^2 \sim \tau_1 \mu = (\tilde \phi_{|D\tilde {\bm\ell}|}')^{-1}(\mu) \mu = (\tilde \phi_{|D\tilde {\bm\ell}|}^*)'(\mu)\mu \sim  \tilde \phi_{|D\tilde {\bm\ell}|}^*(\mu)\,.
\end{equation}
Collecting the previous estimates, we then have
\begin{equation}\label{Dup1}
\left(\fint_{Q_r} |D\tilde \bfu_{\tilde {\bm\ell}}|^{p_1}\, \d z+ \fint_{Q_r} |{\bf H}|^{p_1}\, \d z\right)^{\frac{1}{p_1}}
\le c \mu \,.
\end{equation}

%\comment{Jihoon: We can obtain caloric approximation in $L^{p_0}$ space since ${\bf H}\in L^{\psi^*}$. This seems unavoidable even in the parabolic system with $p$-growth case.}
Therefore, by Theorem~\ref{thm:Acaloric} with, in particular, $\A:=\tilde{\bA}(D{\tilde {\bm\ell}})$, $\psi(\tau):=\tau^{p_0}$ and $\sigma:=\frac{p_1}{p_0}-1$ (i.e., $p_0(1+\sigma)=p_1$), for $\epsilon\in(0,1)$ to be determined small later, there exists small $\delta_0>0$ depending on $n,N,L,\nu,p,q,\gamma$ and $\epsilon$ such that  
\begin{equation}\label{comparisonp0}
\fint_{Q_{r}} |D\tilde \bfu_{\tilde {\bm\ell}}- D \bfh|^{p_0}\, \d z \le \epsilon \mu^{p_0}\,,
\end{equation}
where $\bfh$ is the weak solution to 
$$
\begin{cases}
\partial_t  \bfh- \div (\A D\bfh) = {\bf0}  \quad \text{in}\ \ Q_r, \\
 \bfh =\tilde \bfu_{\tilde {\bm\ell}} \quad \text{on}\ \ \partial_{\mathrm{p}} Q_r\,.
\end{cases}$$
We note from  \eqref{DhLip}, \eqref{comparisonp0}, \eqref{Dup1} and \eqref{gammaequiv} that
\begin{equation}\label{Dhphisigma}
\left(\fint_{Q_{r/2}}\tilde\phi_{|D\tilde {\bm\ell}|} (|D\bfh|)^{1+\sigma_0} \,\d z\right)^{\frac{1}{1+\sigma_0}}\le c \tilde\phi_{|D\tilde {\bm\ell}|} \left(\fint_{Q_{r}} |D \bfh|\,\d z\right) \le  c \tilde\phi_{|D\tilde {\bm\ell}|}(\mu)\le c \mu^2  \,.
\end{equation}
Therefore, by H\"older's inequality, the Jensen's inequality in Lemma~\ref{lem:Jensen} with $\psi^{-1}(t):=\tilde\phi_{|D\tilde {\bm\ell}|}(t)^{\frac{1}{q}}$  and the estimates \eqref{Duphisigma}, \eqref{comparisonp0}, \eqref{Dhphisigma} and \eqref{gammaequiv}, we have that  with  $\kappa_0\in(0,1)$ satisfying $\frac{\kappa_0}{q}+(1-\kappa_0)(1+\sigma_0)=1$,
\begin{equation}\label{comparisonphi}\begin{aligned}
&\fint_{Q_{r/2}} \tilde\phi_{|D\tilde {\bm\ell}|}(|D\tilde \bfu_{\tilde {\bm\ell}}- D \bfh|)\, \d z\\ 
& \le \left(\fint_{Q_{r/2}} \tilde\phi_{|D\tilde {\bm\ell}|}(|D\tilde \bfu_{\tilde {\bm\ell}}- D \bfh|)^{\frac{1}{q}}\, \d z\right)^{\kappa_0}\left(\fint_{Q_{r/2}} \tilde\phi_{|D\tilde {\bm\ell}|}(|D\tilde \bfu_{\tilde {\bm\ell}}- D \bfh|)^{1+\sigma_0}\, \d z\right)^{1-\kappa_0}\\
& \le c \tilde\phi_{|D\tilde {\bm\ell}|}\left(\left[\fint_{Q_{r/2}} |D\tilde \bfu_{\tilde {\bm\ell}}- D \bfh|^{p_0}\, \d z\right]^{\frac{1}{p_0}}\right)^{\frac{\kappa_0}{q}} \mu^{2(1-\kappa_0)(1+\sigma)}\\
& \le c \epsilon^{\frac{p\kappa_0}{p_0q}} \tilde\phi_{|D\tilde {\bm\ell}|}(\mu)^{\frac{\kappa_0}{q}} \mu^{2(1-\kappa_0)(1+\sigma)} \le c \epsilon^{\frac{p\kappa_0}{p_0q}} (\mu^2)^{\frac{\kappa_0}{q}+(1-\kappa_0)(1+\sigma)} = c \epsilon^{\frac{p\kappa_0}{p_0q}} \mu^2\,.
\end{aligned}\end{equation}
 Moreover, by \eqref{DhLip} in Lemma~\ref{Lem:Acaloricmap} with $\rho=r/2$, \eqref{Dhphisigma} and \eqref{gammadelta}, we also have
$$
\sup_{Q_{r/4}} |D\bfh| \le c (\tilde\phi_{|D\tilde {\bm\ell}|})^{-1} \left(\fint_{Q_{r/2}} \tilde\phi_{|D\tilde {\bm\ell}|}(|D \bfh|)\, \d z \right) \le c  \mu \le c \sqrt{\delta_0}.
$$
Note that we choose $\delta_0$ small so that 
\begin{equation}\label{Dhsmall}
\sup_{Q_{r/4}} |D\bfh| \le c\mu \le \frac{1}{4K}\,.
\end{equation}

\textit{Step 3. (Decay estimate)} Let $\theta\in (0,1/8)$ to be determined later and recall function $\bV$ corresponding to the $N$-function $\tilde\phi$ defined as in \eqref{Vfunction}. We first observe from \eqref{aveDu} and \eqref{Dhsmall} that
$$
\frac{1}{8}|(D\tilde \bfu)_{r}|\le \frac{K}{4} \le |(D\tilde \bfu)_{r}| - | (D \bfh)_{\theta r}| \le |(D\tilde \bfu)_{r} + (D \bfh)_{\theta r}|\le |(D\tilde \bfu)_{r} |+\frac{1}{4K} \le  \frac{9}{4}|(D\tilde \bfu)_{r}|\,.
$$
Then, using \eqref{monotonicity1}, \eqref{Vintequivalent} and the preceding estimate,
$$\begin{aligned}
\fint_{Q_{\theta r}} \tilde\phi_{|(D\tilde \bfu)_{\theta r}|}(|D\tilde\bfu-(D\tilde \bfu)_{\theta r}|)|)\, \d z& \sim \fint_{Q_{\theta r}} |\bV(D\tilde \bfu) - \bV( (D\tilde \bfu)_{\theta r})|^2\, \d z \\
 &\sim  \fint_{Q_{\theta r}} |\bV(D\tilde \bfu) - (\bV( D\tilde \bfu))_{\theta r}|^2\, \d z\\
&\le \fint_{Q_{\theta r}} |\bV(D\tilde \bfu) - \bV( (D\tilde \bfu)_{r} + (D \bfh)_{\theta r})|^2\, \d z\\
&\sim  \fint_{Q_{\theta r}} \tilde\phi_{|(D\tilde \bfu)_{r} + (D \bfh)_{\theta r}|}(|D\tilde \bfu -(D\tilde \bfu)_{r} - (D \bfh)_{\theta r}|)\, \d z\\
&\sim  \fint_{Q_{\theta r}} \tilde\phi_{|(D\tilde \bfu)_{r}|}(|D\tilde \bfu -(D\tilde \bfu)_{r} - (D \bfh)_{\theta r}|)\, \d z\,.
\end{aligned}$$
Moreover, by \eqref{aveDu} and \eqref{Dhsmall} we have
$$\tilde\varphi_{|D\tilde{\bm\ell}|}
(\theta (|D{\bf h}|)_{r/4}) \lesssim
\tilde\varphi_{|D\tilde{\bm\ell}|}
(\theta \mu) \sim
\theta^2\mu^2\,,
$$
Using the above two estimates, \eqref{Dhdecay} in Lemma~\ref{Lem:Acaloricmap} with $\rho=r/2$ and \eqref{comparisonphi},
we obtain
$$\begin{aligned}
&\fint_{Q_{\theta r}} \tilde\phi_{|(D\tilde \bfu)_{\theta r}|}(|D\tilde\bfu-(D\tilde \bfu)_{\theta r}|)\, \d z\\
& \le c \fint_{Q_{\theta r}} \tilde\phi_{|(D\tilde \bfu)_{r}|}(|D\tilde \bfu -(D\tilde \bfu)_{r} - (D \bfh)_{\theta r}|)\, \d z \\
&\le c  \fint_{Q_{\theta r}} \tilde\phi_{|D\tilde {\bm\ell}|}(|D\tilde \bfu_{\tilde{\bm\ell}}  - D \bfh|)\, \d z +c \fint_{Q_{\theta r}} \tilde\phi_{|D\tilde {\bm\ell}|}(|D \bfh - (D \bfh)_{\theta r})|)\, \d z\\
&\le c \theta^{-(n+2)}\fint_{Q_{r/2}} \tilde\phi_{|D\tilde {\bm\ell}|}(|D\tilde \bfu_{\tilde {\bm\ell}}- D \bfh |)\, \d z +  c \tilde\phi_{|D\tilde {\bm\ell}|}\left( \theta \fint_{Q_{r/4}} |D \bfh- (D \bfh)_{r/4}|\, \d z \right) \\
&\le c\theta^{-(n+2)} \epsilon^{\frac{p\kappa_0}{p_0q}}\mu^2 + c \theta^{2}\mu^2.
\end{aligned}$$
Finally, choosing $\theta$ small so that $c \theta^{2(1-\beta)}\le \frac{1}{2}$ and then $\epsilon$ small so that $c\theta^{-(n+2)} \epsilon^{\frac{p\kappa_0}{p_0q}}\le \frac{1}{2}\theta^{2\beta}$, we obtain 
$$
\fint_{Q_{\theta r}} \tilde\phi_{|(D\tilde \bfu)_{\theta r}|}(|D\tilde\bfu-(D\tilde \bfu)_{\theta r}|)\, \d z  \le \theta^{2\beta}\mu^2\,,
$$
whence, by scaling back,
$$
\fint_{Q_{\theta r}} \phi_{|(D\bfu)_{\theta r}^\lambda|}(|D\bfu-(D \bfu)^\lambda_{\theta r}|)\, \d z \le \theta^{2\beta} \lambda \phi'(\lambda) \mu^2 \,.
$$ 
This estimate, together with \eqref{Duphidelta}, yields \eqref{nondegenerate_decay} by choosing $\theta$ sufficiently small depending on $n,N,p,q,L,\nu,K,\gamma$ and $\beta$. This concludes the proof.
\end{proof}

%\newpage

From the previous lemma, we obtain decay estimates for $D \bfu$ in the nondegenerate regime. 
%Now we derive decay estimate for $D \bfu$ in the nondegenerate case.  

\begin{lemma} \label{Lem:nondegenerate2}
Let $\lambda>0$, $\beta\in(0,1)$, $Q_{R}^\lambda = Q_{R}^\lambda(z_0)\Subset \Omega_T$ and $\bfu$ be a weak solution to \eqref{system1}.
Suppose
\begin{equation*}\label{Lem:nondegenerate2_ass1}
\frac{\lambda}{K_0}\le |(D\bfu)^\lambda_{R}| \le K_0 \lambda
\end{equation*}
for some $K_0>0$. There exists small $\delta_1\in (0,1)$ depending on $n,N,p,q,L,\nu,\beta,\gamma$ and $K_0$ such that if 
\begin{equation}\label{Lem:nondegenerate2_ass2}
\fint_{Q^\lambda_{R}} \phi_{|(D\bfu)_{R}^\lambda|} (|D\bfu - (D\bfu)_{R}^\lambda|) \, \d z  \le \delta_1 \phi(|(D\bfu)^\lambda_{R}|)
\end{equation}
then the limit 
\begin{equation}\label{Lem:nondegenerate2_result1}
\Gamma_{z_0} := \lim_{r \to 0^+} \,(D\bfu)_{Q_r(z_0)}
\end{equation}
exists with 
\begin{equation}\label{Lem:nondegenerate2_result2}
\frac{\lambda}{2K_0} \le |\Gamma_{z_0}| \le 2K_0 \lambda\, ,
\end{equation}
and  for every $r\in(0,R)$,
\begin{equation}\label{nondegenerate_decayDu}
\fint_{Q^\lambda_{r}(z_0)} \phi (|D\bfu - \Gamma_{z_0}|) \, \d z  \le c \Big(\frac{r}{R}\Big)^{\beta_1} \phi (  \lambda ) 
\end{equation}
for some $c=c(n,N,p,q,L,\nu,K_0,\beta,\gamma)>0$ and $\beta_1=\beta_1(p,q,\beta)>0$.
\end{lemma}

\begin{proof}
For simplicity, we shall omit to write the center $z_0$.  
We recall the parameters  $\theta$ and $\delta_0$ from Lemma~\ref{Lem:nondegenerate1}. However, we notice that for any smaller $\theta$ and $\delta_0$ satisfying additional conditions, \eqref{Lem:nondegenerate1_Ass2} and \eqref{nondegenerate_decay}  still hold.
% First we assume that $\theta \le \min\{ 2^{-q/(2\beta)}, 2^{-q/\beta}\}$. 
Then we choose $\delta_1\le \delta_0$. We divide the proof into two steps.

\textit{Step 1.}  We shall prove by induction that for every $i\in\mathbb N$,
\begin{equation}\label{induction1-1}
\fint_{Q^\lambda_{\theta^{i}R}} \phi_{|(D\bfu)_{\theta^{i}R}^\lambda|} (|D\bfu - (D\bfu)_{\theta^{i}R}^\lambda|) \, \d z  \le \theta^{2\beta i} \fint_{Q^\lambda_{R}} \phi_{|(D\bfu)_{R}^\lambda|} (|D\bfu - (D\bfu)_{R}^\lambda|) \, \d z\,,  
\end{equation}
and
\begin{equation}\label{induction1-2}
\tfrac{1}{2}|(D\bfu)_{R}^\lambda| \le  \bigg[1-\tfrac{1}{4}\sum_{k=0}^{i-1} 2^{-k}\bigg]|(D\bfu)_{R}^\lambda|  \le |(D\bfu)_{\theta^{i}R}^\lambda| 
\le \bigg[1+\tfrac{1}{4}\sum_{k=0}^{i-1} 2^{-k}\bigg] |(D\bfu)_{R}^\lambda| \le \tfrac{3}{2}|(D\bfu)_{R}^\lambda|\,.
\end{equation}
Suppose $i=1$. Then \eqref{nondegenerate_decay} yields \eqref{induction1-1}.  In order to prove  \eqref{induction1-2}, we first observe from \eqref{eq:approx} and \eqref{Lem:nondegenerate2_ass2} that
$$
\fint_{Q^\lambda_{R}} \phi (|D\bfu|) \, \d z \le c \fint_{Q^\lambda_{R}} \phi_{|(D\bfu)_{R}^\lambda|} (|D\bfu - (D\bfu)_{R}^\lambda|) \, \d z  +  c \phi (| (D\bfu)_{R}^\lambda|) \le c \phi (| (D\bfu)_{R}^\lambda|) .
$$
and, applying \eqref{phipqinequality} with $\epsilon=\delta_0^{\frac{1}{2}}$,
$$\begin{aligned}
&\fint_{Q^\lambda_{R}} \phi (|D\bfu - (D\bfu)_{R}^\lambda|) \, \d z \\
&\le    \fint_{Q^\lambda_{R}} \delta_0^{\frac{1}{2}} [ \phi (|D\bfu|) + \phi(|(D\bfu)_{R}^\lambda|)]  + c \delta_0^{-\frac{1}{2}}   \phi_{|(D\bfu)_{R}^\lambda|} (|D\bfu - (D\bfu)_{R}^\lambda|) \, \d z 
\le c\delta_0^{\frac{1}{2}} \phi(|(D\bfu)_{R}^\lambda|) \,.
\end{aligned}$$
Hence we have
$$\begin{aligned}
|(D\bfu)_{\theta R}^\lambda - (D\bfu)_{R}^\lambda| &\le \theta^{-\frac{n+2}{p}}\phi^{-1} \left(\fint_{Q^\lambda_{R}} \phi (|D\bfu - (D\bfu)_{R}^\lambda|) \, \d z\right) \\
&\le c \theta^{-\frac{n+2}{p}} \phi^{-1}(\delta_0^{\frac{1}{2}} \phi(|(D\bfu)_{R}^\lambda|)) \le c \theta^{-\frac{n+2}{p}} \delta_0^{\frac{1}{2q}} |(D\bfu)_{R}^\lambda| \le \frac{1}{4} |(D\bfu)_{R}^\lambda|\,,
\end{aligned}$$
after choosing sufficiently small 
$\delta_0=\delta_0(n,N,p,q,L,\nu,K_0,\beta)$,
which implies \eqref{induction1-2}. 

Suppose that \eqref{induction1-1} and \eqref{induction1-2} hold for $i=1,2,\dots,j-1$ for some $j\ge 2$. Then, using \eqref{induction1-1}, \eqref{induction1-2} with $i=j-1$ and \eqref{Lem:nondegenerate2_ass2},
 and choosing $\theta$ such that $\theta^{2\beta} \le 2^{-q} $, we have
\begin{equation}\label{Lem:nondegenerate2_pf3}\begin{aligned}
\fint_{Q^\lambda_{\theta^{j-1}R}} \phi_{|(D\bfu)_{\theta^{j-1}R}^\lambda|} (|D\bfu - (D\bfu)_{\theta^{j-1}R}^\lambda|) \, \d z &  \le  \theta^{2\beta (j-1)} \delta_1 \phi(2|(D\bfu)^\lambda_{\theta^{j-1}R}|) \\
&\le   \delta_0 \phi(|(D\bfu)^\lambda_{\theta^{j-1}R}|) \,.
\end{aligned}\end{equation}
Since  $\frac{1}{2K_0} \le |(D\bfu)^\lambda_{\theta^{j-1}R}| \le 2K_0$, applying Lemma~\ref{Lem:nondegenerate1} with $r=\theta^{j-1}R$ we obtain 
$$ \begin{aligned}
\fint_{Q^\lambda_{\theta^{j}R}} \phi_{|(D\bfu)_{\theta^{j}R}^\lambda|} (|D\bfu - (D\bfu)_{\theta^{j}R}^\lambda|) \, \d z 
& \le \theta^{2\beta} \fint_{Q^\lambda_{\theta^{j-1}R}} \phi_{|(D\bfu)_{\theta^{j-1}R}^\lambda|} (|D\bfu - (D\bfu)_{\theta^{j-1}R}^\lambda|) \, \d z \\
& \le \theta^{2\beta j} \fint_{Q^\lambda_{R}} \phi_{|(D\bfu)_{R}^\lambda|} (|D\bfu - (D\bfu)_{R}^\lambda|) \, \d z \,,
\end{aligned}$$
which proves \eqref{induction1-1} with $i=j$. We next prove \eqref{induction1-2} with $i=j$. As in the case $i=1$, we have from \eqref{Lem:nondegenerate2_pf3} that 
\begin{equation*}\label{Lem:nondegenerate2_pf6}
\fint_{Q^\lambda_{\theta^{j-1}R}} \phi (|D\bfu - (D\bfu)_{\theta^{j-1}R}^\lambda|) \, \d z
\le c\theta^{\beta (j-1)} \delta_1^{\frac{1}{2}}\phi(|(D\bfu)^\lambda_{\theta^{j-1}R}|)\,,
\end{equation*}
and hence
$$
\begin{aligned}
|(D\bfu)_{\theta^{j} R}^\lambda - (D\bfu)_{\theta^{j-1}R}^\lambda| &\le \theta^{-\frac{n+2}{p}}\phi^{-1} \left(\fint_{Q^\lambda_{\theta^{j-1}R}} \phi (|D\bfu - (D\bfu)_{\theta^{j-1}R}^\lambda|) \, \d z\right) \\
&\le c \theta^{-\frac{n+2}{p}} \phi^{-1} (\delta_1^{\frac{1}{2}}\theta^{\beta (j-1)} \phi( |(D\bfu)_{\theta^{j-1}R}^\lambda|)) \\
&\le c \theta^{-\frac{n+2}{p}}\delta_1^{\frac{1}{2q}}\theta^{\frac{\beta (j-1)}{q}} |(D\bfu)_{\theta^{j-1}R}^\lambda|\,.
\end{aligned}$$
Therefore, choosing $\theta$ such that $\theta^{\frac{\beta}{q}}\le 2^{-1}$ and  $\delta_1$ such that $c \theta^{-\frac{n+2}{p}}\delta_1^{\frac{1}{2q}} \le 1/4$, we obtain 
\begin{equation}\label{Lem:nondegenerate2_pf5}
|(D\bfu)_{\theta^{j} R}^\lambda - (D\bfu)_{\theta^{j-1}R}^\lambda| \le \theta^{\frac{\beta (j-1)}{q}} |(D\bfu)_{\theta^{j-1}R}^\lambda|  \le \frac{1}{4} 2^{-(j-1)}|(D\bfu)_{\theta^{j-1}R}^\lambda|\,,
\end{equation}
which, together with  \eqref{induction1-2} with $i=j-1$, implies  \eqref{induction1-2} with $i=j$.

\textit{Step 2.}  We first note that $\{(D\bfu)_{\theta^{i} R}^\lambda\}_{i\in \mathbb N}$ is a Cauchy sequence. Indeed, by \eqref{Lem:nondegenerate2_pf5} and \eqref{induction1-2}, we have that any $i,j\in \mathbb N$ with $i<j$
\begin{equation}\label{Lem:nondegenerate2_pf4}\begin{aligned}
|(D\bfu)_{\theta^{j} R}^\lambda - (D\bfu)_{\theta^{i} R}^\lambda| 
\le    \sum_{k=i}^{j-1} \theta^{\frac{\beta k}{q}} |(D\bfu)_{\theta^{k}R}^\lambda| 
 \le \frac{1}{4} \sum_{k=i}^{j-1} 2^{-k}| (D\bfu)_{\theta^{k} R}^\lambda| \le \frac{3}{4} 2^{-i} |(D\bfu)_{R}^\lambda|\,.
\end{aligned}\end{equation}
Therefore, set 
$$
\Gamma_{0} :=  \lim_{i\to \infty}  (D\bfu)_{\theta^{i} R}^\lambda\,.
$$
Then \eqref{induction1-2} implies \eqref{Lem:nondegenerate2_result2}. We shall prove \eqref{nondegenerate_decayDu}. 
Note that  by \eqref{induction1-1}, \eqref{Lem:nondegenerate2_ass2}, \eqref{Lem:nondegenerate2_pf4} and \eqref{induction1-2},
$$\begin{aligned}
\fint_{Q^\lambda_{\theta^iR}} \phi(|D\bfu - \Gamma_0 |) \, \d z 
& \le c \fint_{Q^\lambda_{\theta^iR}} \phi(|D\bfu - (D\bfu)_{\theta^{i} R}^\lambda |) \, \d z +c \phi(|(D\bfu)_{\theta^{i} R}^\lambda -\Gamma_0|) \\
& \le c \theta^{\beta i} \phi(\lambda) +  c \phi\bigg(\sum_{k=i}^\infty \theta^{\frac{\beta k}{q}} \lambda \bigg) \le c \theta^{\frac{p\beta i}{q}} \phi(\lambda)
\end{aligned}$$
Therefore, for every $r<\theta R$ with $i\in \mathbb N$ satisfying $\theta^{i+1} R\le r \le \theta^iR$, we have 
$$
\fint_{Q^\lambda_{r}} \phi(|D\bfu - \Gamma_0|) \, \d z  \le \frac{1}{\theta^{n+2}} \fint_{Q^\lambda_{\theta^iR}} \phi(|D\bfu - \Gamma_0|) \, \d z \le c  \theta^{\frac{p\beta i}{q}} \phi(\lambda) \le c \left(\frac{r}{R}\right)^{\frac{p\beta}{q}} \phi(\lambda),
$$
which implies \eqref{nondegenerate_decayDu}. 

It remains to prove \eqref{Lem:nondegenerate2_result1}. 
For $0< r \le \min\{1, (\phi(\lambda)/\lambda^2)^{-\frac{1}{2}}\}\theta R$, set
$$
\rho:= \max\{1, (\phi(\lambda)/\lambda^2)^{\frac{1}{2}}\} \,r \le \theta R\,.
$$
Then we have  $Q_r\subset Q^\lambda_{\rho}$ and by \eqref{nondegenerate_decayDu}
$$\begin{aligned}
|(D\bfu)_{Q_r} - \Gamma_0 |
& \le  \frac{\rho^{n+2}\lambda^2}{\phi(\lambda)r^{n+2}} \fint_{Q^\lambda_{\rho}} |D\bfu-\Gamma_0|\, \dz \\
& \le  c \frac{\max\{1, (\phi(\lambda)/\lambda^2)^{-\frac{n+2}{2}}\}}{(\phi(\lambda)/\lambda^2)} \phi^{-1}  \left( \Big(\frac{\rho}{R}\Big)^{\frac{p\beta}{q}} \phi(\lambda)\right)\\
& \le  c \frac{\max\{1, (\phi(\lambda)/\lambda^2)^{-\frac{n+2}{2}}\}}{\phi(\lambda)/\lambda^2}   \Big(\frac{r}{R}\max\{1, (\phi(\lambda)/\lambda^2)^{-\frac{1}{2}}\} \Big)^{\frac{p\beta}{q^2}}\lambda  \  \longrightarrow \ 0  
\end{aligned}$$
as $r\to 0$, which implies \eqref{Lem:nondegenerate2_result1}. 
Therefore, the proof is completed.
\end{proof}

%\newpage

\section{Degenerate regime}\label{Sec:degenerate}

We consider the \emph{degenerate} regime, which means that the average of the gradient of solution is relatively smaller than the relevant excess function, see for instance \eqref{Lem:degenerate_ass2}. In this regime, we apply the $\phi$-caloric approximation. We start by investigating regularity results for $\phi$-caloric maps. We refer to \cite{OSS23} for regularity results for  $\phi$-caloric maps.

Let ${\bf h}$ be a weak solution to 
\begin{equation}\label{system:phi}
\partial_t \bfh -   \div\left(\frac{\phi'(|D\bfh |)}{|D \bfh|} D \bfh\right) = \zero \quad \text{in }\  Q^\lambda_{R}.
\end{equation}
Then by \cite[Corollary 5.3]{OSS23} with the scaling argument used in the proof of Corollary~\ref{cor:high}, we have 
\begin{equation}\label{Lip_phicaloric}
\sup_{Q^\lambda_{R/2}}\phi (|D\bfh|) \le c \phi(\lambda)\,.
%\fint_{Q^\lambda_{R}} \phi(|D\bfh|)\,\d z .
\end{equation}
for some $c>0$ depending on $n,N,p,q$ and $\tilde c$ if $\bfh$ satisfies 
$$
\fint_{Q^\lambda_{R}} \phi (|D\bfh|) \, \d z \le \tilde{c} \phi(\lambda)\,.
$$
%\eqref{Lem:nondegenerate1_Ass1} with $(D\bfh)_{Q^{\lambda}_{r}}$ in place of $(D\bfu)_{Q^{\lambda}_{2r}}$. 
Moreover, from \cite[Section 6]{OSS23}, we have the following result concerned with $C^{1,\alpha}$-regularity for $\bfh$.

\begin{lemma}\label{Lem:phiLaplacesystem}
%\textcolor{blue}{Let ${\bf h}=(h^1,h^2,\dots,h^N) \in C_{\loc}(0,T; L^2_{\loc}(\Omega,\R^N)) \cap L^\phi_\loc(0,T;W^{1,\phi}_{\loc}(\Omega,\R^N))$ be a (local) weak solution to}
Let ${\bf h}$ be a weak solution to \eqref{system:phi}  in $Q^\lambda_{R}=Q^\lambda_{R}(z_0)$
%\comment{Jihoon: we don't write function spaces of solutions of systems in other places. Instead, we comment it in Preliminary.}  
%$$
%\partial_t \bfh -   \div\left(\frac{\phi'(|D\bfh |)}{|D \bfh|} D \bfh\right) =0 \quad \text{in }\  Q^\lambda_{R}=Q^\lambda_{R}(z_0)
%$$
with
\begin{equation}
\sup_{Q^{\lambda}_{R}} |D \bfh| \le    \lambda
\label{eq:bounded1}
\end{equation}
for some $\lambda>0$. Then there exist $\alpha_1\in(0,1)$ depending on $n,N,p,q$, a switching radius $r_s\in[0,R]$ and $\lambda_r>0$ for each $r\in (0,R]$, such that  
\begin{equation}\label{lambdar}
\lambda_r = \lambda_{r_s}  \  \ \text{if}\ \ r\in(0,r_s]
\quad\text{and}\quad
\Big(\frac{r}{R}\Big)^{\alpha_1}\lambda 
\le  \lambda_r
\le 2 \Big(\frac{r}{R}\Big)^{\alpha_1}\lambda 
\ \ \text{if}\ \ r\in (r_s,R]\,,
\end{equation}
\begin{equation}\label{phiDhLipr}
\sup_{Q^{\lambda_r}_{r}} |D \bfh| \le    \lambda_r 
\quad \text{for all}\ \ r\in (0,R]\,,
\end{equation}
and
\begin{equation}\label{degenerate_excess}
\fint_{Q^{\lambda_r}_{ r}} \phi_{|(D\bfh)_{ r}^{\lambda_r}|}(|D\bfh - (D\bfh)_{r}^{\lambda_r}|) \,\dz 
\le c \left(\frac{r}{r_{s}}\right)^{3/4}  \phi( \lambda_r)\, 
\quad \text{if}\ \ r\in(0,r_s].
\end{equation}
Moreover,  we also have 
\begin{equation}\label{aveDulowerbound}
|(D\bfh)^{\lambda_{r}}_{r}| \ge C_s^{-1}\lambda_{r} 
 \quad \text{and}\quad 
 \underset{Q^{\lambda_{r}}_{r}}{\mathrm{osc}}\, D\bfh \le c \left(\frac{r}{r_s}\right)^{3/4}\lambda_r  \quad \text{if}\ \ r\in(0,r_s] \,,
\end{equation}
for some $C_s>1$ and $c>0$ depending on $n,N,p,q$.
\end{lemma}

\begin{proof}
%We first note that the second inequality in \eqref{degenerate_excess} follows from \eqref{lambdar}, hence  for \eqref{degenerate_excess} we prove the first inequality only.  
For each $i=0,1,2,\dots$, we inductively define
$$
\lambda_0:=\lambda, \quad   \lambda_{i+1}:= \nu \lambda_i 
\quad \text{and}\quad
r_0:=R, \quad  r_{i+1} = \tilde \sigma r_i,
$$
where $\sigma\in(0,1)$ is from \cite[Proposition 6.2]{OSS23}, $\nu\in (0,1)$ is from \cite[Proposition 6.3]{OSS23} corresponding to the preceding $\sigma$, and 
\begin{equation}
\tilde \sigma:= \min\left\{\frac{\sigma p^\frac{1}{2}\nu^{\frac{q-2}{2}}}{2q^\frac{1}{2}},\nu^{\frac{4q}{3}}\right\}<\frac\sigma2\,. 
\label{eq:sigmatild}
\end{equation}
Note that we may assume that $\nu>1/2$.
Then we have
$$
\frac{r_{i+1}^2\lambda_{i+1}^2}{\varphi(\lambda_{i+1})} 
\leq \frac{q{\tilde \sigma}^2 \lambda_{i+1}}{\varphi'(\lambda_{i+1})} r_i^2 
\leq  \frac{q{\tilde \sigma}^2 \lambda_{i}}{\varphi'(\lambda_{i})\nu^{q-2}} r_i^2
 \leq \frac{q{\tilde \sigma}^2 }{p\nu^{q-2}}  \frac{ \lambda_{i}^2r_i^2}{\varphi(\lambda_{i})}   
 \leq \left(\frac{\sigma}{2}\right)^2 \frac{r_i^2\lambda_{i}^2}{\varphi(\lambda_{i})} 
 <  \frac{r_i^2\lambda_{i}^2}{\varphi(\lambda_{i})}\,,
$$
%\textcolor{blue}{\begin{equation*}
%\begin{split}
%\frac{r_{i+1}^2}{\varphi''(\lambda_{i+1})} \leq \frac{{\tilde \sigma}^2 \lambda_{i+1}}{(p-1)\varphi'(\lambda_{i+1})} r_i^2 \leq 
%\frac{{\tilde \sigma}^2 \lambda_{i}^{q-1}}{(p-1)\varphi'(\lambda_{i})\lambda_{i+1}^{q-2}} r_i^2 & \leq \frac{{\tilde \sigma}^2 \lambda_{i}^{q-2}(q-1)}{(p-1)\varphi''(\lambda_{i})\lambda_{i+1}^{q-2}} r_i^2 \\
%& =  \frac{{\tilde \sigma}^2 (q-1)}{(p-1)\nu^{q-2}} \frac{r_i^2}{\varphi''(\lambda_{i})} \\
%& = \left(\frac{\sigma}{2}\right)^2 \frac{r_i^2}{\varphi''(\lambda_{i})} \\
%& <  \frac{r_i^2}{\varphi''(\lambda_{i})}\,,
%\end{split}
%\end{equation*}}
%$$
%\textcolor{blue}{%\frac{r_{i+1}^2}{\phi''(\lambda_{i+1})} =  \frac{\tilde \sigma r_{i}}{\phi''(\nu \lambda_{i})} \le \frac{c_0^2 r_i^2}{(p-1)^2 \nu^{q-2}\phi''(\lambda_i)} \le  \frac{\left(\frac{\sigma}{2}r_i\right)^2}{\phi''(\lambda_i)} < \frac{r_i^2}{\phi''(\lambda_i)}, }
%$$
hence  $Q^{\lambda_{i+1}}_{r_{i+1}} \subset Q^{\lambda_{i}}_{\frac{\sigma}{2}r_{i}} \subset Q^{\lambda_{i}}_{r_{i}} $. Moreover, we have 
$$
\lambda_i=\nu^i\lambda_0 = \tilde \sigma^{i\frac{\ln \nu}{\ln \tilde\sigma}} \lambda_0 = \left(\frac{r_i}{R}\right)^{\alpha_1}\lambda_0\,,
\quad \text{where }\ 
\alpha_1:=\frac{\ln \nu}{\ln \tilde\sigma}\le \frac{3}{4q}\,,
$$
hence for every $r\in[r_{i+1},r_i]$,
\begin{equation}\label{inequality:lambdai}
\left(\frac{r}{R}\right)^{\alpha_1}\lambda_0 \le  \lambda_i =\left(\frac{r_{i+1}}{R}\right)^{\alpha_1}\tilde\sigma^{-\alpha_1}\lambda_0
= \left(\frac{r_{i+1}}{R}\right)^{\alpha_1}\nu^{-1}\lambda_0
\le 2 \left(\frac{r}{R}\right)^{\alpha_1}\lambda_0\,.
\end{equation}

Now we consider the inequality 
\begin{equation}\label{density} 
|\{D\bfh \le (1-\sigma)\lambda_i \}| \cap Q_{r_i}^{\lambda_i}| > \sigma |Q_{r_i}^{\lambda_i}|  
\quad \text{for }\ i=0,1,2, \dots\,.
\end{equation}
Then we have the following three cases:
\begin{itemize}
\item[(i)] \eqref{density} does not hold when $i=0$.
\item[(ii)] There exists $n_0\in \mathbb N$ such that \eqref{density} holds when $i=0,1,2,\dots, n_0-1$, but not when $i=n_0$.   
\item[(iii)] \eqref{density} holds for every $i$.
\end{itemize}

\smallskip

If the case (i) holds, then by \cite[Proposition 6.2]{OSS23}  we have for every $r\in(0,R]$
% ($R=r_0$),
\begin{equation*}
\underset{Q^{\lambda}_r}{\mathrm{osc}} \, D{\bf h} \le c \left(\frac{r}{R}\right)^{\frac{3}{4}}  \underset{Q^{\lambda}_{R}}{\mathrm{osc}} \, D{\bf h }\,,
\end{equation*}
whence, with \eqref{eq:bounded1}, 
$$\begin{aligned}
&\fint_{Q^{\lambda_0}_r} \phi_{|(D\bfh)_{r}^{\lambda_0}|}(|D\bfh - (D\bfh)_{r}^{\lambda_0}|) \,\dz \\
&\le c \fint_{Q^{\lambda_0}_r} \phi'(|(D\bfh)_{r}^{\lambda_0}|+ |D\bfh - (D\bfh)_{r}^{\lambda_0}|)|D\bfh - (D\bfh)_{r}^{\lambda_0}| \,\dz \\
&\le c \phi'(\lambda_0) \left(\frac{r}{R}\right)^{\frac{3}{4}}\lambda_0  \le c \phi(\lambda_0) \left(\frac{r}{R}\right)^{\frac{3}{4}},
\end{aligned}$$
which implies the inequalities \eqref{lambdar}, \eqref{phiDhLipr}, \eqref{degenerate_excess} and the second inequality in \eqref{aveDulowerbound} with $\lambda_r=\lambda_0=\lambda$ for all $r\in(0,R]$ and $r_s=R$.

If the case (ii) holds, then by \cite[Proposition 6.2]{OSS23} with $R=r_i$, $i=0,\dots, n_0-1$ and \cite[Proposition 6.3]{OSS23} with $R=r_{n_0}$, we have that
\begin{equation}\label{phiDhLipi}
\sup_{Q^{\lambda_i}_{r_i}} |D \bfh| \le  \lambda_i \quad \text{for all }\ i=0,1,2,\dots, n_0,
\end{equation}
and for every $r\in(0,r_{n_0}]$
\begin{equation*}
\underset{Q^{\lambda_{n_0}}_r}{\mathrm{osc}} \, D{\bf h} \le c \left(\frac{r}{r_{n_0}}\right)^{\frac{3}{4}}  \underset{Q^{\lambda}_{r_{n_0}}}{\mathrm{osc}} \, D{\bf h }\,,
\end{equation*}
hence
$$
\fint_{Q^{\lambda_{n_0}}_r} \phi_{|(D\bfh)_{r}^{\lambda_{n_0}}|}(|D\bfh - (D\bfh)_{r}^{\lambda_{n_0}}|) \,\dz \le c \phi(\lambda_{n_0}) \left(\frac{r}{r_{n_0}}\right)^{3/4}.
$$
Therefore, choosing
$$
\lambda_r=
\begin{cases}
\lambda_i  & \text{when }\ r\in(r_{i+1},r_i] \ \ \text{and} \ \ i=0,1,\dots,n_0-1\,,\\
\lambda_{n_0} & \text{when }\ r\in(0,r_{n_{i_0}}]\,,
\end{cases}
\quad \text{and}\quad r_s=r_{n_0}\,.
$$
we have \eqref{lambdar} (see \eqref{inequality:lambdai}) and \eqref{phiDhLipr}, \eqref{degenerate_excess} and and the second inequality in \eqref{aveDulowerbound}.

If the case (iii) holds,  we have \eqref{phiDhLipi} for all $i=0,1,2,\dots$, which implies the desired estimates with  $r_s=0$ and $\lambda_r=\lambda_i$ for every $r\in (r_{i+1},r_i ] $ and $i=0,1,2,\dots$.

Finally, we are left to prove the first inequality in \eqref{aveDulowerbound}. In case (iii) since $r_s=0$, there is nothing to prove. Then we consider  the cases (i) and (ii) where  $r_s>0$.
% \textcolor{red}{In these cases,  \eqref{aveDulowerbound} follows from \cite[Lemma 6.9]{OSS23} with $R=r_s$ and \cite[Lemma 6.8]{OSS23} with $R=\frac{r_s}{2}$, respectively.}
By \cite[Lemma 6.8]{OSS23} with $R=\frac{r_s}{2}$, together with \cite[Lemma 6.9]{OSS23} with $R=r_s$,  we have 
$$
|(D{\bfh})_{Q^{\lambda_{r_s}}_{\theta^j r_s/2}}| \ge \frac12  \lambda,
$$
for all $j\in \mathbb N_0$ and for some sufficiently small $ \theta \in(0,1)$ depending on $n,N,p,q$.  If $r\in (\frac{r_s}{2}, r_s]$ then 
$|(D{\bfh})_{Q^{\lambda_{r_s}}_{r}}| \ge \frac{1}{2^{n+2}}|(D{\bf u})_{Q^{\lambda_r}_{r_s/2}}| \ge \lambda / 2^{n+3}$. If $r\in (\frac{\theta^{j+1}r_s}{2},\frac{\theta^{j}r_s}{2}]$,
$
|(D{\bfh})_{Q^{\lambda_{r_s}}_{r}}| \ge  \frac{1}{\theta^{n+2}} |(D{\bfh})_{Q^{\lambda_{r_s}}_{\frac{\theta^{j+1}r_s}{2}}}| \ge \frac{\theta^{n+2}}{2} \lambda$. 
Thus, we obtain the first inequality in  \eqref{aveDulowerbound} with $C_1=\max\{2^{n+3},2\theta^{-(n+2)}\}$. 
\end{proof}

\begin{remark}\label{rmk:nested} We list some remarks about the previous lemma.
\begin{enumerate}
\item The numbers  $r_s$ and $\lambda_r$ may depend on $\bfh$ and the center $z_0$ of $Q^\lambda_R$.
\item
Recalling the constants $\sigma$ and $\tilde\sigma$ in the first part of the proof, one can see that if $r \in(0, \tilde\sigma R]$ then $Q^{\lambda_r}_r \subset Q^{\lambda}_{\frac{\sigma}{2}R}$.
\end{enumerate}
\end{remark}

The following lemma will be crucially used in the iteration process in Section~\ref{Sec:iteration}.

\begin{lemma} \label{Lem:degenerate}
Let $M_1\ge 1$, $\lambda\in(0,1]$, $ \chi,\chi_1\in(0,1]$ and $\alpha_1\in(0,\frac{3}{4q})$ be given in Lemma~\ref{Lem:phiLaplacesystem}. There exist large constants $K_1,C_1\ge 1$ depending on $n, N, p, q, \nu, L, \gamma, M_1$ and $\alpha_1$ such that the following holds: for every $\vartheta\in(0,C^{-1}_1]$ there exists large $K\ge 1$ depending on $n,N,p,q,\nu,L, \gamma, M_1,\alpha_1$ and $\vartheta$ and small $\epsilon_1\in(0,1)$ depending on $n, N, p, q, \nu, L,$ $\gamma, M_1, \alpha_1, \vartheta$ and $\chi$,  such that the following holds: 
if 
\begin{equation*}\label{Lem:degenerate_ass1}
|(D \bfu )^\lambda_r| \le M_1\lambda,
\end{equation*}
\begin{equation}\label{Lem:degenerate_ass2}
\chi  \phi(|(D \bfu)^\lambda_r|) \le \fint_{Q^\lambda_r} \phi_{|(D \bfu)^\lambda_r|} (|D \bfu-(D \bfu)^\lambda_r|) \, \dz 
\quad\text{or}\quad
|(D \bfu)^\lambda_r| \le \frac{\lambda}{K}
\end{equation}
and 
\begin{equation*}\label{Lem:degenerate_ass3}
\fint_{Q^\lambda_r} \phi_{|(D \bfu)^\lambda_r|} (|D \bfu-(D \bfu)^\lambda_r|) \, \dz \le \min\{\phi(\lambda),\epsilon_1\}\,,
\end{equation*}
then there exists 
\begin{equation*}\label{Lem:degenerate_lambda1}
\lambda_1\in[\vartheta^{\alpha_1}\lambda, C_1\lambda]
\end{equation*} 
such that 
%for every $\vartheta \in (0,C_1^{-1}]$, we have 
$Q^{\lambda_1}_{\vartheta r} \subset Q^\lambda_{\frac{\sigma}{2}r}$ with $\sigma\in(0,1)$ from Lemma~\ref{Lem:phiLaplacesystem},
\begin{equation}\label{Lem:degenerate_decay}
 \fint_{Q^{\lambda_1}_{\vartheta r}} \phi_{|(D \bfu)^{\lambda_1}_{\vartheta r}|} (|D \bfu-(D \bfu)^{\lambda_1}_{\vartheta r}|) \, \dz \le \phi(\lambda_1)  
\quad\text{and}\quad
|(D \bfu)^{\lambda_1}_{2\vartheta r}| \le \lambda_1\,.
\end{equation}
Additionally, for each $\alpha\in(0,\alpha_1)$ there exists small $\vartheta_1\in (0,C^{-1}_1]$ depending on $n, N, p, q, \nu,$ $L, M_1, \alpha_1$, $\alpha$ and $\chi_1$ such that for every $\vartheta\in (0,\vartheta_1]$  if
\begin{equation}\label{Lem:degenerate_ass4}
\chi_1  \phi(|(D \bfu)^{\lambda_1}_{\vartheta r}|)  \le  \fint_{Q^{\lambda_1}_{\vartheta r}} \phi_{|(D \bfu)^{\lambda_1}_{\vartheta r}|} (|D \bfu-(D \bfu)^{\lambda_1}_{\vartheta r}|) \, \dz   
\quad\text{or}\quad
|(D \bfu)^{\lambda_1}_{\vartheta r}| \le \frac{\lambda_1}{K_1}\,,
\end{equation}
then
\begin{equation}\label{Lem:degenerate_lambda1-1}
\lambda_1\le \vartheta^\alpha \lambda.
\end{equation}
\end{lemma}

\begin{proof} We first observe that, from  all the assumptions
% \eqref{Lem:degenerate_ass1}, \eqref{Lem:degenerate_ass2} and \eqref{Lem:degenerate_ass3}, 
and  \eqref{eq:approx}, we get
\begin{equation}\label{Lem:degenerate_pf1_0} \begin{aligned}
\fint_{Q^\lambda_r} \phi(|D\bfu|) \, \dz
& \le c \left( \fint_{Q^\lambda_r} \phi_{|(D\bfu)^\lambda_r|}(|D\bfu- (D\bfu)^\lambda_r|) \, \dz + \phi (|(D\bfu)^\lambda_r|)\right)  \\
& \le \begin{cases}
c (1+M_1^q) \phi(\lambda) \,,\\
c \left\{(1+\chi^{-1})\epsilon_1 + K^{-q}\right\}\,.  
\end{cases}
\end{aligned}\end{equation}
We divide the proof into two steps.

\textit{Step 1. ($\phi$-caloric approximation)} We show that  $\bfu$ is an almost $\phi$-caloric mapping.  Namely, for every $\bm\zeta\in C^\infty_0(Q^\lambda_r)$
\begin{equation}\label{ualmostphi}
\left|\fint_{Q^\lambda_r} \bfu \cdot \bm\zeta_t - \frac{\phi'(|D\bfu|)}{|D\bfu|} \langle D\bfu , D\bm\zeta \rangle \, \dz \right|  \le  c \epsilon_0 \left( \fint_{Q^\lambda_r} \phi(|D\bfu|) \,\dz  + \phi(\|D\bm\zeta \|_{\infty}) \right)\,,
\end{equation}
where $\epsilon_0>0$ is a sufficiently small constant determined later. From the weak form of \eqref{system1} we have
$$
\left|\fint_{Q^\lambda_r} \bfu \cdot \bm\zeta_t - \frac{\phi'(|D\bfu|)}{|D\bfu|} \langle D\bfu , D\bm\zeta \rangle \, \dz \right| = \left|\fint_{Q^\lambda_r}  \langle  \bA(D\bfu) - \frac{\phi'(|D\bfu|)}{|D\bfu|} D\bfu , D\bm\zeta \rangle \, \dz \right| \,.
$$
Choose $\delta=\delta(\epsilon_0)$ such that \eqref{almostphi} holds for $\epsilon=\epsilon_0$, and set
$E_1:=\{z\in Q^\lambda_r : |D\bfu(z)|\le \delta\}$ and $E_2=Q^\lambda_r\setminus E_1$. Then one has from the Young inequality \eqref{eq:young} and \eqref{eq:hok2.4} that
$$\begin{aligned}
\frac{1}{|Q^\lambda_r |}\left|\int_{E_1}  \langle  \bA(D\bfu) - \frac{\phi'(|D\bfu|)}{|D\bfu|} D\bfu , D\bm\zeta \rangle \, \dz \right| 
& \le \epsilon_0 \fint_{Q^\lambda_r} \phi'(|D\bfu|) \|D\bm\zeta \|_{\infty} \, \dz \\
& \le c \epsilon_0 \left( \fint_{Q^\lambda_r} \phi(|D\bfu|)  \, \dz  + \phi(\|D\bm\zeta \|_{\infty}) \right) \,.
\end{aligned}$$
On the other hand, by \eqref{growth}, the Young inequality \eqref{eq:young}, \eqref{characteristic} and \eqref{Lem:degenerate_pf1_0},
$$\begin{aligned}
&\frac{1}{|Q^\lambda_r |}\left|\int_{E_2}  \langle  \bA(D\bfu) - \frac{\phi'(|D\bfu|)}{|D\bfu|} D\bfu , D\bm\zeta \rangle \, \dz \right| \\
& \le \frac{1}{|Q^\lambda_r |}  \int_{E_2}  \left(|\bA(D\bfu)|+ \phi' (|D\bfu|)\right) \|D\bm\zeta\|_{\infty} \, \dz \\
& \le  \frac{c \|D\bm\zeta\|_{\infty}}{\delta(\epsilon_0)} \fint_{Q^\lambda_r}  \phi (|D\bfu|)  \, \dz  \\
& \le  c \phi^* \left( \frac{1}{\delta(\epsilon_0) \epsilon_0^{1/p}} \fint_{Q^\lambda_r}  \phi (|D\bfu|)  \, \dz \right)   +  \epsilon_0 \phi(\|D\bm\zeta \|_{\infty}) \\
& \le  \frac{c}{\delta(\epsilon_0)^q \epsilon_0^{q/p}} \phi^* \left(  \fint_{Q^\lambda_r}  \phi (|D\bfu|)  \, \dz \right)   + \epsilon_0 \phi(\|D\bm\zeta \|_{\infty})\\
& \le  \frac{c}{\delta(\epsilon_0)^q \epsilon_0^{q/p}} \psi\left((1+\chi^{-1})\epsilon_1 + K^{-q}\right)   \fint_{Q^\lambda_r}  \phi (|D\bfu|)  \, \dz   +  \epsilon_0 \phi(\|D\bm\zeta \|_{\infty}) \,,
\end{aligned}$$
where $\psi(t):= \frac{\phi^*(t)}{t}$. Therefore, choosing $\epsilon_1=\epsilon_1(\chi,\epsilon_0)$ sufficiently small and $K=K(\epsilon_0)$ sufficiently large, we obtain \eqref{ualmostphi}. 

Therefore, by Theorem~\ref{thm:phi-caloric} applied with ${\bf G}:=\bA(D\bfu)$ and $\gamma_1:=\frac12$, we have that  for a constant $\epsilon>0$  to be determined later, there exists $\epsilon_0=\epsilon_0(\epsilon)$ such that
\begin{equation}
\left(\fint_{Q^\lambda_r} |\bV(D\bfu)-\bV(D\bfh)|  \, \dz    \right)^{2} \le \epsilon \fint_{Q^\lambda_r} \phi(|D\bfu|)\, \dz \le c \epsilon \phi(\lambda)\,,
%\left(\fint_{Q^\lambda_r} |\bV(D\bfu)-\bV(D\bfh)|^{2\gamma_1}  \, \dz    \right)^{\frac{1}{\gamma_1}} \le \epsilon \fint_{Q^\lambda_r} \phi(|D\bfu|)\, \dz\,,
\label{eq:V(Du)-V(Dh)estim0}
\end{equation}
where $\bfh$ is the weak solution to 
$$\begin{cases}
\partial_t \bfh - \div \left(\frac{\phi'(|D\bfh|)}{|D\bfh|}D\bfh\right)=\zero \quad \text{in }\ Q^\lambda_r\,,\\
\bfh=\bfu \quad \text{in }\ \partial_{\mathrm p} Q^\lambda_r \,.
\end{cases}
$$
The existence and uniqueness of the solution $\bfh$  follows from the theory of monotone operators or by utilizing the Galerkin approximation method, see for instance \cite{Lions}. Here from higher integrability result in \eqref{eq:high1} with $\bm\ell={\bf 0}$ and the Lipschitz estimate for $\bfh$ in \eqref{Lip_phicaloric} we have
\begin{equation}
\left(\fint_{Q^\lambda_{r/2}} \phi(|D\bfu|)^{1+\sigma}\, \dz\right)^{\frac{1}{1+\sigma}} \le c \phi(\lambda) 
%\fint_{Q^\lambda_{r}} \phi(|D\bfu|)\, \dz  
\label{eq:highintestimate1}
\end{equation}
and 
\begin{equation}
\sup_{Q^\lambda_{r/2}} \phi(|D\bfh|) \le  c \phi(\lambda)  \,,\label{eq:DhDuestimate}
\end{equation}
since 
$$
 \fint_{Q^\lambda_{r}} \phi(|D\bfh|)\, \dz  \le  c \fint_{Q^\lambda_{r}} \phi(|D\bfu|)\, \dz \le c \phi(\lambda)
$$
by a standard energy estimate of the above system.
Moreover, with the change of shift formula \eqref{(5.4diekreu)} and the triangle inequality, we get
\begin{equation*}
\begin{split}
\fint_{Q^\lambda_{r/2}} \phi_{|D\bfh|}^{1+\sigma}(|D\bfu-D\bfh|)\,\mathrm{d}z & \leq c\left(c_\eta \fint_{Q^\lambda_{r/2}} \phi^{1+\sigma}(|D\bfu-D\bfh|)\,\mathrm{d}z + c\eta \sup_{Q^\lambda_{r/2}} \phi^{1+\sigma}(|D\bfh|)\right) \\
& \leq \tilde{c}\left(\fint_{Q^\lambda_{r/2}} \phi^{1+\sigma}(|D\bfu|)\,\mathrm{d}z + \sup_{Q^\lambda_{r/2}} \phi^{1+\sigma}(|D\bfh|)\right)\,,
\end{split}
\end{equation*}
whence
\begin{equation*}
\begin{split}
\left(\fint_{Q^\lambda_{r/2}} \phi_{|D\bfh|}^{1+\sigma}(|D\bfu-D\bfh|)\,\mathrm{d}z\right)^\frac{1}{1+\sigma}  \leq {c}\left(\left(\fint_{Q^\lambda_{r/2}} \phi^{1+\sigma}(|D\bfu|)\,\mathrm{d}z\right)^\frac{1}{1+\sigma} + \sup_{Q^\lambda_{r/2}} \phi(|D\bfh|)\right)\,.
\end{split}
\end{equation*}
This, together with \eqref{eq:highintestimate1} and \eqref{eq:DhDuestimate}, implies
\begin{equation}
\left(\fint_{Q^\lambda_{r/2}} |\bV(D\bfu)-\bV(D\bfh)|^{2(1+\sigma)}  \, \dz \right)^{\frac{1}{1+\sigma}} \leq c \phi(\lambda)\,.
%\fint_{Q^\lambda_{r}} \phi(|D\bfu|)\, \dz  \,.
\label{eq:V(Du)-V(Dh)estim}
\end{equation}
Therefore, by interpolation, choosing $\tau\in(0,1)$ such that $\frac{1-\tau}{2}+\tau(1+\sigma)=1$; i.e., $\tau=\frac1{1+2\sigma}$,
%$(1+\tau)\gamma_1+\tau(1+\sigma)=1$
 we have, with \eqref{eq:V(Du)-V(Dh)estim0}, \eqref{eq:V(Du)-V(Dh)estim} and the last inequality in \eqref{eq:DhDuestimate}, 
\begin{equation}\label{uhcomparison1}\begin{aligned}
&\fint_{Q^\lambda_{r/2}} |\bV(D\bfu)-\bV(D\bfh)|^{2}  \, \dz \\
&\le   \left(\fint_{Q^\lambda_{r/2}} |\bV(D\bfu)-\bV(D\bfh)|\, \dz \right)^{1-\tau}\left(\fint_{Q^\lambda_{r/2}} |\bV(D\bfu)-\bV(D\bfh)|^{2(1+\sigma)}  \, \dz \right)^{\tau} \\
&\le    c \epsilon^{\frac{1-\tau}{2}} \phi ( \lambda )\,.
\end{aligned}\end{equation}
Moreover,  by applying \eqref{phipqinequality}, with \eqref{eq:V(Du)-V(Dh)estim} and again by \eqref{eq:DhDuestimate}, we also have
$$\begin{aligned}
& \phi\bigg(\fint_{Q^\lambda_{r/2}} |D\bfu -D\bfh |  \, \dz\bigg) \le \fint_{Q^\lambda_{r/2}} \phi(|D\bfu-D\bfh|)  \, \dz  \\
& \le \epsilon^{\frac{1-\tau}{4}}   \fint_{Q^\lambda_{r/2}} \left[\phi(|D\bfu|) +\phi(|D\bfh|) \right]  \, \dz + c   \epsilon^{-\frac{1-\tau}{4}} \fint_{Q^\lambda_{r/2}} |\bV(D\bfu)-\bV(D\bfh)|^2 \, \dz \\
& \le  c \epsilon^{\frac{1-\tau}{4}} \phi(\lambda)\,,
\end{aligned}$$
hence
\begin{equation}\label{uhcomparison2}
\fint_{Q^\lambda_{r/2}} |D\bfu -D\bfh |  \, \dz \le c \epsilon^{\frac{1-\tau}{4q}} \lambda \,.
\end{equation}

\textit{Step 2.} Let $0< \theta \le 1/2$. Applying Lemma~\ref{Lem:phiLaplacesystem} with $R=r/2$ and writing $\lambda_{\theta}:=\lambda_{\theta r}$ for each $\theta\in(0,1/2]$ and $\theta_s:=\frac{r_s}{r}$, we have  that 
\begin{equation}\label{Lem:degenerate_pf1}
\lambda_\theta = \lambda_{\theta_s}  \  \ \text{if}\ \ \theta \in(0,\theta_s ]
\quad\text{and}\quad 
\theta^{\alpha_1}\lambda 
\le  \lambda_\theta
\le 2\theta^{\alpha_1}\lambda \  \ \text{if}\ \ \theta\in(\theta_s,1/2]\,,   
\end{equation}
\begin{equation}\label{Lem:degenerate_pf3}
\sup_{Q^{\lambda_\theta}_{\theta r}} |D\bfh| \le \lambda_\theta 
\quad\text{for all}\ \ \theta \in(0,1/2 ]\,,
\end{equation}
\begin{equation}\label{Lem:degenerate_pf4}
\fint_{Q^{\lambda_\theta}_{\theta r}} \phi_{|(D\bfh)^{\lambda_\theta}_{\theta r}|}(|D\bfh-(D\bfh)^{\lambda_\theta}_{\theta r}|)  \, \dz \le c  \Big(\frac{\theta}{\theta_{s}}\Big)^{3/4} \phi (\lambda_\theta) \,
\quad \text{if}\ \ \theta \in(0,\theta_s ]\,
\end{equation}
and 
\begin{equation}\label{Lem:degenerate_pf5}
|(D\bfh)^{\lambda_{\theta}}_{\theta r}| \ge \tfrac{1}{C_s}\lambda_{\theta} 
 \quad \text{and}\quad 
 \underset{Q^{\lambda_{\theta}}_{\theta r}}{\mathrm{osc}}\, D\bfh \le c \Big(\frac{\theta}{\theta_s}\Big)^{\alpha_1}\lambda_\theta 
 \quad \text{if}\ \ \theta \in(0,\theta_s ]\,.
\end{equation}
Moreover, we have from Remark~\ref{rmk:nested} (2) that $Q^{\lambda_\theta}_{\theta r}\subset Q^\lambda_{\frac{\sigma}{2}r}$.

For $0< \theta \le \tilde\sigma$ with $\tilde\sigma\in(0,1)$ as in \eqref{eq:sigmatild} and  a large constant  $C_0>1$ to be determined later, set
$$
C_1:= \max\{2C_0, 2C_0^{\frac{2-p}{2}} \tilde\sigma^{-1}\}.
$$
For $\vartheta \in (0,C^{-1}_1]$ we set
\begin{equation}\label{Lem:degenerate_pf2}
\theta := 2 C_0^{\frac{2-p}{2}}\vartheta  \in (0,\tilde \sigma]
\quad \text{and then}\quad 
\lambda_1 :=  C_0 \lambda_\theta \,.
%\begin{cases}
%\lambda_1 :=  C_0 \lambda_\theta 
%\quad\text{and}\quad 
%\vartheta := \tfrac12C_0^{\frac{p-2}{2}}\theta
%\quad \text{if }\ \theta>\theta_s\,,\\
%\lambda_1 :=  \lambda_\theta 
%\quad\text{and}\quad 
%\vartheta := \tfrac{1}{2}C_0^{-1}\theta
%\quad \text{if }\ \theta\le\theta_s\,.
%\end{cases}
\end{equation}
Note that $Q^{\lambda_1}_{2\vartheta r} \subset Q^{\lambda_\theta}_{\theta r}\subset Q^\lambda_{\frac{\sigma}{2}r}$ since $\frac{(2\vartheta)^2\lambda_1^2}{\phi(\lambda_1)}\le \frac{4C_0^{2-p}\vartheta^2\lambda_{\theta}^2}{\phi(\lambda_\theta)} = \frac{\theta^2\lambda_\theta^2}{\phi(\lambda_\theta)}$ 
%when $\theta>\theta_s$, 
and that by \eqref{Lem:degenerate_pf1}
$$
\vartheta^{\alpha_1}\lambda  \le \theta^{\alpha_1} \lambda \le \lambda_{\theta} \le  \lambda_1 \le 2\theta^{\alpha_1} C_0 \lambda  \le 2C_0 \lambda \le C_1\lambda \,.
$$
%{\color{blue}hence $\lambda_1$ satisfies \eqref{Lem:degenerate_lambda1}.}
We then prove \eqref{Lem:degenerate_decay} for $\vartheta\in (0,C_1^{-1}]$ with choosing $\epsilon$ and $C_0$. Note that
using \eqref{monotonicity1}, \eqref{Vintequivalent} and \eqref{Lem:degenerate_pf1}, we have
$$\begin{aligned}
&\fint_{Q^{\lambda_1}_{\vartheta r}} \phi_{|(D\bfu)^{\lambda_1}_{\vartheta r}|}(|D\bfu- (D\bfu)^{\lambda_1}_{\vartheta r}|)  \, \dz
 \le  c \fint_{Q^{\lambda_1}_{\vartheta r}} |\bV(D\bfu)- \bV((D\bfu)^{\lambda_1}_{\vartheta r})|^{2}  \, \dz \\
 & \le  c \fint_{Q^{\lambda_1}_{\vartheta r}} |\bV(D\bfu)- (\bV(D\bfu))^{\lambda_1}_{\vartheta r}|^{2}  \, \dz 
  \le  c \fint_{Q^{\lambda_1}_{\vartheta r}} |\bV(D\bfu)- (\bV(D\bfh))^{\lambda_1}_{\vartheta r}|^{2}  \, \dz \\
  & \le  c \fint_{Q^{\lambda_1}_{\vartheta r}} |\bV(D\bfu)- \bV(D\bfh)|^{2}  \, \dz  + c \fint_{Q^{\lambda_1}_{\vartheta r}} |\bV(D\bfh)- (\bV(D\bfh))^{\lambda_1}_{\vartheta r}|^{2}  \, \dz \\
& \le  c 
%\frac{\phi(\lambda_1)\lambda_\theta^2 \vartheta^{n+2}}{\phi(\lambda_\theta)\lambda_1^2\theta^{n+2}}\theta^{-(n+2)-(q-2)\alpha_1}
\frac{\varphi(\lambda_{1})\lambda^2\vartheta^{-(n+2)}}{\varphi(\lambda)\lambda^2_{1}}
\fint_{Q^{\lambda}_{r/2}} |\bV(D\bfu)- \bV(D\bfh)|^{2}  \, \dz +  c \sup_{Q^{\lambda_\theta}_{\theta r}} \phi(|D \bfh|) \,.
\end{aligned}$$
For the second term on the right hand side, by \eqref{Lem:degenerate_pf3} and \eqref{Lem:degenerate_pf2}, we have 
$$
 \sup_{Q^{\lambda_\theta}_{\theta r}} \phi(|D \bfh|) \le \phi (C_0^{-1}\lambda_1) \le C_0^{-p}\phi(\lambda_1)\,.
$$
As for the first term on the right hand side,
by  \eqref{Lem:degenerate_pf2} and \eqref{uhcomparison1}, we have 
$$\begin{aligned}
&
\frac{\varphi(\lambda_{1})\lambda^2\vartheta^{-(n+2)}}{\varphi(\lambda)\lambda^2_{1}}
\fint_{Q^{\lambda}_{r/2}} |\bV(D\bfu)- \bV(D\bfh)|^{2}  \, \dz\\
& \le  c \vartheta^{-(n+2+2\alpha_1)}
%C_0^{q-2+\frac{(2-p)(2n+4+(q-2)\alpha_1)}{2}}\vartheta^{-(n+2)-(q-2)\alpha_1}
\fint_{Q^{\lambda}_{r/2}} |\bV(D\bfu)- \bV(D\bfh)|^{2}  \, \dz \\
& \le  c \vartheta^{-(n+2)-(q+2)\alpha_1}
%C_0^{q-2+\frac{(2-p)(2n+4+(q-2)\alpha_1)}{2}-p}\vartheta^{-(n+2)-(q-1)\alpha_1}  
\epsilon^{\frac{1-\tau}{2}} \phi(\lambda_1)\,.
\end{aligned}$$
%and if $\theta \le \theta_1$,
%$$\begin{aligned}
%&\frac{\phi(\lambda_1)\lambda_\theta^2\vartheta^{n+2}}{\phi(\lambda_\theta)\lambda_1^2\theta^{n+2}}\theta^{-(n+2)-(q-2)\alpha_1}\fint_{Q^{\lambda}_{r}} |\bV(D\bfu)- \bV(D\bfh)|^{2}  \, \dz\\
%&\le c C_0^{n+2-p}\theta^{-(n+2)-(q-1)\alpha_1}  \epsilon^{\frac{1-\tau}{2}} \phi(\lambda_1)\,.
%\end{aligned}$$
Therefore, choosing $C_0$ large and then $\epsilon=\epsilon(\vartheta)$ small, we obtain the first estimate in \eqref{Lem:degenerate_decay}.
Moreover, in a similar way with Jensen's inequality, 
%using \eqref{Lem:degenerate_pf2} and the fact that $Q^{\lambda_1}_{2\vartheta r}\subset Q^{\lambda_\theta}_{\theta r}$, 
we also have
$$\begin{aligned}
\phi\left(|(D\bfu)^{\lambda_1}_{2\vartheta r}|\right) 
&\le \fint_{Q^{\lambda_1}_{2\vartheta r}} \phi(|D\bfu|)\, \dz \le  c \fint_{Q^{\lambda_1}_{2\vartheta r}} |\bV(D\bfu)|^2\, \dz \\
&\le c \fint_{Q^{\lambda_1}_{2\vartheta r}} |\bV(D\bfu)- \bV(D\bfh)|^{2}  \, \dz  + c \fint_{Q^{\lambda_1}_{2\vartheta r}} | \bV(D\bfh)|^{2}  \, \dz \\
&\le  c \vartheta^{-(n+2+2\alpha_1+q\alpha_1)}  \epsilon^{\frac{1-\tau}{2}} \phi(\lambda_1)+  c C_0^{-p} \phi(\lambda_1) \le  \phi(\lambda_1)\,,
\end{aligned}$$
provided that $C_0$ is large enough and then $\epsilon=\epsilon(\vartheta)$ is small enough, which together with the monotonicity of $\phi$ implies the second estimate in  \eqref{Lem:degenerate_decay}. 

 We next prove \eqref{Lem:degenerate_lambda1-1} under the condition \eqref{Lem:degenerate_ass4}, by choosing $\vartheta$ sufficiently small. Recall the definition of $\theta$ and we distinguish between two cases.
 If $\theta \in [\theta_s , 1/2)$, by \eqref{Lem:degenerate_pf1} and \eqref{Lem:degenerate_pf2}, we have that for each $\alpha\in(0,\alpha_1)$
$$
\lambda_1 = C_0 \lambda_\theta \le 2 C_0  \theta^{\alpha_1}\lambda =    2 C_0  \Big(2C_0^{\frac{2-p}{2}}\Big)^{\alpha_1} \vartheta^{\alpha_1-\alpha}   \vartheta^{\alpha}\lambda \le \vartheta^{\alpha}\lambda,
$$
where we  
choose $\vartheta_1$ small to satisfy the last inequality for $\vartheta \in (0,\vartheta_1]$.
%chose $\vartheta$  small to satisfy the last inequality. 
Therefore, we obtain \eqref{Lem:degenerate_lambda1-1} without considering \eqref{Lem:degenerate_ass4}.
On the other hand, let $\theta \in (0,\theta_s)$. Note that, as $\vartheta< \theta <\theta_s$, there hold $\lambda_{\theta_s}=\lambda_\theta=\lambda_{\vartheta}$ and $Q^{\lambda_1}_{\vartheta r} \cup  Q^{\lambda_\vartheta}_{\vartheta r} \subset Q^{\lambda_\theta}_{\theta r}$. Hence, by \eqref{phipqinequality} and \eqref{uhcomparison2}
$$\begin{aligned}
& | (D\bfu)^{\lambda_1}_{\vartheta r}- (D\bfh)^{\lambda_\theta}_{\theta r}| \le  c \fint_{Q^{\lambda_\theta}_{\theta r}} | D \bfu- (D\bfh)^{\lambda_\theta}_{\theta r} |\, \dz \\
& \le c \fint_{Q^{\lambda_\theta}_{\theta r}} | D \bfu-  D\bfh |\, \dz +c \fint_{Q^{\lambda_\theta}_{\theta r}} | D \bfh- (D\bfh)^{\lambda_\theta}_{\theta r} |\, \dz\\
& \le c \theta^{-n-2-\alpha_1(q-2)} \fint_{Q^{\lambda}_{r/2}} | D \bfu-  D\bfh |\, \dz +c \left(\frac{\theta}{\theta_s}\right)^{\alpha_1} \lambda_\theta\\
& \le c \vartheta^{-n-2-\alpha_1(q-2)} \epsilon^{\frac{1-\tau}{4q}} \lambda +c \left(\frac{\theta}{\theta_s}\right)^{\alpha_1} \lambda_\theta\\
& \le c \vartheta^{-\kappa_1} \epsilon^{\kappa_2} \lambda_1 +c \left(\frac{\theta}{\theta_s}\right)^{\alpha_1} \lambda_1\,,
\end{aligned}$$
%$$\begin{aligned}
%& | (D\bfu)^{\lambda_1}_{\vartheta r}- (D\bfh)^{\lambda_\vartheta}_{\vartheta r}| 
%\le  c \fint_{Q^{\lambda_1}_{\vartheta r}} | D \bfu- (D\bfh)^{\lambda_\vartheta}_{\vartheta r} |\, \dz \\
%& \le c \fint_{Q^{\lambda_1}_{\vartheta r}} | D \bfu-  D\bfh |\, \dz + c \fint_{Q^{\lambda_1}_{\vartheta r}} | D \bfh- (D\bfh)^{\lambda_\vartheta}_{\vartheta r} |\, \dz  \\
%& \le c \fint_{Q^{\lambda_\theta}_{\theta r}} | D \bfu-  D\bfh |\, \dz +c \fint_{Q^{\lambda_1}_{\vartheta r}} \fint_{Q^{\lambda_\vartheta}_{\vartheta r}}  | D \bfh (z)- D\bfh(\tilde z) |\, \dz \d \tilde z \\
%& \le c \fint_{Q^{\lambda_\theta}_{\theta r}} | D \bfu-  D\bfh |\, \dz +c \fint_{Q^{\lambda_\theta}_{\theta r}} \fint_{Q^{\lambda_\theta}_{\theta r}}  | D \bfh (z)- D\bfh(\tilde z) |\, \dz \d \tilde z \\
%& \le c \fint_{Q^{\lambda_\theta}_{\theta r}} | D \bfu-  D\bfh |\, \dz +c \fint_{Q^{\lambda_\theta}_{\theta r}}   | D \bfh - (D\bfh)_{Q^{\lambda_\theta}_{\theta r}} |\, \dz  \\
%& \le c \theta^{-n-2-\alpha_1(q-2)} \fint_{Q^{\lambda}_{r/2}} | D \bfu-  D\bfh |\, \dz +c \left(\frac{\theta}{\theta_s}\right)^{\alpha_1} \lambda_\theta\\
%& \le c \vartheta^{-n-2-\alpha_1(q-2)} \epsilon^{\frac{1-\tau}{4q}}  \lambda +c \left(\frac{\theta}{\theta_s}\right)^{\alpha_1} \lambda_\theta\\
%& \le c \vartheta^{-\kappa_1} \epsilon^{\kappa_2} \lambda_1 +c \left(\frac{\theta}{\theta_s}\right)^{\alpha_1} \lambda_1\,,
%\end{aligned}$$
where 
%$\kappa_1= n-2-\alpha_1(q-2)-\alpha_1$ 
$\kappa_1= n+2+\alpha_1(q-1)$
and $\kappa_2=\frac{1-\tau}{4q} $. Then using \eqref{Lem:degenerate_pf5}  and choosing $\epsilon=\epsilon(\vartheta)$ sufficiently small we have 
$$\begin{aligned}
|(D\bfu)^{\lambda_1}_{\vartheta r}| &\ge |(D\bfh)^{\lambda_\theta}_{\theta r}|  - | (D\bfu)^{\lambda_1}_{\vartheta r}- (D\bfh)^{\lambda_\theta}_{\theta r}| \\
&\ge \left(C_0^{-1}C_s^{-1}  - c \vartheta^{-\kappa_1} \epsilon^{\kappa_2}  - c \left(\frac{\theta}{\theta_s}\right)^{\alpha_1} \right)\lambda_1  \\
&\ge \left(\frac{1}{2C_0C_s}   - c \left(\frac{\theta}{\theta_s}\right)^{\alpha_1} \right)\lambda_1\,.
\end{aligned}$$
Moreover, if  the first condition in \eqref{Lem:degenerate_ass4} holds, then by \eqref{monotonicity1} and \eqref{Lem:degenerate_pf4} we have 
$$ \begin{aligned}
\phi(|(D\bfu)^{\lambda_1}_{\vartheta r}|)  & \le \chi_1^{-1} \fint_{Q^{\lambda_1}_{\vartheta r}} \phi_{|(D \bfu)^{\lambda_1}_{\vartheta r}|} (|D \bfu-(D \bfu)^{\lambda_1}_{\vartheta r}|) \, \dz 
\le c \chi_1^{-1} \fint_{Q^{\lambda_1}_{\vartheta r}} | \bV(D\bfu) - ( \bV(D \bfu))^{\lambda_1}_{\vartheta r}|^2 \, \dz \\
&\le c \chi_1^{-1} \fint_{Q^{\lambda_\theta}_{\theta r}} | \bV(D\bfu) - ( \bV(D \bfu))^{\lambda_\theta}_{\theta r}|^2 \, \dz  
\le c  \chi_1^{-1}  \Big(\frac{\theta}{\theta_{s}}\Big)^{3/4} \phi (\lambda_\theta)\,,
\end{aligned}$$
which implies that
$$\begin{aligned}
& \left(\frac{1}{2C_0C_s}  - c \left(\frac{\theta}{\theta_s}\right)^{\alpha_1} \right)\lambda_1  \le c \chi_{1}^{-\frac{1}{p}}\left(\frac{\theta}{\theta_{s}}\right)^{\frac{3}{4q}} \lambda_\theta  \le c\chi_{1}^{-\frac{1}{p}} \left(\frac{\theta}{\theta_s}\right)^{\alpha_1}\lambda_1,
\end{aligned}$$
hence
\begin{equation}\label{thetastheta}
\theta_s^{\alpha_1} \le c\chi_{1}^{-\frac{1}{p}} \theta^{\alpha_1} \,.
\end{equation}
If  the second condition in \eqref{Lem:degenerate_ass4} holds, then choosing $K_1\ge 4C_0C_s$,
$$
\left(\frac{1}{2C_0C_s}  - c \left(\frac{\theta}{\theta_s}\right)^{\alpha_1} \right)\lambda_1 \le \frac{\lambda_1}{K_1} \le \frac{\lambda_1}{4C_0C_s} \,,
$$
which implies \eqref{thetastheta} again. Consequently, we have
$$
\lambda_1  = C_0\lambda_\theta = 2 C_0 \theta_s^{\alpha_1}\lambda \le c {\color{blue}\chi_{1}^{-\frac{1}{p}}}  \theta^{\alpha_1}\lambda  \le c {\color{blue}\chi_{1}^{-\frac{1}{p}}} \vartheta^{\alpha_1-\alpha}\vartheta^\alpha \lambda \le \vartheta^\alpha \lambda\,,
$$
Therefore, choosing $\vartheta_1$ sufficiently small depending on $\chi_1$, we get
$$
\lambda_1 \le \vartheta_1^\alpha \lambda\,,
$$
whenever $\vartheta\in (0,\vartheta_1]$.
\end{proof}

%%%%%%%%%%%%%%%%%%%%%%%%%%%%%%%%%%%%%%%%%%%%%%%%%%%%%%%%%%%%%%%%%%%%%%%%%%%%%%%%%%%%%%%%%%%%%%%%%%%%%%%%%%%%%%%%%%%%%%%%%%%%%%%%%%%%%%%%%%%%%%%%%%%%%%%%%%%%%%%%%%%%%%%%%%%%%%%%%%%%%%%%%%%%%%%%%%%%%%%%%%%%%%%%%%%%%%%%%%%%%%%%%%%%%%%%%%%%%%%%%%%%%%%%%%%%%%%%%%%%%%%%%%%%%%%%%%%%

%\newpage

\section{Iteration and Proof of Theorem~\ref{mainthm}}\label{Sec:iteration}

With the following result we will combine the degenerate and the nondegenerate regimes by an inductive iteration scheme. Roughly speaking, as long as on an iteration scale the degenerate case holds, we shall apply on this scale Lemma \ref{Lem:degenerate}. On the other hand, when on an iteration scale the nondegenerate
regime occurs we can apply Lemma \ref{Lem:nondegenerate2} which provides a suitable excess decay estimate. Either this happens at a certain scale $\vartheta^mR$, $m<\infty$, or we go on by iterating
the excess improvement from the degenerate case on each scale (i.e, $m=\infty$), thus obtaining the desired excess decay estimate.

\begin{lemma}\label{Lem:iteration1}
Let $\bfu$ be a weak solution to \eqref{mainthm}, $M_0\ge 1$ and $\alpha\in(0,\alpha_1)$. There exist small $\vartheta,\epsilon_2\in (0,1)$ and large $K_2\ge 1$ depending on $n,N,p,q,\nu,L,\alpha_1$ and $M_0$ such that the following holds: 
suppose $Q_{2R}(z_0)\Subset\Omega_T$,
\begin{equation}\label{Lem:iteration1_ass1}
|(D\bfu)_{Q_{2R}(z_0)}| \le M_0\,,
\end{equation}
and
\begin{equation}\label{Lem:iteration1_ass2}
\fint_{Q_{R}(z_0)} \phi_{|(D \bfu)_{Q_{R}(z_0)}|} (|D \bfu-(D \bfu)_{Q_{R}(z_0)}|) \, \dz \le \epsilon_2\,. 
\end{equation}
Then the limit
\begin{equation*}\label{Lem:iteration1_result1}
\Gamma_{z_0} :=\lim_{r\to 0^+} (D\bfu)_{Q_{r}(z_0)}
\end{equation*}
exists and there exist $m\in \mathbb N_0 \cup\{\infty\}$ and positive numbers $\lambda_j$ with $ j\in\{0,1,2,\dots,m\}$ such that $\lambda_0 =1$,
\begin{equation}\label{Lem:iteration1_result2}
\vartheta^{\alpha_1} \lambda_{j-1}\le \lambda_j  \le \vartheta^{\alpha}  \lambda_{j-1}  \ \text{ for } 1\le  j < m\,,\ \ 
\vartheta^{\alpha_1}\lambda_{m-1} \le \lambda_m  \le C_1\lambda_{m-1} \ \text{ for } 0< m<\infty\,,
\end{equation}
\begin{equation}\label{Lem:iteration1_result4}
Q^{\lambda_j}_{\vartheta^j R}(z_0) \subset Q^{\lambda_{j-1}}_{\sigma_1\vartheta^{j-1}R}(z_0) \qquad \text{for } \ 1\le j \le m\ \text{ with }\ j<\infty\,,  
\end{equation}
\begin{equation}\label{Lem:iteration1_result3}
\frac{\lambda_m}{2K_2} \le | \Gamma_{z_0}| \le 2K_2 \lambda_m\ \text{ if }\ m<\infty 
\quad \text{and} \quad 
\Gamma_{z_0} = {\bf 0}\ \text{ if }\ m=\infty  \,,
\end{equation}
\begin{equation}\label{Lem:iteration1_result5}
\fint_{Q^{\lambda_j}_{\vartheta^{j}R}(z_0)} \phi (|D \bfu-\Gamma_{z_0}|) \, \dz 
\le c \phi(\lambda_j) \qquad \text{for }\ 0\le j \le m  \ \text{ with }\ j<\infty\,,
\end{equation}
and, if $m<\infty$,
\begin{equation}\label{Lem:iteration1_result6}
\fint_{Q^{\lambda_m}_{r}(z_0)} \phi (|D \bfu-\Gamma_{z_0}|) \, \dz 
\le c \left(\frac{r}{\vartheta^m R}\right)^{\alpha_2}\phi(\lambda_m) \quad \text{for all }\ 0<r \le \vartheta^mR \,,
\end{equation}
where $\alpha_2\in(0,1)$ depends on $p$ and $q$.
\end{lemma}
%\comment{\color{blue}Jihoon: I revised the proof.}
\begin{proof}
%For simplicity, we shall omit to write the center $z_0$.
\textit{Step 1. (Choice of parameters)} 
Without loss of generality, we may assume that $z_0=0$. We start by fixing the parameters. Let $M_1:=2^{n+2}M_0$. First, we fix $K_1$ and $C_1$ as in Lemma~\ref{Lem:degenerate}, and then choose $\delta_1=\delta_1(K_1)$ in Lemma~\ref{Lem:nondegenerate2} with $K_0=\max\{M_1,2^{n+2}K_1\}$, and set $\chi_1:=\delta_1(K_1)$  in Lemma~\ref{Lem:degenerate}.  
Then we next determine $\vartheta_1$ as in Lemma~\ref{Lem:degenerate}, and then $K$ as in Lemma~\ref{Lem:degenerate} with $\vartheta\le \vartheta_1$.  Then choose $\delta_1=\delta_1(K)$ in Lemma~\ref{Lem:nondegenerate2} with $K_0=\max\{M_1, 2^{n+2}K\}$, and set $\chi:=\delta_1(K)$ in Lemma~\ref{Lem:degenerate}. We then determine $\epsilon_1$ as in  Lemma~\ref{Lem:degenerate}. 
Note that there exists $j_*\in \mathbb N$ such that 
\begin{equation}\label{j*}
\vartheta^{\alpha p( j_*+1)} \le \epsilon_1\,.
\end{equation}
With this $j_*$, we determine 
\begin{equation}\label{epsilon2}
   \epsilon_2 =  C_2^{-1}\vartheta^{(n+4-p)(j_*+1)} \epsilon_1\,,
\end{equation}
where $C_2>1$ is a large constant to be determined in \eqref{Lem:iteration1_pf7}. 
Finally, set
$$K_2=\max\{M_1, 2^{n+2}K_1, 2^{n+2}K\}\,.$$

We next choose $\lambda_j$ inductively. Set $\lambda_0=1$. For some $j\in \mathbb N_0$ and $\lambda_j\in (0,1]$, we consider the following condition:
\begin{equation}\label{induction2}
\chi |(D\bfu)^{\lambda_j}_{\vartheta^{j}R}| \le \fint_{Q^{\lambda_j}_{\vartheta^{j}R}} \phi_{|(D \bfu)^{\lambda_j}_{\vartheta^{j}R}|} (|D \bfu-(D \bfu)^{\lambda_j}_{\vartheta^{j}R}|) \, \dz 
\quad\text{or}\quad
|(D \bfu)^{\lambda_j}_{\vartheta^{j}R}| \le \frac{\lambda_j}{K}\,.
\end{equation}
%If this condition holds, then by Lemma~\ref{Lem:degenerate} with $\lambda_j= \lambda$
If the first three conditions in Lemma~\ref{Lem:degenerate} hold with $\lambda_j= \lambda$ and $r=\vartheta^jR$,
there exists $\lambda_{j+1}\in[\vartheta^{\alpha_1}\lambda_j, C_1\lambda_j]$  such that $Q^{\lambda_{j+1}}_{\vartheta^{j+1}R} \subset Q^{\lambda_j}_{\frac{\sigma}{2}\vartheta^j R}$,
\begin{equation}\label{pf11}
 \fint_{Q^{\lambda_{j+1}}_{\vartheta^{j+1} R}} \phi_{|(D \bfu)^{\lambda_{j+1}}_{\vartheta^{j+1} R}|} (|D \bfu-(D \bfu)^{\lambda_{j+1}}_{\vartheta^{j+1} R}|) \, \dz \le \phi(\lambda_{j+1})  
\quad\text{and}\quad
|(D \bfu)^{\lambda_{j+1}}_{2\vartheta^{j+1} R}| \le \lambda_{j+1} \,.
\end{equation}
Then we have two cases:
\begin{equation}\label{induction3}
\chi_1 |(D\bfu)^{\lambda_{j+1}}_{\vartheta^{j+1}R}| \le \fint_{Q^{\lambda_{j+1}}_{\vartheta^{j+1}R}} \phi_{|(D \bfu)^{\lambda_{j+1}}_{\vartheta^{j+1}R}|} (|D \bfu-(D \bfu)^{\lambda_{j+1}}_{\vartheta^{j+1}R}|) \, \dz 
\quad\text{or}\quad
|(D \bfu)^{\lambda_{j+1}}_{\vartheta^{j+1}R}| \le \frac{\lambda_{j+1}}{K_1}\,,
\end{equation}
and the other case. 
Let $m_1$ be the unique number in $\mathbb{N}_0\cup\{\infty\}$ such that (7.10) holds for all $0\le j<m_1$, but does not hold for $j=m_1<\infty$.
%Set $\widetilde{\mathbb N}:=\{j\in \mathbb N_0: \eqref{induction2}\ \text{with $j$ does not holds.}\}$ and 
%$m_1\in \mathbb N_0 \cup \{\infty\}$ such that
%$$\begin{cases}
%m_1= \min \widetilde{\mathbb N} \quad \text{if }\ \widetilde{\mathbb N} \neq \emptyset \,,\\ 
%m_1 = \infty \quad \text{if }\ \widetilde{\mathbb N} = \emptyset\,.
%\end{cases}$$

\textit{Step 2. (Nondegenerate decay)}  
If $m_1=0$, then the lemma with $m=0$ follows directly from Lemma~\ref{Lem:nondegenerate2} with $R=r$, $K_0=2^{n+2}\max\{M_1, K_1\}$ and $\delta_1=\delta_1(K)$. 

We next suppose $m_1\in \mathbb N \cup \{\infty\}$. If  there exists $m_0\in \mathbb N$ with $m_0\le m_1+1$ such that \eqref{induction3} holds  for all $ j < m_0-1$ but not $j= m_0-1$  then we choose $m=m_0$. On the other hand, if $m_1\in \mathbb N$ and \eqref{induction3} holds for all $ j  \le m_1 -1$, then we choose $m=m_1$.  We note that if all the assumptions of Lemma~\ref{Lem:degenerate}
%, with $\theta^jr$ in place of $r$, 
with $\lambda=\lambda_j$ and $r=\vartheta^j R$
hold and $\lambda_j\le \vartheta^{\alpha j}$, then $\lambda_{j+1} \le \vartheta^{\alpha(j+1)}$, 
\begin{equation}\label{lambdaj+1}
 \fint_{Q^{\lambda_{j+1}}_{\vartheta^{j+1} R}} \phi_{|(D \bfu)^{\lambda_{j+1}}_{\vartheta^{j+1} R}|} (|D \bfu-(D \bfu)^{\lambda_{j+1}}_{\vartheta^{j+1} R}|) \, \dz \le  \phi(\lambda_{j+1}) \le \phi(\vartheta^{\alpha(j+1)}) \le \vartheta^{\alpha p( j+1)}   
 \end{equation}
and 
 $$
|(D \bfu)^{\lambda_{j+1}}_{\vartheta^{j+1} R}| \le 2^{n+2}  |(D \bfu)^{\lambda_{j+1}}_{2\vartheta^{j+1} R}| \le 2^{n+2}  \lambda_{j+1} \le M_1\lambda_{j+1} \,.
$$
Moreover if $j\ge j_*$, the estimate \eqref{lambdaj+1} together with \eqref{j*} implies
$$ 
\fint_{Q^{\lambda_{j+1}}_{\vartheta^{j+1} R}} \phi_{|(D \bfu)^{\lambda_{j+1}}_{\vartheta^{j+1} R}|} (|D \bfu-(D \bfu)^{\lambda_{j+1}}_{\vartheta^{j+1} R}|) \, \dz  \le  \epsilon_1\,, 
$$
while if $j<j_*$, then by \eqref{epsilon2}, taking into account \eqref{monotonicity1} and \eqref{Vintequivalent},
\begin{equation}\label{Lem:iteration1_pf7} 
\begin{aligned}
&\fint_{Q^{\lambda_{j+1}}_{\vartheta^{j+1} R}} \phi_{|(D \bfu)^{\lambda_{j+1}}_{\vartheta^{j+1} R}|}  (|D \bfu-(D \bfu)^{\lambda_{j+1}}_{\vartheta^{j+1} R}|) \, \dz  \\
& \le c \fint_{Q^{\lambda_{j+1}}_{\vartheta^{j+1} R}} |\bV(D\bfu) - (\bV(D\bfu))^{\lambda_{j+1}}_{\vartheta^{j+1}R}|^2\,\dz \\
& \le c\vartheta^{-(n+4-p)(j+1)} \fint_{Q_{R}} |\bV(D\bfu) - (\bV(D\bfu))_{R}|^2\,\dz \\
& \le c\vartheta^{-(n+4-p)(j_*+1)} \fint_{Q_{R}} \phi_{|(D \bfu)_{R}|} (|D \bfu-(D \bfu)_{R}|) \, \dz  \\
&\le  C_2\vartheta^{-(n+4-p)(j_*+1)} \epsilon_2 \le \epsilon_1\,.
\end{aligned}\end{equation}
Therefore, we can inductively apply Lemma~\ref{Lem:degenerate} 
%with $r=\theta^jR$ 
with $\lambda=\lambda_j$ and $r=\vartheta^j R$
for $j=0,1,\dots,m-1$. 

Moreover, by the second inequality in \eqref{pf11} with  $j=m-1$ and the reverse inequality of the second one in 
\eqref{induction2} when $m=m_1$, or \eqref{induction3} when $m<m_1$,
we have either
$$
\lambda_m  \ge |(D \bfu)^{\lambda_m}_{2\vartheta^{m}R}| \ge \frac{1}{2^{n+2}}  |(D \bfu)^{\lambda_m}_{\vartheta^{m}R}| \ge \frac{1}{2^{n+2} K}\lambda_m \ge \frac{1}{K}\lambda_m\,,
$$
or 
$$
\lambda_m  \ge |(D \bfu)^{\lambda_m}_{2\vartheta^{m}R}| \ge \frac{1}{2^{n+2}}  |(D \bfu)^{\lambda_m}_{\vartheta^{m}R}| \ge \frac{1}{2^{n+2} K}\lambda_m \ge \frac{1}{K_1}\lambda_m\,. 
$$
Therefore, applying Lemma~\ref{Lem:nondegenerate2} with $R=\vartheta^m r$, $K_0=\max\{M_1,K\}$ and $\delta_1=\delta_1(K)$ when $m=m_1$, or with $R=\vartheta^m r$, $K_0=\max\{M_1,K_1\}$ and $\delta_1=\delta_1(K_1)$ when $m<m_1$, we can get all the estimates except \eqref{Lem:iteration1_result5}.

Thus, we are left to prove  \eqref{Lem:iteration1_result5} for  $ j \le m-1$. Note that the case $j=m$ follows from \eqref{Lem:iteration1_result6}.   
We first observe that if $ j\le m-1$, we have \eqref{pf11} 
%with $j=m-1$ that
%$$
% \fint_{Q^{\lambda_{m}}_{\vartheta^{m} r}} \phi_{|(D \bfu)^{\lambda_{m}}_{\vartheta^{m} r}|} (|D \bfu-(D \bfu)^{\lambda_{m}}_{\vartheta^{m} r}|) \, \dz \le \phi(\lambda_{m})  
%\quad\text{and}\quad
%|(D \bfu)^{\lambda_{m}}_{2\vartheta^{m} r}| \le \lambda_{m} \,.
%$$
which, together with \eqref{eq:approx}, implies
\begin{equation}\label{Lem:iteration1_pf6}
\fint_{Q^{\lambda_j}_{\vartheta^j R}} \phi(|D \bfu-(D \bfu)^{\lambda_j}_{\vartheta^j R}|) \, \dz \le \fint_{Q^{\lambda_j}_{\vartheta^j R}} \phi(|D \bfu-(D \bfu)^{\lambda_j}_{\vartheta^j R}|+|(D \bfu)^{\lambda_j}_{\vartheta^j R}|) \, \dz \le c\phi(\lambda_j)\,,
\end{equation}
whenever $1\le j \le m$. Moreover, by the same argument, we also have \eqref{Lem:iteration1_pf6} when $j=0$  from   \eqref{Lem:iteration1_ass1} and \eqref{Lem:iteration1_ass2}. 
 %, provided that $\delta_2>0$ is sufficiently small. 
From this estimate we have that 
$$\begin{aligned}
\fint_{Q^{\lambda_j}_{\vartheta^{j} R}(z_0)} \phi (|D \bfu-\Gamma_0|) \, \dz &  \le c  \fint_{Q^{\lambda_j}_{\vartheta^{j} R}} \phi (|D \bfu- (D \bfu)^{\lambda_j}_{\vartheta^j R}|) \, \dz  + c \phi (|(D \bfu)^{\lambda_j}_{\vartheta^j r}  -\Gamma_0|)\\
& \le c \phi(\lambda_j)+ c \phi (|(D \bfu)^{\lambda_j}_{\vartheta^j R}  -\Gamma_0|) \,.
\end{aligned}$$
Moreover, by \eqref{Lem:iteration1_pf6} and \eqref{Lem:iteration1_result2},
$$\begin{aligned}
&|(D \bfu)^{\lambda_j}_{\vartheta^j R}  -\Gamma_0|   \le \sum_{k=j}^{m-1} |(D \bfu)^{\lambda_k}_{\vartheta^k R}  -(D \bfu)^{\lambda_{k+1}}_{\vartheta^{k+1} R}  | + |(D \bfu)^{\lambda_m}_{\vartheta^m R}  -\Gamma_0|\\
&\le \sum_{k=j}^{m-1}  \phi^{-1}\bigg(\frac{|Q^{\lambda_{k}}_{\vartheta^{k} R}|}{|Q^{\lambda_{k+1}}_{\vartheta^{k+1} R}|} \fint_{Q^{\lambda_{k}}_{\vartheta^{k} R}}\phi(|D\bfu-(D \bfu)^{\lambda_k}_{\vartheta^k R}|)\,\dz\bigg) +\phi^{-1} \bigg(\fint_{Q^{\lambda_{m}}_{\vartheta^{m} R}}\phi(|D\bfu-\Gamma_0|)\,\dz\bigg) \\
&\le \sum_{k=j}^{m-1}  \phi^{-1}\bigg(\vartheta^{-(n+2)}\frac{\phi(\lambda_{k+1})\lambda_{k}^2}{ \phi(\lambda_{k})\lambda_{k+1}^2}\phi(\lambda_k)\bigg) +c \lambda_m\\
& \le  \sum_{k=j}^{m-1}  \phi^{-1}\left(\vartheta^{-(n+2+2\alpha_1+q)} \phi(\lambda_{k})\right) +c \lambda_m \\
&\le c \sum_{k=j}^{m} \lambda_k \le c \lambda_j \sum_{k=j}^m \vartheta^{\alpha (k-j)} \le c \lambda_j\,.
\end{aligned}$$
Therefore, combining the preceding two estimates, we obtain \eqref{Lem:iteration1_result5}.

\textit{Step 3. (Degenerate decay)} 
Suppose $m_1=\infty$ and \eqref{induction3} holds for all $j\in \mathbb N$. Then we choose $m=\infty$ and by Lemma~\ref{Lem:degenerate} with $\lambda=\lambda_j$ and $r=\vartheta^j R$, we have the first inequality in  \eqref{Lem:iteration1_result2} and \eqref{Lem:iteration1_result4}. We next prove the remaining results, namely, the second  condition in \eqref{Lem:iteration1_result3} and  \eqref{Lem:iteration1_result5} with $\Gamma_0=\zero$.
Observe that applying Lemma~\ref{Lem:degenerate} inductively, we have that for all $j\in\mathbb N_0$, 
\begin{equation}\label{Lem:degenerate_decayj}
 \fint_{Q^{\lambda_{j}}_{\vartheta^{j} R}} \phi_{|(D \bfu)^{\lambda_{j}}_{\vartheta^{j} R}|} (|D \bfu-(D \bfu)^{\lambda_{j}}_{\vartheta^{j} R}|) \, \dz \le \phi(\lambda_{j})  
\quad\text{and}\quad
|(D \bfu)^{\lambda_{j}}_{2\vartheta^{j} R}| \le \lambda_{j} \le \vartheta^{\alpha j}\,,
\end{equation}
hence 
$$
\lim_{j\to \infty} (D\bfu)^{\lambda_j}_{\vartheta^j R} = {\bf 0}\,.
$$
Fix any $r \in(0,R)$. Since
$$
Q^{\lambda_{j}}_{\vartheta^{j} R} \subset Q^{\lambda_0}_{\tilde \sigma_1^j R}=Q_{\sigma_1^j R} \,, \qquad  j\in \mathbb{N}_0
$$ 
by \eqref{Lem:iteration1_result4}, there exists $j \in \mathbb{N}_0$ such that
$$
Q_r \subset Q^{\lambda_{j}}_{\vartheta^{j}R}  
\quad \text{and} \quad
Q_r \not\subset Q^{\lambda_{j+1}}_{\vartheta^{j+1} R} \,,
$$
which implies that
$$
\min\left\{\vartheta^{j+1}R, \frac{\vartheta^{j+1} R}{\sqrt{\phi(\lambda_{j+1})/\lambda_{j+1}^2}}  \right\}
<  r \le \min\left\{ \vartheta^{j} R, \frac{\vartheta^{j} R}{\sqrt{\phi(\lambda_j)/\lambda_{j+1}^2}}\right\} \,.
$$
By \eqref{Lem:degenerate_decayj} and \eqref{Lem:iteration1_pf6} and the inequality $\frac{2n}{n+2} < p < 2$, we have 
$$\begin{aligned}
|(D\bfu)_{r}| & \le   \fint_{ Q_{r} }|D\bfu-(D\bfu)^{\lambda_j}_{\vartheta^{j} R}| \,\dz +   |(D\bfu)^{\lambda_j}_{\vartheta^{j} R}| \\
& \le \phi^{-1}\left(\frac{|Q^{\lambda_{j}}_{\vartheta^{j} R}|}{|Q_{r}|}\fint_{ Q^{\lambda_{j}}_{\vartheta^{j} R} }\phi(|D\bfu-(D\bfu)^{\lambda_j}_{\vartheta^{j} R}|) \,\dz\right) +   |(D\bfu)^{\lambda_j}_{\vartheta^{j} R}| \\
&\le c \phi^{-1} \left(\frac{(\vartheta^{j}R)^{n+2}}{r^{n+2}} \lambda_j^2 \right) + \lambda_j\\
&\le c \phi^{-1} \left(\vartheta^{-(n+2)} \max\{1,(\phi(\lambda_{j+1})/\lambda_{j+1}^2)^{\frac{n+2}{2}}\} \lambda_j^2 \right) + \lambda_j\\
&\le c \phi^{-1} \left( \max\{\lambda_j^2,\lambda_j^2(\phi(\lambda_j)/\lambda_{j+1}^2)^{\frac{n+2}{2}}\}  \right) + \lambda_j\\
&\le c \phi^{-1} \left( \max\{\lambda_j^2,\lambda_j^{\frac{n+2}{2}p-n}\} \right) + \lambda_j  \le c \lambda_j^{\frac{1}{q}(\frac{n+2}{2}p-n)}\,.
\end{aligned}$$
Moreover, since 
$$
\frac{r}{R} \ge \vartheta^{j+1} \min \left\{ 1, (\phi(\lambda_{j+1})/\lambda_{j+1}^2)^{-1/2} \right\} \ge  \vartheta^{j+1} \lambda_{j+1}^{\frac{2-p}{2}} \ge c \lambda_j^{\frac{1}{\alpha}+\frac{2-p}{2}}
$$ 
by   \eqref{Lem:iteration1_result2}  we have 
$$
|(D\bfu)_{r}|   \le c \left(\frac{r}{R}\right)^{\frac{1}{q}(\frac{n+2}{2}p-n)/(\frac{1}{\alpha}+\frac{2-p}{2})} \longrightarrow 0 \quad \text{as }\ r\to 0 \,,
$$
hence 
$$
\lim_{r\to 0^+}(D\bfu)_{r} = {\bf 0} =: \Gamma_0\,.
$$
Finally, by \eqref{Lem:iteration1_pf6}  and the second inequality in \eqref{Lem:degenerate_decayj} we have 
$$
\fint_{Q^{\lambda_j}_{\vartheta^j R}} \phi(|D \bfu|) \, \dz  \le c \fint_{Q^{\lambda_j}_{\vartheta^j R}} \phi(|D \bfu-(D \bfu)^{\lambda_j}_{\vartheta^j R}|) \, \dz + c\phi(|(D \bfu)^{\lambda_j}_{\vartheta^j R}|)  \le c\phi(\lambda_j)\,,
$$
which proves \eqref{Lem:iteration1_result5} when $m=\infty$.
\end{proof}

Now we prove the partial H\"older continuity result for $D\bfu$.

\begin{proof}[Proof of Theorem~\ref{mainthm}] 
By the parabolic Lebesgue differentiation theorem, we may assume that $D\bfu(z)=\lim_{r\to 0^+}(D\bfu)_{Q_r(z)}$ if the limit exists.
Fix  $z_0\in \Omega_T \setminus (\Sigma_1\cup \Sigma_2)$,  where $\Sigma_1$ and $\Sigma_2$ are defined in \eqref{Sigma1} and \eqref{Sigma2}, respectively. Hence we have  
$$
\liminf_{r\to 0^+}\fint_{Q_{r}(z_0)}  \phi_{|(D\bfu)_{Q_r(z_0)}|}(|D\bfu -(D\bfu)_{Q_r(z_0)} |) \,\dz =0\,,
$$
$$
\limsup_{r\to 0^+}  |(D\bfu)_{Q_r(z_0)} | < \infty  \,.
$$
From these results and the absolute continuity of the integral, %with respect the the underlying region, 
one can find $R>0$ such that $Q_{2R}(z_0)\Subset \Omega_T$ and for every $\tilde z\in Q_R(z_0)$,
$$
\fint_{Q_{R}(\tilde z)}  \phi_{|(D\bfu)_{Q_R(\tilde z)}|}(|D\bfu -(D\bfu)_{Q_R(\tilde z)} |) \,\dz \le \epsilon_2
$$
with $\epsilon_2$ as in \eqref{epsilon2}, and
$$
|(D\bfu)_{Q_R(\tilde z)}| \le M_0\,,
$$
for some $M_0<\infty$.
 Then, in view of Lemma~\ref{Lem:iteration1}, we have that for each $\tilde z\in Q_R(z_0)$,
$$
\Gamma_{\tilde z} : = \lim_{r\to 0^+} (D\bfu)_{Q_r(\tilde z)} 
$$
exists and there exist $m_{\tilde z} \in \mathbb N_0 \cup \{\infty\}$ and positive numbers $\lambda_{\tilde z, j}$ with $j\in\{0,1,\dots,  m_{\tilde z}\}$ such that
\begin{equation}\label{mainthm_pf1}
\begin{cases}
\lambda_{\tilde z, 0}=1\,,\\
\vartheta^{\alpha_1} \lambda_{\tilde z, j-1}\le \lambda_{\tilde z, j}  \le \vartheta^{\alpha}  \lambda_{\tilde z, j-1} \quad \text{for}\ \ 1\le  j < m_{\tilde z}\,,\\
\vartheta^{\alpha_1}\lambda_{\tilde z, m_{\tilde z}-1} \le \lambda_{\tilde z, m_{\tilde z}}  \le C_1\lambda_{\tilde z, m_{\tilde z}-1} \quad \text{for}\ \  0< m_{\tilde z} <\infty \,,
\end{cases}
\end{equation}
\begin{equation*}\label{mainthm_pf2}
\frac{\lambda_{\tilde z,m_{\tilde z}}}{2K_2} \le |\Gamma_{\tilde z}| \le 2K_2 \lambda_{\tilde z,m_{\tilde z}}\,,
\end{equation*}
\begin{equation}\label{mainthm_pf3}
\fint_{Q^{\lambda_{\tilde z,j}}_{\vartheta^{j} R}(\tilde z)} \phi (|D \bfu-\Gamma_{\tilde z}|) \, \dz 
\le c \phi(\lambda_{\tilde z,j}) \qquad \text{for }\ 0\le j \le m_{\tilde z}  \ \text{ with }\ j<\infty\,,
\end{equation}
and, if $m_{\tilde z}<\infty$,
\begin{equation}\label{mainthm_pf4}
\fint_{Q^{\lambda_{\tilde z,m_{\tilde z}}}_{r}(\tilde z)} \phi (|D \bfu-\Gamma_{\tilde z}|) \, \dz 
\le c \left(\frac{r}{\vartheta^{m_{\tilde z} } R}\right)^{\alpha_2}\phi(\lambda_{\tilde z,m_{\tilde z}}) \quad \text{for all }\ 0< r \le \vartheta^{m_{\tilde z}} R \,.
\end{equation}

We shall prove that the mapping $z \mapsto \Gamma_z=D\bfu(z)$ from $Q_{R/2}(z_0)$ to $\R^{Nn}$ is H\"older continuous  with respect to the parabolic distance in $Q_{R/2}(z_0)\subset \R^{n+1}$ and the Euclidean distance in $\R^{Nn}$.
For  $\tilde z\in Q_{R/2}(z_0)$ and $r\in(0,R)$,
we first suppose
$Q_r(\tilde z) \subset  Q_{\vartheta^{m_{\tilde z}}R}^{\lambda_{\tilde z, m_{\tilde z}}}(\tilde z)$. Note that in this case $m_{\tilde z}<\infty$. Then define $\rho>0$ as
$$
\rho := \max\left\{1, \sqrt{\phi(\lambda_{m_{\tilde z}})/\lambda_{m_{\tilde z}}^2}  \right\} r\,, 
$$
so that $Q_r(\tilde z) \subset  Q_{\rho}^{\lambda_{\tilde z, m_{\tilde z}}}(\tilde z)$. Hence  by \eqref{mainthm_pf4}, the inequality $\frac{2n}{n+2}<p<2$ and \eqref{mainthm_pf1}, we have 
$$ \begin{aligned}
\fint_{Q_r(\tilde z)} \phi (|D \bfu-\Gamma_{\tilde z}|) \, \dz 
&\le \frac{|Q^{\lambda_{\tilde z,m_{\tilde z}}}_{\rho}(\tilde z)|}{|Q_r(\tilde z)|}\fint_{Q^{\lambda_{\tilde z,m_{\tilde z}}}_{\rho}(\tilde z)} \phi (|D \bfu-\Gamma_{\tilde z}|) \, \dz \\
&\le c \frac{\rho^{n+2}\lambda_{\tilde z, m_{\tilde z}}^2}{r^{n+2}\phi(\lambda_{\tilde z, m_{\tilde z}})}\left(\frac{\rho}{\vartheta^{m_{\tilde z} } R}\right)^{\alpha_2}\phi(\lambda_{\tilde z,m_{\tilde z}}) \\
&= c  \left(\frac{\rho}{\vartheta^{m_{\tilde z} } R}\right)^{\alpha_2} \lambda_{\tilde z,m_{\tilde z}}^2  \max\left\{1, (\phi(\lambda_{\tilde z,m_{\tilde z}})/\lambda_{\tilde z,m_{\tilde z}}^2 )^{\frac{n+2}{2}} \right\} \\
&\le c  \left(\frac{\rho}{\vartheta^{m_{\tilde z} } R}\right)^{\alpha_3}   \max\{\lambda_{\tilde z,m_{\tilde z}}^2, \lambda_{\tilde z,m_{\tilde z}}^{\frac{n+2}{2}p-n} \} \\
&\le c  \left(\frac{\max\left\{1, \sqrt{\phi(\lambda_{m_{\tilde z}})/\lambda_{m_{\tilde z}}^2}  \right\} r}{\lambda_{\tilde z, m_{\tilde z} }^{1/\alpha} R}\right)^{\alpha_3}   \lambda_{\tilde z,m_{\tilde z}}^{\frac{n+2}{2}p-n} \\
&\le c  \left(\frac{r}{ R}\right)^{\alpha_3}   \lambda_{\tilde z, m_{\tilde z}}^{ -(\frac{2-p}{2}+\frac{1}{\alpha})\alpha_3+ \alpha (\frac{n+2}{2}p-n)}  \le c \left(\frac{r}{R}\right)^{\alpha_3}\,,
\end{aligned}$$
where 
$$
0<\alpha_3 \le \min\left\{\alpha_2, \frac{1}{2}\frac{\alpha (\frac{n+2}{2}p-n)}{\frac{2-p}{2}+\frac{1}{\alpha}}\right\}\,.
$$
We next suppose
$Q_r(\tilde z) \not\subset  Q_{\vartheta^{m_{\tilde z}}R}^{\lambda_{\tilde z, m_{\tilde z}}}(\tilde z)$. Then there exists $0\le j < m_{\tilde z}$ such that
$$
Q_r(\tilde z) \not\subset  Q_{\vartheta^{l+1}R}^{\lambda_{\tilde z, l+1}}(\tilde z) 
\quad \text{and}\quad
 Q_r(z_0)\subset 
Q_{\vartheta^{l}R}^{\lambda_{\tilde z, l}}(\tilde z)\,,
$$
which implies 
\begin{equation}\label{mainthm_pf5}
\vartheta^{j+1}R < \max\left\{1,\sqrt{\phi(\lambda_{\tilde z, j+1})/\lambda_{\tilde z, j+1}^2}\right\}r  
\quad \text{and}\quad 
\vartheta^{j}R \ge  \max\left\{1,\sqrt{\phi(\lambda_{\tilde z, j})/\lambda_{\tilde z, j+1}^2}\right\}r \,.
\end{equation}
Then by \eqref{mainthm_pf3}, the first inequality in \eqref{mainthm_pf5}, the inequality $\frac{2n}{n+2}<p<2$ and \eqref{mainthm_pf1}, we have 
$$\begin{aligned}
\fint_{Q_r(\tilde z)} \phi (|D \bfu-\Gamma_{\tilde z}|) \, \dz  
& \le \frac{|Q^{\lambda_{\tilde z,j}}_{\vartheta^{j} R}(\tilde z)|}{|Q_r(\tilde z)|} \fint_{Q^{\lambda_{\tilde z,j}}_{\vartheta^{j} R}(\tilde z)} \phi (|D \bfu-\Gamma_{\tilde z}|) \, \dz \\
& \le c \frac{(\vartheta^j R)^{n+2}\lambda_{\tilde z, j}^2}{r^{n+2}\phi(\lambda_{\tilde z, j})} \phi(\lambda_{\tilde z,j}) \\
& \le c \lambda_{\tilde z,j}^2  \max\{1, (\phi(\lambda_{\tilde z,j+1} )/\lambda_{\tilde z,j+1} ^2)^{\frac{n+2}{2}}  \}  \le  \lambda_{\tilde z,j+1}^{\frac{n+2}{2}p-n}  \,.
\end{aligned}$$
Moreover, by the first inequality in \eqref{mainthm_pf5} again we have 
$$
\frac{r}{R} >  \frac{\vartheta^{j+1}}{\max\left\{1,\sqrt{\phi(\lambda_{\tilde z, j+1})/\lambda_{\tilde z, j+1}^2}\right\}} \ge c \lambda_{\tilde z, j+1}^{\frac{1}{\alpha}} \min\left\{ 1,  \lambda_{\tilde z, j+1}^{1-\frac{p}{2}} \right\} \ge c  \lambda_{\tilde z, j+1}^{1+\frac{1}{\alpha}-\frac{p}{2}}\,.
$$
Therefore, combining the above results, we have that 
$$
\fint_{Q_r(\tilde z)} \phi (|D \bfu-\Gamma_{\tilde z}|) \, \dz \le c \left(\frac{r}{R}\right)^{\alpha_3}
$$
for some small $\alpha_3\in(0,1)$.

Now, let $z_1,z_2\in Q_{R/2}(z_0)$ be any two points with $z_1\neq z_2$ and $r:= d_p(z_1,z_2)$, where
$$
d_p(z_1,z_2):= \max\{|x_1-x_2|,\sqrt{|t_1-t_2|}\}\,.
$$
Set $Q:= Q_{r}(z_1)\cap Q_{r}(z_2)$. Note that $ c(n)|Q_r| \le |Q| \le |Q_r|$.  Therefore,
$$\begin{aligned}
|\Gamma_{z_1}-\Gamma_{z_2}| 
&\le \fint_{Q}|D\bfu - \Gamma_{z_1}| \,\dz + \fint_{Q}|D\bfu - \Gamma_{z_1}| \,\dz\\
& \le c \fint_{Q_r(z_1)}|D\bfu - \Gamma_{z_1}| \,\dz + c \fint_{Q_r(z_2)}|D\bfu - \Gamma_{z_1}| \,\dz \\
& \le c \phi^{-1}\left(\left(\frac{r}{R}\right)^{\alpha_3}\right) \le c \left(\frac{r}{R}\right)^{\alpha_3/q}=c \left(\frac{d_p(z_1,z_2)}{R}\right)^{\alpha_3/q}\,,
\end{aligned}$$
which implies the H\"older continuity of $z \mapsto \Gamma_z=D\bfu(z)$ with respect to the parabolic metric on $Q_{R/2}(z_0)$ with H\"older exponent $\alpha_3/q$. Since  $z_0\in \Omega_T \setminus (\Sigma_1\cup \Sigma_2)$ was an arbitrary point and both $\Sigma_1$ and $\Sigma_2$ are $\mathcal{L}^{n+1}$-null sets, the proof is complete. 
\end{proof}

\section*{Acknowledgments} 

J. Ok was supported by the National Research Foundation of Korea by the Korean Government (NRF-2022R1C1C1004523) and by GNAMPA visiting program. 
G. Scilla has been supported by the Italian Ministry of Education, University and Research through the MIUR – PRIN project 2017BTM7SN “Variational methods for stationary and evolution problems with singularities and interfaces”. The research of B. Stroffolini was supported by PRIN Project 2017TEXA3H “Gradient flows, Optimal Transport and Metric Measure Structures”.

\subsection*{Data availability}
Data sharing is not applicable to this article as obviously no datasets were generated or
analyzed during the current study.

\bibliographystyle{amsplain}

\end{document}